%% file: arxiv_v4.tex
\documentclass[11pt,a4paper,oneside,reqno]{amsart} 
\input{drawGraphs}
\input{packages_and_notation}
\usepgfplotslibrary{external}
\tikzsetfigurename{tikz-figure} 
\pgfplotsset{plot coordinates/math parser=false} 

\title[Cusp Universality for Random Matrices I: Local Law and the Complex Hermitian Case]{Cusp Universality for Random Matrices I:\\ Local Law and the Complex Hermitian Case} 
\author{L\'aszl\'o Erd\H{o}s$^{\dagger}$}
\address[L. Erd\H{o}s and D. Schr\"oder]{IST Austria, Am Campus 1, A-3400 Klosterneuburg, Austria}
\email{dschroed@ist.ac.at}
\email{lerdos@ist.ac.at}
\thanks{$^\dagger$Partially supported by ERC Advanced Grant No.~338804}
\author{Torben Kr\"uger$^{\ast}$}
\address[T. Kr\"uger]{University of Bonn, Endenicher Allee 60, 53115 Bonn, Germany}
\email{torben-krueger@uni-bonn.de}
\thanks{$^\ast$Partially supported by the Hausdorff Center for Mathematics} 
\author{Dominik Schr\"oder$^{\dagger\ddagger}$}
\thanks{$^\ddagger$Partially supported by the IST Austria Excellence Scholarship}
\subjclass[2010]{60B20, 15B52}  
\keywords{Cusp universality, Pearcey kernel, Local law}
\date{\today} 
\begin{document}

\begin{abstract}
For complex Wigner-type  matrices, i.e.\ Hermitian random matrices with independent,
not necessarily identically distributed entries above the diagonal, we show that at any cusp singularity of the limiting eigenvalue distribution the local eigenvalue statistics are universal and form a Pearcey process. Since the density of states typically exhibits only square root or
cubic root cusp singularities, our work complements previous results on the bulk and edge universality and 
it thus completes the resolution of the Wigner-Dyson-Mehta universality conjecture 
for the last remaining universality type in the complex Hermitian class. Our analysis holds not only for exact cusps, but
approximate cusps as well, where an extended Pearcey process emerges.
As a main technical ingredient we prove an optimal local law at the cusp for both symmetry classes. This result is 
also the key input  in the companion paper \cite{1811.04055} where the cusp universality for real symmetric Wigner-type matrices is proven. The novel cusp fluctuation mechanism is also essential for the recent results on the spectral radius of non-Hermitian random matrices~\cite{1907.13631}, and the non-Hermitian edge universality~\cite{1908.00969}.
\end{abstract} 
\maketitle

\section{Introduction}
The celebrated Wigner-Dyson-Mehta (WDM) conjecture asserts that local eigenvalue statistics of large random
matrices are universal: they only depend on the symmetry type of the matrix and are otherwise independent
of the details of the distribution  of the matrix ensemble. This remarkable spectral robustness was first observed by Wigner
in the bulk of the spectrum. The correlation functions are determinantal and  they were computed in terms  the \emph{sine kernel} 
 via explicit Gaussian calculations by Dyson, Gaudin and Mehta \cite{MR0220494}.
Wigner's vision continues to hold at the spectral edges, where the correct statistics was identified by Tracy and Widom for both 
symmetry types in terms of the \emph{Airy kernel} \cite{MR1257246,MR1385083}.  
These universality  results have been originally formulated and proven \cite{MR3253704,MR2662426,MR2810797,MR1727234,MR2784665,MR2669449} for traditional \emph{Wigner matrices}, i.e.\ Hermitian random matrices with independent, identically distributed (i.i.d.) entries
and their  diagonal \cite{MR3405746,MR3502606} and non-diagonal \cite{MR3704770} deformations. 
 More recently they have been extended  to 
\emph{Wigner-type ensembles}, where the identical distribution is not required, and even to a large class of matrices with general correlated entries \cite{MR3719056,MR3916109,1804.07744}. 
In different directions of generalization,  sparse matrices
\cite{MR2964770,MR3800840,1509.03368,MR3429490}, 
 adjacency matrices of regular graphs \cite{MR3729611} and band matrices \cite{MR3695802,1807.01559,MR2726110} have also been considered. 
  In  parallel developments  bulk and edge universal statistics have been proven for invariant $\beta$-ensembles  \cite{MR3253704,MR3192527,MR2306224,MR1702716,MR2012268,MR2375744,MR2525225,MR3351052,28950489,MR3390602,MR2534097,MR3433632} and even for their discrete analogues \cite{1705.05527,MR2283089,MR1758751,MR1826414} but often with very different methods.  
 
A precondition for  the Tracy-Widom distribution 
 in all these generalizations of Wigner's original ensemble is that 
the density of states  vanishes as a square root near the spectral edges.  
The recent  classification of the singularities of the solution to the underlying Dyson equation 
indeed revealed that at the edges  only square root singularities appear \cite{MR3684307,1804.07752}. 
The density of states may also form a cusp-like singularity in the interior of the asymptotic  spectrum, i.e.~single points of vanishing density with a cubic root growth behaviour on either side.
 Under very general conditions,  no other type of singularity may occur.
At the cusp  a new local eigenvalue process emerges: the correlation functions are
still determinantal but the \emph{Pearcey kernel} replaces the  sine- or
the Airy kernel. %

The Pearcey process was first established by Br\'ezin and Hikami for the eigenvalues close to a cusp singularity of a deformed complex Gaussian Wigner (GUE) matrix. They considered the model of a GUE matrix plus a deterministic matrix (``external source'') having eigenvalues $\pm 1$ with equal multiplicity \cite{MR1662382,MR1618958}. The name \emph{Pearcey kernel} and the corresponding \emph{Pearcey process}
 have been coined by \cite{MR2207649} in reference to related functions introduced by Pearcey in the context of electromagnetic fields \cite{MR0020857}. Similarly to the universal sine and Airy processes, it has later been observed that also the Pearcey process universality extends beyond the realm of random matrices. %
Pearcey statistics have been established for non-intersecting Brownian bridges \cite{MR2642890}
and  in skew plane partitions \cite{MR2276355},  always  at  criticality.  We remark, however, that critical cusp-like singularity does not always induce a Pearcey kernel, see e.g.~\cite{MR3546394}. 

In random matrix theory there are still only a handful of rather specific models for which the emergence of the Pearcey process has been proven. This has been achieved for deformed GUE matrices  %
\cite{MR3500269,MR2337504,MR2820069} and
for Gaussian sample covariance matrices \cite{MR3502605,MR3485343,MR3440796} by a contour integration  method based upon the Br\'ezin-Hikami formula.  Beyond linear deformations,
the Riemann-Hilbert method has been used for proving Pearcey statistics for a  certain  \emph{two-matrix model} with 
  a special quartic potential with appropriately tuned coefficients  \cite{MR3384456}. 
  All these previous results concern only specific ensembles with a
  matrix integral representation. In particular, Wigner-type matrices are out of the scope
  of this approach. 
  
The main result of the current paper is  the proof of  the Pearcey universality at the cusps for  complex Hermitian Wigner-type matrices
under very general conditions.  Since the classification theorem  excludes any other singularity, this is the 
third and last universal statistics that emerges  from  natural generalizations of Wigner's ensemble.

This third  universality class has received somewhat less attention than the other two, presumably because
cusps are not present in the classical Wigner ensemble. We also note that
   the most common invariant $\beta$-ensembles do not exhibit the Pearcey statistics as their densities do not feature cubic root cusps but are instead $1/2$-H\"older continuous for somewhat regular potentials \cite{MR1657691}. 
  The density vanishes either as  $2k$-th  or   $(2k +\frac{1}{2})$-th power with their own local statistics (see  \cite{MR3833603} also for the persistence of these statistics under small additive
    GUE perturbations before the critical time).
Cusp singularities, hence Pearcey statistics, however,
naturally arise  within any one-parameter family of Wigner-type ensembles
whenever two spectral bands merge as the parameter varies. The classification theorem implies that cusp formation is the 
only possible way for bands to merge, so in that sense  Pearcey universality is ubiquitous as well.

The bulk and edge universality is characterized by the symmetry type alone: up to a natural shift and  rescaling there is only one bulk and one edge statistic.
 In contrast, the cusp universality has a much richer structure: it is naturally embedded in a one-parameter family of universal statistics within each symmetry class.
 In the complex Hermitian case these are given by the one-parameter family of (extended) Pearcey kernels, see \eqref{Pearcey kernel} later.
 Thinking in terms of  fine-tuning a single parameter in the space of Wigner-type  ensembles, the density of states already 
 exhibits a universal local shape right before and right after the cusp formation; it features
 a tiny gap or a tiny nonzero local minimum, respectively \cite{1506.05095,1804.07752}.
   When the local lengthscale $\ell$ of these  \emph{almost cusp} shapes is comparable with the local eigenvalue spacing $\delta$, then 
   the general Pearcey statistics is expected to emerge whose  parameter is determined by the ratio $\ell/\delta$.
Thus the full Pearcey universality typically appears in a \emph{double scaling limit}.

 Our proof follows the \emph{three step strategy} that is the backbone of the recent approach to the WDM universality,
 see \cite{MR3699468} for a pedagogical expos\'e and for detailed history of the method. The first step in this strategy is a 
 \emph{local law} that identifies, with very high probability, the empirical eigenvalue distribution on a scale slightly
 above the typical eigenvalue spacing. The second step is to prove universality for ensembles with a tiny Gaussian
 component. Finally, in the third step this Gaussian component is removed by perturbation theory. The local law is used
 for precise  apriori bounds  in the second and third steps.  
 
 The main novelty of the current paper is the proof of the local law at optimal scale near the cusp. 
To put the precision in 
 proper context, we normalize the  $N\times N$ real symmetric or complex Hermitian Wigner-type matrix $H$ 
 to have norm of order one.    As customary, the local law is formulated
 in terms of the Green function $G(z)\defeq(H-z)^{-1}$ with spectral parameter $z$ in the upper half plane.
 The local law then asserts that $G(z)$ becomes deterministic in the large $N$ limit as long as  $\eta \defeq\Im z$
 is much larger than the local eigenvalue spacing around $\Re z$. The deterministic approximant $M(z)$ 
 can be computed as the unique solution of the corresponding Dyson equation  (see \eqref{Dyson equation} and \eqref{MDE}
 later).
 Near the cusp the typical eigenvalue spacing is 
 of order $N^{-3/4}$; compare this with the 
 $N^{-1}$ spacing in the bulk and $N^{-2/3}$ spacing near the edges.  We remark that
 a  local law at the cusp  on the non-optimal scale $N^{-3/5}$ has already been proven in \cite{MR3719056}. %
  In the current paper we improve this result  to the optimal scale $N^{-3/4}$
 and this is essential for our universality proof at the cusp.
 
 The main  ingredient behind this improvement is an optimal
estimate of the error term $D$  (see \eqref{error matrix D} later) 
  in the  approximate Dyson equation
that $G(z)$ satisfies. The difference $M-G$ is then roughly estimated by ${\mathcal B}^{-1} (MD)$, where ${\mathcal B}$ is the linear stability operator of the Dyson equation.
   Previous estimates on $D$ (in averaged sense) were of order $\rho/N\eta$, where $\rho$ is
the local density; roughly speaking $\rho\sim 1$ in the bulk, $\rho\sim N^{-1/3}$ at the edge and $\rho\sim N^{-1/4}$
near the cusp.  While this estimate cannot be improved in general, our main observation is that, to leading order,  we need only
the projection of $MD$ in the single unstable direction of ${\mathcal B}$. We found that this projection 
carries an extra hidden cancellation
due to  a special local symmetry at the  cusp and thus  the  estimate on $D$  effectively  improves to
 $\rho^2/N\eta$. Customary  power counting is not sufficient, we need to compute this error term explicitly at least to leading order.
We call this subtle mechanism \emph{cusp fluctuation averaging} since it combines the well established
fluctuation averaging procedure with the additional cancellation at the cusp.
 Similar estimates extend to  the vicinity of the exact cusps. We identify
a key quantity, denoted by $\sigma(z)$ (in \eqref{definition of sigma} later), that measures the distance from the cusp
in a canonical way:
$\sigma(z)=0$ characterizes an exact cusp, while $\abs{\sigma(z)}\ll 1$ indicates that $z$ is  near an almost cusp.
Our final estimate on $D$ is of order $(\rho+\abs{\sigma})\rho/N\eta$.
Since the error term $D$ is random and we need to control it in high moment sense, we need to lift this idea
to a high moment calculation, meticulously extracting the improvement from every single term. This is 
performed in the  technically most involved Section~\ref{sec:Cusp fluctuation averaging} where we use a
Feynman diagrammatic formalism to bookkeep the contributions of all terms.  Originally we have developed
this language in \cite{MR3941370} to handle random matrices with slow correlation decay, based on the revival of the cumulant expansion technique in~\cite{MR3678478} after~\cite{MR1411619}. In the current paper
we incorporate the cusp into this analysis.  We identify a finite set of  Feynman subdiagrams, called \emph{$\sigma$-cells} (Definition \ref{sigma cell def})
with value $\sigma$ that embody  the cancellation effect  at the cusp.   To  exploit the full strength of the
  cusp fluctuation averaging mechanism,  we need to trace the fate of the $\sigma$-cells along the high moment 
expansion. The key point is that $\sigma$-cells are local objects in
the Feynman graphs thus  their cancellation effects act simultaneously and the corresponding gains  are multiplicative.
 
Formulated in the jargon of diagrammatic field theory, extracting the deterministic  Dyson equation for $M$ from the resolvent equation  $(H-z)G(z)=1$
corresponds to a consistent self-energy renormalization of $G$. One way or another, such procedure is behind every 
proof of the optimal local law with high  probability.  Our $\sigma$-cells conceptually correspond to a next order resummation
of certain Feynman diagrams carrying a special cancellation.

We remark that  we prove the optimal local law only for Wigner-type matrices and not yet for 
general correlated matrices unlike in \cite{MR3941370,1804.07744}. 
In fact we use the simpler setup only for the estimate on $D$  (Theorem \ref{thm pfD bound moments})
the rest of the proof  is already  formulated for the general case.
This simpler setup allows us to present the cusp fluctuation averaging mechanism with the least
amount of technicalities.
The extension %
 to the correlated case is based on the same mechanism but
it requires  considerably more involved  diagrammatic
manipulations which is better to develop in a separate work to contain the length of this paper.

Our cusp fluctuation averaging mechanism has further applications. It is  used  in~\cite{1907.13631} 
 to prove an optimal cusp local law for the Hermitization of non-Hermitian random matrices 
 with a variance profile, demonstrating that   the 
 technique is also applicable in settings where the flatness assumption is violated. The cusp of the Hermitization corresponds 
 to the edge of the non-Hermitian model via Girko's formula, thus
   the optimal cusp local law leads to
  an optimal bound on the spectral radius~\cite{1907.13631} and ultimately also to edge universality~\cite{1908.00969}
  for non-Hermitian random matrices.
 
Armed with the optimal local law we then perform the other two steps of the three step analysis. The third step, relying
on the \emph{Green function comparison theorem}, is fairly standard and previous proofs used in the bulk and at
the edge need only minor adjustments.  The second step, extracting universality from an ensemble with a tiny Gaussian component
can be done in two ways: (i) Br\'ezin-Hikami formula with contour integration or (ii) Dyson Brownian Motion (DBM).  Both methods require the local law as an input.
In the current work we follow (i) mainly because this approach directly yields the Pearcey kernel, at least for the complex Hermitian 
symmetry class. In the companion work \cite{1811.04055} we perform the DBM analysis adapting methods of
\cite{1712.03881,MR3729630,MR3687212} to the cusp. The main novelty in the current work and in \cite{1811.04055} is the rigidity at the cusp on the optimal scale provided below. Once this key input is given, the proof of the edge universality from \cite{MR3687212} is modified in \cite{1811.04055} to the cusp setting, proving universality for the real symmetric case as well. 
We remark, however, that, to our best knowledge,  the analogue of the Pearcey kernel for the real symmetric case  has not yet been
explicitly identified.

We now explain some novelty in  the contour integration method. We first note that a similar approach was initiated in the fundamental work of Johansson on the bulk universality for Wigner matrices with a large Gaussian component in \cite{MR1810949}. This method was generalised later 
to Wigner matrices with a  small Gaussian component in \cite{MR2662426} 
as well as it inspired the proof of bulk universality via the moment matching idea \cite{MR2784665} 
 once the necessary local law became available.
The double scaling regime has also been studied, where the density is very small but the Gaussian component compensates for it \cite{MR3985252}. 
More recently, the same approach was extended to the cusp  for deformed GUE matrices \cite[Theorem 1.3]{MR3500269} and 
for sample covariance matrices
but only for large Gaussian component \cite{MR3502605,MR3485343,MR3440796}. For our cusp universality, we need to perform 
a similar analysis but with a small Gaussian component.
We represent our matrix $H$ as $\widehat H + \sqrt{t} U$, 
where $U$ is GUE and $\widehat H$ is an independent Wigner-type matrix. The contour integration analysis  (Section \ref{sec: free conv})
requires a Gaussian component of size at least $t\gg N^{-1/2}$.

The input  of the  analysis in Section \ref{sec: free conv}  for the correlation kernel of $H$
 is a very precise description of the eigenvalues of $\widehat H$ just above $N^{-3/4}$, 
the scale of the typical spacing between eigenvalues --- this information is provided by our optimal  local law.
 While in the bulk  and in the regime of the regular edge finding an appropriate $\wh H$ is a relatively simple matter, 
  in the vicinity of a cusp  point the issue is very delicate. 
The main reason is that the cusp, unlike the bulk or the regular edge, is unstable under small perturbations; in 
fact it typically disappears and turns into a small positive local minimum if a small GUE component is added. 
Conversely, a cusp emerges if a small GUE component is added to an ensemble 
that has a   density with a small gap. In particular,
even if the density function $\rho(\tau)$ of $H$
exhibits an exact cusp, the density $\wh\rho(\tau)$ of $\wh H$ will have a small gap: in fact $\rho$ is given by the evolution 
of the semicircular flow up to time $t$ with initial data $\wh\rho$.
Unlike in the bulk and edge cases,  here one cannot match the density of $H$ and $\wh H$
by a simple shift and rescaling.
Curiously, the contour integral analysis for the local  statistics of $H$ at the cusp relies 
on an optimal local law of $\wh H$ with a small gap far away from the cusp.

Thus we need an additional ingredient: the precise analysis of the semicircular flow $\rho_s\defeq \wh \rho \boxplus \rho_{\mathrm{sc}}^{(s)}$ 
near the cusp up to a relatively long times $s\lesssim N^{-1/2+\epsilon}$; note that $\rho_t=\rho$ is the original density with the cusp. 
Here $\rho_{\mathrm{sc}}^{(s)}$ is the semicircular density with variance $s$ and $\boxplus$ indicates the free convolution.
In Sections \ref{sec: contour integral}--\ref{sec: contour deformation}   we will see that  the  edges of the support of the density $\rho_s$ typically move linearly in the time $s$
while the gap closes at a much slower rate. Already $s\gg N^{-3/4}$
is beyond the simple perturbative regime of the cusp whose natural lengthscale is $N^{-3/4}$. 
Thus we need a very careful tuning of the parameters: the analysis of a cusp for $H$ requires
constructing a matrix $\wh H$ that is far from having a cusp but that after a relatively long time $t=N^{-1/2+\epsilon}$ will develop a cusp exactly at the right location.
In the estimates we  heavily rely on various properties of the solution to the Dyson equation established in 
the recent paper \cite{1804.07752}.
 These results go well beyond the precision of the previous work \cite{1506.05095} and they apply
 to a very general class of Dyson equations, including a non-commutative von-Neumann algebraic setup.\\

\noindent\textbf{Notations.} We now introduce some custom notations we use throughout the paper. For non-negative functions $f(A,B)$, $g(A,B)$ we use the notation $f \le_A g$ if there exist constants $C(A)$ such that $f(A,B)\le C(A) g(A,B)$ for all $A,B$. Similarly, we write $f\sim_A g$ if $f\le_A g$ and $g\le_A f$. We do not indicate the dependence of constants on basic parameters that will be called model parameters later. If the implied constants are universal, we instead write $f\lesssim g$ and $f\sim g$. Similarly we write $f \ll g$ if $f\le c g$ for some tiny absolute constant $c>0$. 

We denote vectors by bold-faced lower case Roman letters $\vx,\vy\in\C^N$, and matrices by upper case Roman letters $A,B\in\C^{N\times N}$ with entries \(A=(a_{ij})_{i,j=1}^N\). The standard scalar product and Euclidean norm on $\C^N$ will be denoted by $\braket{\vx,\vy}\defeq N^{-1}\sum_{i\in[N]}\overline{x_i}y_i$ and $\norm{\vx}$, while we also write $\braket{A,B}\defeq N^{-1}\Tr A^\ast B$ for the scalar product of matrices, and $\braket{A}\defeq N^{-1}\Tr A$, $\braket{\vx}\defeq N^{-1}\sum_{a\in[N]}x_a$. We write $\diag R$, $\diag\vr$ for the diagonal vector of a matrix $R$ and the diagonal matrix obtained from a vector $\vr$, and $S\odot R$ for the entrywise (Hadamard) product of matrices $R,S$. The usual operator norm induced by the vector norm $\norm{\cdot}$ will be denoted by $\norm{A}$, while the Hilbert-Schmidt (or Frobenius) norm will be denoted by $\norm{A}_\text{hs}\defeq \sqrt{\braket{A,A}}$. For integers $n$ we define $[n]\defeq\{1,\dots,n\}$. \\

\noindent\textbf{Acknowledgement.}  The authors are very grateful to Johannes  Alt for numerous discussions
on the Dyson equation and for his 
invaluable help in adjusting  \cite{1804.07752} to the needs of the present work.

\section{Main results}

\subsection{The Dyson equation}

Let $W=W^* \in \C^{N \times N}$ be a self-adjoint random matrix and $A=\diag(\bm{a})$ be a deterministic diagonal matrix with entries $\bm{a}=(a_i)_{i=1}^N \in \R^N$. We say that $W$ is of \emph{Wigner-type} \cite{MR3719056} if its entries $w_{ij}$ for $i \le j$ are centred, $\E  w_{ij} =0$, independent random variables. We define the \emph{variance matrix} or \emph{self-energy matrix} $S=(s_{ij})_{i,j=1}^N$ by
\begin{equation}\label{S def}
s_{ij}\defeq  \E \abs{w_{ij}}^2.
\end{equation}
This matrix is symmetric with non-negative entries. In \cite{MR3719056} it was shown that as $N$ tends to infinity, the resolvent $G(z)\defeq (H-z)^{-1}$ of the \emph{deformed Wigner-type matrix} $H=A+W$ entrywise approaches a diagonal matrix 
\[
M(z)\defeq \diag(\vm(z)).
\]
The entries $\vm=(m_1 \dots , m_N)\colon \HC \to \HC^N$  of $M$ have positive imaginary parts and solve the \emph{Dyson equation}
\begin{equation} \label{Dyson equation} -\frac{1}{m_i(z)}= z-a_i +\sum_{j=1}^Ns_{ij}m_j(z),\qquad z \in \HC\defeq\Set{z\in\C|\Im z>0}, \quad i\in[N]. \end{equation}
We call $M$ or $\vm$ the \emph{self-consistent Green's function}.
The normalised trace of $M$ is the Stieltjes transform of a unique probability measure on $\R$ that approximates the empirical eigenvalue distribution of $A+W$ increasingly well as $N \to \infty$, motivating the following definition.
\begin{definition}[Self-consistent density of states] \label{def scdos}The unique probability measure $\rho$ on $\R$, defined through
\begin{equation*} \braket{M(z)}=\frac{1}{N}\Tr M(z)= \int\frac{\rho(\dd \tau)}{\tau-z},\qquad z \in \HC, \end{equation*}
is called the self-consistent density of states (scDOS). Accordingly, its support $\supp\rho$ is called self-consistent spectrum.
\end{definition}

\subsection{Cusp universality}
We make the following assumptions:
\begin{assumption}[Bounded moments]\label{bdd moments}
The entries of the Wigner-type matrix $\sqrt{N}W$ have bounded moments and the expectation $A$ is bounded, i.e.\ there are positive $C_k$ such that 
\[
\abs{a_i}\le C_0, \qquad 
\E\abs{w_{ij}}^k \le C_kN^{-k/2} , \qquad k \in \N.
\]
\end{assumption}
\begin{assumption}[Fullness]\label{fullness}
If the matrix $W = W^* \in \C^{N \times N}$ belongs to the complex hermitian symmetry class, then we assume 
\begin{equation}\begin{split} \label{fullness cplx} \mtwo{\E(\Re w_{ij})^2& \E(\Re w_{ij})(\Im w_{ij})}{\E(\Re w_{ij})(\Im w_{ij}) & \E(\Im w_{ij})^2}\ge \frac{c}{N} \1_{2 \times 2}, \end{split}\end{equation}
as quadratic forms, for some positive constant $c>0$. If $W = W^T \in \R^{N \times N}$ belongs to the real symmetric symmetry class, then we assume $\E w_{ij}^2 \ge \frac{c}{N}$. 
\end{assumption}
\begin{assumption}[Bounded self-consistent Green's function] \label{bdd m}
In a neighbourhood of some fixed spectral parameter $\tau \in \R$ the self-consistent Green's function is bounded, i.e.\ for positive $C,\kappa$ we have 
\[
\abs{m_{i}(z)}\le C, \qquad z \in \tau+(-\kappa,\kappa)+ \ii \R^+.
\]
\end{assumption}

We call the constants appearing in Assumptions \ref{bdd moments}-\ref{bdd m} \emph{model parameters}. All generic constants $C$ in this paper may implicitly depend on these model parameters. Dependence on further parameters however will be indicated.

\begin{remark}
The boundedness of $\vm$ in Assumption \ref{bdd m} can be ensured by assuming some regularity of the variance matrix $S$. For more details we refer to \cite[Chapter~6]{1506.05095}.
\end{remark}

From the extensive analysis in \cite{1804.07752} we know that the self-consistent density $\rho$ is described by explicit 
\emph{shape functions} in the vicinity of 
local minima with small value of 
 $\rho$ and around small gaps in the support of $\rho$.
The density in such \emph{almost cusp regimes} is given by 
precisely one of the following three asymptotics:
\begin{subequations}\label{gamma def eqs}
\begin{enumerate}[(i)]
\item \emph{Exact cusp}. There is a cusp point $\cu\in\R$ in the sense that $\rho(\cu)=0$ and $\rho(\cu\pm\delta)>0$ for $0\ne\delta\ll1$. In this case the self-consistent density is locally around $\cu$ given by 
\begin{equation}\label{gamma def}
\rho(\cu\pm x) = \frac{\sqrt 3\gamma^{4/3}}{2\pi} x^{1/3} \Big[1+\landauO{x^{1/3}}\Big],\qquad x\ge 0
\end{equation}
for some $\gamma>0$.
\item \emph{Small gap.} There is a maximal interval $[\ed_-,\ed_+]$ of size $0<\Delta \defeq  \ed_+-\ed_-\ll1$ such that $\rho\rvert_{[\ed_-,\ed_+]}\equiv 0$. In this case the density around $\ed_\pm$ is, for some $\gamma>0$, locally given by
\begin{equation}\label{gamma def edge}
\rho(\ed_\pm\pm x)=\frac{\sqrt{3}(2\gamma)^{4/3}\Delta^{1/3}}{2\pi}\Psi_{\mathrm{edge}}(x/\Delta)\left[1+\landauO{\Delta^{1/3}\Psi_{\mathrm{edge}}(x/\Delta)}\right],\qquad x\ge 0
\end{equation}
where the shape function around the  edge is given by
\begin{equation}\label{Psi edge}
\Psi_{\mathrm{edge}}(\lambda)\defeq \frac{\sqrt{\lambda(1+\lambda)}}{(1+2\lambda+2\sqrt{\lambda(1+\lambda)})^{2/3}+(1+2\lambda-2\sqrt{\lambda(1+\lambda)})^{2/3}+1},\quad \lambda\ge0.
\end{equation}
\item \emph{Non-zero local minimum.} There is a local minimum at $\mi\in\R$ of $\rho$ such that $0<\rho(\mi)\ll1$. In this case there exists some $\gamma>0$ such that
\begin{equation}\label{gamma def min}
\rho(\mi + x) = \rho(\mi) + \rho(\mi) \Psi_{\mathrm{min}}\left(\frac{3\sqrt 3 \gamma^4 x}{2(\pi\rho(\mi))^3 }\right) \left[1+\landauO{\rho(\mi)^{1/2}+ \frac{\abs{x}}{\rho(\mi)^3}}\right],\qquad x\in\R
\end{equation}
where the shape function around the local minimum is given by
\begin{equation}\label{Psi min}
\Psi_{\mathrm{min}}(\lambda) \defeq \frac{\sqrt{1+\lambda^2}}{(\sqrt{1+\lambda^2}+\lambda)^{2/3}+(\sqrt{1+\lambda^2}-\lambda)^{2/3}-1}-1,\qquad \lambda\in\R.
\end{equation}
\end{enumerate}
\end{subequations}
We note that the parameter $\gamma$ in \eqref{gamma def} is chosen in a way which is convenient for the universality statement. We also note that the choices for $\gamma$ in \eqref{gamma def edge}--\eqref{gamma def min} are consistent with \eqref{gamma def} in the sense that in the regimes $\Delta\ll x\ll 1$ and $\rho(\mi)^3\ll \abs{x}\ll 1$ the respective formulae asymptotically agree. Depending on the three cases (i)--(iii), we define the \emph{almost cusp point} $\bu$ as the cusp $\cu$ in case (i), the midpoint $(\ed_-+\ed_+)/2$ in case (ii), and the minimum $\mi$ in case (iii). When the local length scale of the almost cusp shape starts to match the eigenvalue spacing, i.e. if $\Delta \lesssim N^{-3/4}$ or $\rho(\mi)\lesssim N^{-1/4}$, then we call the local shape a \emph{physical cusp}. This terminology reflects the fact that the shape becomes indistinguishable from the exact cusp with $\rho(\cu)=0$ when resolved with a precision above the eigenvalue spacing. In this case we call $\bu$ a \emph{physical cusp point}.

The extended Pearcey kernel with a real parameter $\alpha$ (often denoted by $\tau$ in the literature) is given by 
\begin{equation}\label{Pearcey kernel}
K_\alpha(x,y) = \frac{1}{(2\pi\ii)^2} \int_\Xi \diff z \int_\Phi \diff w \frac{\exp(-w^4/4 + \alpha w^2/2-yw + z^4/4-\alpha z^2/2 + xz)}{w-z},
\end{equation}
where $\Xi$ is a contour consisting of rays from $\pm\infty e^{\ii\pi/4}$ to $0$ and rays from $0$ to $\pm \infty e^{-\ii\pi/4}$, and $\Phi$ is the ray from $-\ii\infty$ to $\ii\infty$. The simple Pearcey kernel with parameter $\alpha=0$ has been first observed in the context of random matrix theory by \cite{MR1662382,MR1618958}. We note that \eqref{Pearcey kernel} is a special case of a more general extended Pearcey kernel defined in \cite[Eq.~(1.1)]{MR2207649}. 

It is natural to express universality in terms of a rescaled $k$-point function $p_k^{(N)}$ which we define implicitly by 
\[ \binom{N}{k}^{-1} \sum_{\{i_1,\dots,i_k\}\subset[N]} \E f(\lambda_{i_1},\dots,\lambda_{i_k}) = \int_{\R^k} f(x_1,\dots,x_k)p_k^{(N)}(x_1,\dots,x_k)\diff x_1\dots\diff x_k \]
for test functions $f$, where the summation is over all subsets of $k$ distinct integers from $[N]$. 
\begin{theorem}\label{thr:cusp universality}
Let $H$ be a complex Hermitian Wigner matrix satisfying Assumptions \ref{bdd moments}--\ref{bdd m}. Assume that the self-consistent density $\rho$ within $[\tau-\kappa,\tau+\kappa]$ from Assumption \ref{bdd m} has a physical cusp, i.e.~that $\rho$ is locally given by \eqref{gamma def eqs} for some $\gamma>0$ and $\rho$ either (i) has a cusp point $\cu$, or (ii) a small gap $[\ed_-,\ed_+]$ of size $\Delta\defeq \ed_+-\ed_-\lesssim N^{-3/4}$, or (iii) a local minimum at $\mi$ of size $\rho(\mi)\lesssim N^{-1/4}$. Then it follows that for any smooth compactly supported test function $F\colon\R^k\to\R$ it holds that
\[ \abs{\int_{\R^k} F(\bm x)\left[ \frac{N^{k/4}}{\gamma^k} p_k^{(N)}\left( \bu + \frac{\bm x}{\gamma N^{3/4}}\right)- \det (K_\alpha(x_i,x_j))_{i,j=1}^k\right] \diff \bm x} = \landauO{N^{-c(k)}},\]
where 
\begin{equation}\bu \defeq \begin{cases}
\cu & \text{in case (i)},\\
(\ed_++\ed_-)/2 & \text{in case (ii)},\\
\mi & \text{in case (iii)},
\end{cases}\qquad 
\alpha \defeq \begin{cases}
0 & \text{in case (i)}\\
3 \left(\gamma\Delta/4\right)^{2/3} N^{1/2} & \text{in case (ii)},\\
-\left(\pi\rho(\mi)/\gamma\right)^2 N^{1/2} & \text{in case (iii)},
\end{cases}\label{eq pearcey param choice}
\end{equation}
\(\bm x=(x_1,\dots,x_k)\), \(\diff \bm x=\diff x_1\dots\diff x_k\), and $c(k)>0$ is a small constant only depending on $k$.
\end{theorem} 

\subsection{Local law}
We emphasise that 
the proof of Theorem~\ref{thr:cusp universality} requires a very precise a priori control on the fluctuation of the eigenvalues even at singular points of the scDOS. This control is expressed in the form of a \emph{local law} with an optimal convergence rate down to the typical eigenvalue spacing.
We now define the scale on which the eigenvalues are predicted to fluctuate around the spectral parameter $\tau$. 
\begin{definition}[Fluctuation scale]
\label{def:sc fluctuation scale}
We define the {self-consistent fluctuation scale} $\eta_{\mathrm{f}}=\eta_{\mathrm{f}}(\tau)$ through
\[ \int_{-\eta_{\mathrm{f}}}^{  \eta_{\mathrm{f}}} \rho(\tau+\omega) \dd \omega= \frac{1}{N}, \]
if $\tau \in \supp \rho$. If $\tau \not \in \supp \rho$, then $\eta_{\mathrm{f}}$ is defined as the fluctuation scale at a nearby edge. More precisely,  let $I$ be the largest (open) interval with $\tau \in I \subseteq \R \setminus \supp \rho$ and set $\Delta \defeq \min\{\abs{I},1\}$. Then we define  
\begin{equation}\begin{split} \label{eta* in gap} \eta_{\mathrm{f}}\defeq  
\begin{cases}
\Delta^{1/9}/N^{2/3}, &  \Delta > 1/N^{3/4},
\\
1/N^{3/4}, &  \Delta \le 1/N^{3/4}.
\end{cases} \end{split}\end{equation}
\end{definition}

We will see later (cf.~\eqref{eta* at edge}) that \eqref{eta* in gap} is the fluctuation of the edge eigenvalue adjacent to a spectral gap of length $\Delta$ as predicted by the local behaviour of the scDOS.
The control on the fluctuation of eigenvalues is expressed in terms of the following local law.

\begin{theorem}[Local law]
\label{thr:Local law} Let $H$ be a deformed Wigner-type matrix of the real symmetric or complex Hermitian symmetry class. 
 Fix any $\tau \in \R$. Assuming \ref{bdd moments}--\ref{bdd m} 
for any $\epsilon,\zeta>0$ and $\nu \in \N$
the local law holds uniformly for all $z=\tau + \ii \eta$ with $\dist(z,\supp \rho) \in [N^\zeta\eta_{\mathrm{f}}(\tau),N^{100}]$ in the form 
\begin{subequations}
\label{local laws}
\begin{equation} \label{local law inside spectrum} \P\Bigg[\abs{\braket{\vu, (G(z)-M(z))\vv}}\ge N^\epsilon \sqrt{\frac{\rho(z)}{N \eta}} \norm{\vu}\norm{\vv}\Bigg]\le \frac{C}{N^\nu} , \end{equation}
for any $\vu,\vv \in \C^{N}$ and
\begin{equation} \label{average local law inside spectrum} \P\Bigg[ \abs{\braket{B (G(z)-M(z)}}\ge  \frac{N^\epsilon\norm{{B}}}{N \dist(z,\supp \rho)}\Bigg]\le \frac{C}{N^\nu} , \end{equation}
for any $B \in \C^{N \times N}$. Here $\rho(z)\defeq\braket{\Im M(z)}/\pi$ denotes the harmonic extension of the scDOS to the complex upper half plane. 
\end{subequations}
The constants $C>0$ in \eqref{local laws} only depends on $\epsilon,\zeta,\nu$ and the model parameters.  
\end{theorem}
We remark that later we will prove the local law also in a form which is uniform in  $\tau\in[-N^{100},N^{100}]$ and $\eta\in [ N^{-1+\zeta}, N^{100}]$, albeit with a more complicated error term, see Proposition \ref{local law uniform}. The local law Theorem~\ref{thr:Local law} implies a large deviation result for the fluctuation of eigenvalues on the optimal scale uniformly for all singularity types. 
\begin{corollary}[Uniform rigidity] 
\label{crl:Uniform rigidity}
Let $H$ be a deformed Wigner-type matrix of the real symmetric or complex Hermitian symmetry class satisfying Assumptions \ref{bdd moments}-\ref{bdd m} for $\tau \in \interior (\supp \rho)$. Then 
\[
\P\big[ \abs{\lambda_{k(\tau)}-\tau} \ge N^\epsilon \eta_{\mathrm{f}}(\tau) \big] \le \frac{C}{N^\nu}
\]
for any $\epsilon >0$ and $\nu \in \N$ and some $C=C(\epsilon,\nu)$, where we defined the (self-consistent) eigenvalue index $k(\tau)\defeq  \lceil N\rho((-\infty, \tau))\rceil$, and where $\lceil x\rceil=\min\set{k\in\Z|k\ge x}$.
\end{corollary}
In particular, the fluctuation of the eigenvalue whose expected position is closest to the cusp location does not exceed $N^{-3/4+\epsilon}$ for any $\epsilon>0$ with very high probability. The following corollary specialises Corollary~\ref{crl:Uniform rigidity} to the neighbourhood of a cusp.
\begin{corollary}[Cusp rigidity] 
\label{crl:Cusp rigidity} 
Let $H$ be a deformed Wigner-type matrix of the real symmetric or complex Hermitian symmetry class satisfying Assumptions \ref{bdd moments}-\ref{bdd m} and $\tau=\cu$ the location of an \emph{exact cusp}. Then $ N\rho((-\infty, \cu)) = k_\cu$ for some  $k_\cu \in [N]$, that we call  the cusp eigenvalue index. For any $\epsilon>0$, $\nu \in \N$ and $k \in [N]$ with $\abs{k-k_\cu} \le c N$ we have 
\[
\P\bigg[ \abs{\lambda_k-\gamma_k} \ge \frac{N^\epsilon}{(1+\abs{k-k_\cu})^{1/4}N^{3/4}}\bigg]\le \frac{C}{N^\nu},
\]
 where $C=C(\epsilon,\nu)$ and $\gamma_k$ are the self-consistent eigenvalue locations, defined through $ N\rho((-\infty, \gamma_k)) = k$.
\end{corollary}
We remark that a variant of Corollary \ref{crl:Cusp rigidity} holds more generally for almost cusp points. It is another consequence of Corollary \ref{crl:Uniform rigidity} that with high probability there are no eigenvalues much further than the fluctuation scale $\eta_{\mathrm{f}}$ away from the spectrum. We note that the following corollary generalises \cite[Corollary 2.3]{1804.07744} by also covering internal gaps of size $\ll1$.
\begin{corollary}[No eigenvalues outside the support of the self-consistent density]\label{cor no eigenvalues outside}
Let $\tau \not \in \supp \rho$. Under the assumptions of Theorem \ref{thr:Local law} we have 
\[\P\Big[\exists\lambda\in\Spec H\cap[\tau-c,\tau+c], \dist(\lambda,\supp\rho)\ge N^{\epsilon}\eta_{\mathrm{f}}(\tau)\Big]\le C N^{-\nu},\]
for any $\epsilon,\nu>0$, 
where  $c$ and $C$ are positive constants, depending on model parameters.  The latter also depends on $\epsilon$ and $\nu$. 
\end{corollary}

\begin{remark} \label{rmk:Uniformprimitivity}
Theorem~\ref{thr:Local law}  and its consequences, Corollaries~\ref{crl:Uniform rigidity}, \ref{crl:Cusp rigidity} and \ref{cor no eigenvalues outside} also hold for both symmetry classes if Assumption \ref{fullness} is replaced by the
condition that there exists an $L \in \N$ and $c>0$  such that  $\min_{i,j}(S^L)_{ij} \ge c/N$.  A variance profile S
satisfying this condition is called uniformly primitive (cf.~\cite[Eq.~(2.5)]{MR3684307} and \cite[Eq.~(2.11)]{1506.05095}). Note that uniform primitivity is weaker than 
condition \ref{fullness} on two accounts.  First, it involves only the variance matrix $\E \abs{w_{ij}}^2$ unlike \eqref{fullness cplx} in the complex Hermitian case that also involves $\E w_{ij}^2$. Second, uniform primitivity allows certain matrix elements of $W$ to vanish. The proof under these more general assumptions follows the same strategy but requires minor modifications within the stability analysis\footnote{See Appendix B of \href{https://arxiv.org/abs/1809.03971v2}{\texttt{arXiv:1809.03971v2}} for details}.
\end{remark}

\section{Local Law}
 In order to directly appeal to recent results on the shape of solution to Matrix Dyson Equation (MDE) from \cite{1804.07752} and the flexible diagrammatic cumulant expansion from \cite{MR3941370}, we first reformulate the Dyson equation \eqref{Dyson equation} for $N$-vectors $\vm$ into a matrix equation that will approximately be satisfied by the resolvent $G$. This viewpoint also allows us to treat diagonal and off-diagonal elements of $G$ on the same footing. In fact, \eqref{Dyson equation} is a special case of
\begin{equation} \label{MDE} 1+(z-A+\mathcal{S}[M])M=0, \end{equation}
for a matrix $M=M(z) \in \C^{N \times N}$ with positive definite imaginary part, $\Im M =(M-M^*)/2\ii>0$. The uniqueness of the solution $M$ with $\Im M>0$ was shown in \cite{MR2376207}. Here the linear (\emph{self-energy}) operator $\mathcal{S}\colon \C^{N \times N} \to \C^{N \times N}$ is defined as $\mathcal{S}[R]\defeq  \E WRW$ and it preserves the cone of positive definite matrices. Definition~\ref{def scdos} of the scDOS and its harmonic extension $\rho(z)$ (cf. Theorem~\ref{thr:Local law}) directly generalises to the solution to \eqref{MDE}, see \cite[Definition 2.2]{1804.07752}.

In the special case of Wigner-type matrices the self-energy operator is given by
\begin{equation}\label{cal S def}
 \SS[R] =  \diag \big(S\vr\big)+T \odot R^t,
\end{equation}
where $\vr \defeq  (r_{ii})_{i=1}^N$, $S$ was defined in \eqref{S def},  $T = (t_{ij})_{i,j=1}^N \in \C^{N \times N}$ with $t_{ij}=\E w_{ij}^2 \1(i \neq j)$ and $\odot$ denotes the entrywise Hadamard product. The solution to \eqref{MDE} is then given by $M=\diag(\vm)$, where $\vm$ solves \eqref{Dyson equation}.  Note that the action of  $\mathcal{S}$ on diagonal matrices is independent of
$T$, hence the Dyson equation \eqref{Dyson equation} for  Wigner-type matrices is solely determined by the matrix  $S$, the matrix $T$ plays no role. However, $T$ plays  a role
in analyzing the error matrix  $D$, see \eqref{error matrix D} below.

The proof of the local law consists of three largely separate arguments. The first part concerns the analysis of the stability operator
\begin{equation}\label{stability operator}
\BO \defeq 1-M\SS[\cdot]M
\end{equation}
and shape analysis of the solution $M$ to \eqref{MDE}. The second part is proving that the resolvent $G$ is indeed an approximate solution to \eqref{MDE} in the sense that the error matrix
\begin{equation}\label{error matrix D}
D\defeq 1+(z-A+\SS[G])G = WG+\SS[G]G
\end{equation} 
is small. In previous works \cite{1804.07744,MR3941370,MR3719056} it was sufficient to establish smallness of $D$ in an isotropic form $\braket{\vx,D\vy}$ and averaged form $\braket{BD}$ with general bounded vectors/matrices $\vx,\vy,B$. In the vicinity of a cusp, however, it becomes necessary to establish an additional cancellation when $D$ is averaged against the unstable direction of the stability operator $\BO$. We call this new effect \emph{cusp fluctuation averaging}. Finally, the third part of the proof consists of a bootstrap argument starting far away from the real axis and iteratively lowering the imaginary part $\eta=\Im z$ of the spectral parameter while maintaining the desired bound on $G-M$.

\begin{remark}  \label{rmk extension to correlated}
We remark that the proofs of Theorem \ref{thr:Local law}, and Corollaries \ref{crl:Uniform rigidity}, \ref{cor no eigenvalues outside} use the independence assumption on the entries of $W$ only very locally. In fact, only the proof of a specific bound on $D$ (see \eqref{cusp FA} later), which follows directly from the main result of the diagrammatic cumulant expansion, Theorem \ref{thm pfD bound moments}, uses the vector structure and the specific form of $\SS$ in \eqref{cal S def} at all. Therefore, assuming \eqref{cusp FA} as an input, our proof of Theorem \ref{thr:Local law} remains valid also in the correlated setting of \cite{1804.07744,MR3941370}, as long as $\SS$ is flat (see \eqref{flatness of cal S} below), and Assumption \ref{bdd m} is replaced by the corresponding assumption on the boundedness of $\norm{M}$.
\end{remark}

For brevity we will carry out the proof of Theorem \ref{thr:Local law} only in the vicinity of almost cusps as the local law in all other regimes was already proven in \cite{1804.07744,MR3719056} to optimality.
Therefore, within this section we will always assume that $z = \tau +\ii\eta =\tau_0+\omega +\ii\eta \in \HC$ lies inside a small neighbourhood
\[
\DD_{\mathrm{cusp}}\defeq \Set{z \in \HC| \abs{z-\tau_0} \le c},
\] 
of the location $\tau_0$ of a  local minimum of the scDOS within the self-consistent spectrum $\supp \rho$. Here $c$ is a sufficiently small constant depending only on the model parameters. We will further assume that either  (i) $\rho(\tau_0)\ge 0$ is sufficiently small and $\tau_0$ is the location of a cusp or internal minimum,  or
(ii)  $\rho(\tau_0)=0$ and $\tau_0$ is an edge adjacent to a sufficiently small gap of length $\Delta>0$. The results from \cite{1804.07752} guarantee that these are the only possibilities for the shape of $\rho$, see \eqref{gamma def eqs}. In other words, we assume that $\tau_0 \in \supp \rho$ is a local minimum of $\rho$ with a shape close to a cusp (cf.~\eqref{gamma def eqs}). For concreteness we will also assume that if $\tau_0$ is an edge, then it is a right edge (with a gap of length $\Delta>0$ to the right) and $\omega \in (-c, \frac{\Delta}{2}]$. The case when $\tau_0$ is a left edge has the same proof. 

 We now introduce a quantity that will play an important role in the cusp fluctuation averaging mechanism. We define
\begin{subequations}
\begin{equation} \label{definition of sigma mde} \sigma(z)\defeq \braket{(\sgn \Re U)(\Im U/\rho)^3},\quad U \defeq  \frac{(\Im M)^{-1/2}(\Re M)(\Im M)^{-1/2} + \ii}{\abs[0]{(\Im M)^{-1/2}(\Re M)(\Im M)^{-1/2} + \ii}},\end{equation}
where $\Re M\defeq (M+M^\ast)/2$ is the real part of $M=M(z)$. It was proven in \cite[Lemma 5.5]{1804.07752} that $\sigma(z)$ extends to the real line as a $1/3$-H\"older continuous function wherever the scDOS $\rho$ is smaller than some threshold $c\sim 1$, i.e. $\rho\le c$. In the specific case of $\SS$ as in \eqref{cal S def} the definition simplifies to
\begin{equation} \label{definition of sigma} \sigma(z)\defeq \braket{\vp\vf^3} = \frac{1}{N}\sum_{i=1}^N \frac{(\Im m_i(z))^3\sgn \Re m_i(z)}{\rho(z)^3\abs{m_i(z)}^3},\qquad \vf\defeq \frac{\Im\vm}{\rho\abs{\vm}},\qquad \vp\defeq\sgn\Re\vm ,\end{equation}
\end{subequations}  
since $M=\diag(\vm)$ is diagonal, where multiplication and division of vectors are understood entrywise. When evaluated at the location $\tau_0$ the scalar  $\sigma(\tau_0)$  provides a measure of how far the  shape
of the singularity at $\tau_0$ is  from an exact cusp. In fact, if $\sigma(\tau_0)=0$  and $\rho(\tau_0)=0$, then $\tau_0$ is a cusp location. To see the relationship between the emergence of a cusp and the limit $\sigma(\tau_0) \to 0$, we refer to \cite[{Theorem~7.7} and Lemma~6.3]{1804.07752}. The analogues of the quantities $\vf,\vp$ and $\sigma$ in \eqref{definition of sigma} are denoted by $f_u,s$ and $\sigma$ in \cite{1804.07752}, respectively. The significance of $\sigma$ for the classification of singularity types in Wigner-type ensembles was first realised in \cite{1506.05095}. Although in this paper we will use only \cite{1804.07752} and will not rely on \cite{1506.05095}, we remark that the definition of $\sigma$ in \cite[Eq.~(8.11)]{1506.05095} differs slightly from the definition \eqref{definition of sigma}. However, both definitions equally fulfil the purpose of  classifying singularity types, since the ensuing scalar quantities $\sigma$ are comparable  inside the self-consistent spectrum. For the interested reader, we briefly relate our notations to the respective conventions in \cite{1804.07752} and \cite{1506.05095}. The quantity denoted by $f$ in both \cite{1804.07752} and \cite{1506.05095} is the normalized eigendirection of the \emph{saturated self-energy operator} $F$ in the respective settings and is related to $\vf$ from \eqref{definition of sigma} via $f=\vf / \norm{\vf} +\landauO{\eta/\rho}$. Moreover, $\sigma$ in \cite{1506.05095} is defined as $\braket{f^3 \sgn\Re\vm}$, justifying the comparability to $\sigma$ from \eqref{definition of sigma}.

\subsection{Stability and shape analysis}
From \eqref{MDE} and \eqref{error matrix D} we obtain the quadratic \emph{stability equation}
\[
\mathcal{B}[G-M]= -MD+M\mathcal{S}[G-M](G-M),
\]
for the difference $G-M$. In order to apply the results of \cite{1804.07752} to the stability operator $\BO$, we first have to check that the flatness condition \cite[Eq.~(3.10)]{1804.07752} is satisfied for the self-energy operator $\SS$. We claim that $\SS$ is flat, i.e.
\begin{equation} \label{flatness of cal S} \mathcal{S}[R] \sim \braket{R}1= \frac{1}{N} (\Tr R) 1 , \end{equation}
as quadratic forms for any positive semidefinite $R \in \C^{N \times N}$. We remark that in the earlier paper  \cite{MR3719056} in the Wigner-type case
only the upper bound  $s_{ij}\le C/N$  defined the concept of flatness.
Here with the definition \eqref{flatness of cal S} we follow the convention of  the more recent works \cite{1804.07752,1804.07744,MR3941370}
which is more conceptual.
We also warn the reader, that in the complex Hermitian Wigner-type case the condition $c/N\le s_{ij}\le C/N$ implies \eqref{flatness of cal S} only if 
  $t_{ij}$ is bounded away from $-s_{ij}$.

However,  the flatness \eqref{flatness of cal S} is an immediate consequence of the fullness Assumption \ref{fullness}. 
Indeed, \ref{fullness} is equivalent to the condition that the covariance operator $\Sigma$ of all entries above and on the diagonal, defined as $\Sigma_{ab,cd}\defeq \E w_{ab} w_{cd}$, is uniformly strictly positive definite. This implies that $\Sigma \ge c \Sigma_{\mathrm{G}}$ for some constant $c\sim1$, where $\Sigma_{\mathrm{G}}$ is the covariance operator of a GUE or GOE matrix, depending on the symmetry class we consider. This means that $\mathcal{S}$ can be split into $\mathcal{S}=\mathcal{S}_0+c \mathcal{S}_{\mathrm{G}}$, where $\mathcal{S}_\mathrm{G}$ and $\SS_0$ are the self-energy operators corresponding to $\Sigma_\mathrm{G}$ and $\Sigma-c\Sigma_\mathrm{G}$, respectively. It is now an easy exercise to check that $\mathcal{S}_{\mathrm{G}}$ and thus $\mathcal{S}$ is flat. 

In particular, \cite[Proposition~3.5 and Lemma~4.8]{1804.07752} are applicable implying that \cite[Assumption~4.5]{1804.07752} is satisfied. Thus, according to \cite[Lemma~5.1]{1804.07752} for spectral parameters $z$ in a neighbourhood of $\tau_0$ the operator $\mathcal{B}$ has a unique isolated eigenvalue $\beta$ of smallest modulus and associated right $\BO[V_\mathrm{r}]=\beta V_\mathrm{r}$ and left $\mathcal{B}^*[V_\mathrm{l}]= \overline{\beta} V_\mathrm{l}$ eigendirections normalised such that $\norm{V_\mathrm{r}}_{\mathrm{hs}} =\scalar{V_\mathrm{l}}{V_\mathrm{r}} =1$. We denote the spectral projections to $V_\mathrm{r}$ and to its complement by $\cP\defeq \braket{V_\mathrm{l},\cdot}V_\mathrm{r}$ and $\cQ\defeq 1-\cP$. For convenience of the reader we now collect some important quantitative information about the stability operator and its unstable direction from \cite{1804.07752}.
\begin{proposition}[Properties of the MDE and its solution]\label{prop mde}
The following statements hold true uniformly in $z=\tau_0+\omega+\ii\eta\in\DD_\cusp$ assuming flatness as in \eqref{flatness of cal S} and the uniform boundedness of $\norm{M}$ for $z\in\tau_0+(-\kappa,\kappa)+\ii\R_+$,
\begin{subequations}
\begin{enumerate}[(i)]
\item The eigendirections $V_\mathrm{l},V_\mathrm{r}$ are norm-bounded and the operator $\BO^{-1}$ is bounded on the complement to its unstable direction, i.e.
\begin{equation} \label{bounded eigendirections} \norm{\BO^{-1}\cQ}_{\hs\to\hs}+\norm{V_\mathrm{r}}+\norm{V_\mathrm{l}}\lesssim 1.\end{equation}
\item The density $\rho$ is comparable with the explicit function $\rho(\tau_0+\omega+\ii\eta)\sim\wt\rho(\tau_0+\omega+\ii \eta)$ given by 
\begin{equation}\label{rho tilde}
\wt\rho \defeq \begin{cases}
\rho(\tau_0)+(\abs{\omega}+\eta)^{1/3},&\text{in cases (i),(iii) if }\tau_0=\mi,\cu,\\
(\abs{\omega}+\eta)^{1/2}(\Delta+\abs{\omega}+\eta)^{-1/6},&\text{in case (ii) if }\tau_0=\ed_-,\; \omega\in[-c,0]\\
\eta(\Delta+\abs{\omega}+\eta)^{-1/6}(\abs{\omega}+\eta)^{-1/2},&\text{in case (ii) if }\tau_0=\ed_-,\; \omega\in[0,\Delta/2].\\
\end{cases}
\end{equation}
\item \label{prop mde beta comp rel} The eigenvalue $\beta$ of smallest modulus satisfies
\begin{equation}\label{beta asymp}
\abs{\beta} \sim \frac{\eta}{\rho} + \rho(\rho+\abs{\sigma}),
\end{equation}
and we have the comparison relations 
\begin{equation}\label{Vl Vr inner prod asymp}
\begin{split}
  &\abs{\braket{V_\mathrm{l}, M \SS[V_\mathrm{r}]V_\mathrm{r}}} \sim \rho+\abs{\sigma}, \\ 
  &\abs{\braket{V_\mathrm{l},M\SS[V_\mathrm{r}]\BO^{-1}\cQ [M\SS[V_\mathrm{r}]V_\mathrm{r}]+M\SS\BO^{-1}\cQ [M\SS[V_\mathrm{r}]V_\mathrm{r}]V_\mathrm{r}}}\sim 1.
\end{split}
\end{equation}
\item \label{prop mde wt xi} The quantities $\eta/\rho+\rho(\rho+\abs{\sigma})$ and $\rho+\abs{\sigma}$ in \eqref{beta asymp}--\eqref{Vl Vr inner prod asymp} can be replaced by the following more explicit auxiliary quantities 
\begin{equation}\label{xi tilde def}
\begin{split}
  \wt\xi_1(\tau_0+\omega+\ii\eta)&\defeq \begin{cases}
    (\abs{\omega}+\eta)^{1/2} (\abs{\omega}+\eta+\Delta)^{1/6},\\
    (\rho(\tau_0)+(\abs{\omega}+\eta)^{1/3})^2,
    \end{cases}\\
    \wt\xi_2(\tau_0+\omega+\ii\eta)&\defeq \begin{cases}
    (\abs{\omega}+\eta+\Delta)^{1/3}, &\text{if } \tau_0=\ed_-,\\
    \rho(\tau_0)+(\abs{\omega}+\eta)^{1/3}, &\text{if }\tau_0=\mi,\cu.
    \end{cases}
\end{split}
\end{equation}
which are monotonically increasing in $\eta$. More precisely, it holds that $\eta/\rho + \rho(\rho+\abs{\sigma}) \sim \wt\xi_1$ and, in the case where $\tau_0=\cu,\mi$ is a cusp or a non-zero local minimum, we also have that \smash{$\rho+\abs{\sigma}\sim \wt\xi_2$}. For the case when $\tau_0=\ed_-$ is a right edge next to a gap of size $\Delta$ there exists a constant $c_\ast$ such that \smash{$\rho+\abs{\sigma}\sim \wt\xi_2$} in the regime $\omega\in[-c,c_\ast \Delta]$ and \smash{$\rho+\abs{\sigma}\lesssim \wt\xi_2$} in the regime $\omega\in[c_\ast \Delta,\Delta/2]$.
\end{enumerate}
\end{subequations}
\end{proposition}
\begin{proof} We first explain how to translate the notations from the present paper to the notations in \cite{1804.07752}: The operators $\SS,\BO,\cQ$  are simply denoted by $S,B,Q$ in \cite{1804.07752}; the matrices $V_l,V_r$ here are denoted by $l/\scalar{l}{b},b$ there.
 The bound on $\BO^{-1}\cQ$ in \eqref{bounded eigendirections} follows directly from \cite[Eq.~(5.15)]{1804.07752}. The bounds on $V_\mathrm{l},V_\mathrm{r}$ in \eqref{bounded eigendirections} follow from the definition of the stability operator \eqref{stability operator} together with the fact that $\norm{M} \lesssim 1$ (by Assumption \ref{bdd m}) and $\norm{\mathcal{S}}_{\mathrm{hs} \to \norm{\cdot}} \lesssim 1$, following from the upper bound in  flatness \eqref{flatness of cal S}. The asymptotic expansion of $\rho$ in \eqref{rho tilde} follows from \cite[{Remark 7.3}]{1804.07752} and \cite[Corollary A.1]{1506.05095}. The claims in \eqref{prop mde beta comp rel} follow directly from \cite[Proposition 6.1]{1804.07752}. Finally, the claims in \eqref{prop mde wt xi} follow directly from \cite[{Remark 10.4}]{1804.07752}. 
\end{proof}
The following lemma establishes simplified lower bounds on $\wt\xi_1,\wt\xi_2$ whenever $\eta$ is much larger than the fluctuation scale $\eta_\mathrm{f}$. We defer the proof of the technical lemma which differentiates various regimes to the appendix. 
\begin{lemma}\label{lemma tilde xi}
Under the assumptions of Proposition \ref{prop mde} we have uniformly in $z=\tau_0+\omega+\ii\eta\in\DD_\cusp$ with $\eta\ge  \eta_\mathrm{f}$ that
\[ \wt\xi_2\gtrsim \frac{1}{N\eta}+\Bigl(\frac{\rho}{N\eta}\Bigr)^{1/2}, \qquad \wt\xi_1\gtrsim \wt\xi_2\Bigl(\rho+\frac{1}{N\eta}\Bigr).\]
\end{lemma}

We now define an appropriate matrix norm in which we will measure the distance between $G$ and $M$. The $\norm{\cdot}_\ast$-norm is defined exactly as in \cite{1804.07744} and similar to the one first introduced in \cite{MR3941370}.
It is a norm comparing matrix elements on a large but finite set of vectors with a hierarchical structure. To define this set we introduce some notations. 
 For second order cumulants of matrix elements $\kappa(w_{ab},w_{cd})\defeq \E w_{ab}w_{cd}$
 we use the short-hand notation $\kappa(ab,cd)$. We also use the short-hand notation $\kappa(\vx b,cd)$ for the $\vx=(x_a)_{a\in[N]}$-weighted linear combination $\sum_a x_a \kappa(ab,cd)$ of such cumulants. We use the notation that replacing an index in a scalar quantity by a dot ($\cdot$) refers to the corresponding vector, e.g.~$A_{a\cdot}$ is a short-hand notation for the vector $(A_{ab})_{b\in[N]}$. Matrices $R_{\vx\vy}$ with vector subscripts $\vx,\vy$ are understood as short-hand notations for $\braket{\vx,R\vy}$, and matrices $R_{\vx a}$ with mixed vector and index subscripts are understood as $\braket{\vx,R e_a}$ with $e_a$ being the $a$-th normalized $\norm{e_a}=1$ standard basis vector.  We fix two vectors $\vx,\vy$ and some large integer $K$ and define the sets of vectors
\[
\begin{split}
I_0 &\defeq  \{\vx,\vy\} \cup\Set{\delta_{a\cdot}, (V_\mathrm{l}^*)_{a\cdot}| a \in [N]},
\\
I_{k+1} &\defeq  I_k\cup \Set{M{\vu}|{\vu}\in I_k}\cup \Set{\kappa_\mathrm{c}((M{\vu})a,b\cdot),\kappa_\mathrm{d}((M{\vu})a,\cdot b)|{\vu}\in I_k,a,b\in [N]}.
\end{split}
\]
Here the cross and the direct part $\kappa_\mathrm{c},\kappa_\mathrm{d}$ of the 2-cumulants $\kappa(\cdot,\cdot)$ refer to the natural splitting dictated by the Hermitian symmetry. In the specific case of \eqref{cal S def} we simply have $\kappa_\mathrm{c}(ab,cd)=\delta_{ad}\delta_{bc}s_{ab}$ and $\kappa_\mathrm{d}(ab,cd)=\delta_{ac}\delta_{bd}t_{ab}$.
Then the $\norm{\cdot}_\ast$-norm is given by
\[
\norm{R}_\ast = \norm{R}_\ast^{K,{\vx},{\vy}} \defeq  \sum_{0\le k< K}N^{-k/2K} \norm{R}_{I_k} + N^{-1/2} \max_{{\vu}\in I_K}\frac{\norm{R_{\cdot{\vu}}}}{\norm{{\vu}}},\qquad \norm{R}_I \defeq  \max_{{\vu},{\vv}\in I} \frac{\abs{R_{{\vu}{\vv}}}}{\norm{{\vu}}\norm{{\vv}}}. 
\]
We remark that the set $I_k$ hence also $\norm{\cdot}_\ast$ depend on $z$ via $M=M(z)$. We omit this dependence from the notation as it plays no role in the estimates. 

In terms of this norm we obtain the following estimate on $G-M$ in terms of its projection $\Theta=\braket{V_\mathrm{l},G-M}$ onto the unstable direction of the stability operator $\BO$. It is a direct consequence of a general expansion of approximate quadratic matrix equations whose linear stability operators have a single eigenvalue close to $0$, as given in Lemma \ref{lmm:cubic equation}.
\begin{proposition}[Cubic equation for $\Theta$] \label{prp:Cubic equation for Theta} 
\begin{subequations}
Fix $K\in\N$, $\vx,\vy\in\C^N$ and use $\norm{\cdot}_\ast=\norm{\cdot}_\ast^{K,\vx,\vy}$. For fixed $z \in \DD_{\mathrm{cusp}}$  and on the event that $\norm{G-M}_\ast +\norm{D}_\ast \lesssim N^{-10/K}$ the difference $G-M$ admits the expansion
\begin{equation} \label{G-M expansion} 
\begin{split}
G-M&= \Theta  V_\mathrm{r} -\mathcal{B}^{-1}\mathcal{Q}[MD]+\Theta^2\mathcal{B}^{-1}\mathcal{Q}[M\mathcal{S}[V_\mathrm{r}]V_\mathrm{r}]  + E,\\
\norm{E}_\ast&\lesssim N^{5/K}(\abs{\Theta}^3 +\abs{\Theta}\norm{D}_\ast+ \norm{D}_\ast^2), \end{split}\end{equation}
with an error matrix $E$ and the scalar $\Theta\defeq \braket{V_\mathrm{l}, G-M}$ that satisfies the approximate cubic equation
\begin{equation} \label{Theta cubic equation}   \Theta^3+ \xi_2 \Theta^2 +\xi_1 \Theta= \epsilon_\ast. \end{equation}
Here, the error $\epsilon_\ast$ satisfies the upper bound
\begin{equation} \label{bound on eps} \abs{\epsilon_\ast}\lesssim N^{20/K}(\norm{D}_\ast^3+\abs{\braket{R, D}}^{3/2})+ \abs{\braket{V_\mathrm{l}, MD}}  + \abs{\braket{V_\mathrm{l},M(\mathcal{S}\mathcal{B}^{-1}\mathcal{Q}[MD])(\mathcal{B}^{-1}\mathcal{Q}[MD])}}, \end{equation}
where $R$ is a deterministic matrix with $\norm{R}\lesssim 1$
and the coefficients of the cubic equation satisfy the comparison relations
\begin{equation} \label{xi comparison relations} \abs{\xi_1} \sim \frac{\eta}{\rho}+\rho(\rho +\abs{\sigma}),\qquad \abs{\xi_2} \sim \rho +\abs{\sigma}. \end{equation}
\end{subequations}
\end{proposition}
\begin{proof} 
We first establish some important bounds involving the $\norm{\cdot}_\ast$-norm.  We claim that for any matrices $R,R_1,R_2$
\begin{equation}\begin{split} \label{bounds on star norm} &\norm{M\mathcal{S}[R_1]R_2}_\ast\lesssim  N^{1/2K}\norm{R_1}_\ast\norm{R_2}_\ast,\quad 
\norm{MR}_\ast\lesssim  N^{1/2K}\norm{R}_\ast,
\\
&\norm{\mathcal{Q}}_{\ast \to \ast}\lesssim 1,\quad\norm{\mathcal{B}^{-1}\mathcal{Q}}_{\ast \to \ast}\lesssim 1,\quad \abs{\braket{V_\mathrm{l}, R}} \lesssim \norm{R}_\ast. \end{split}\end{equation}
The proof of \eqref{bounds on star norm} follows verbatim as in \cite[Lemma 3.4]{1804.07744} with \eqref{bounded eigendirections} as an input. Moreover, the bound on $\braket{V_\mathrm{l}, \cdot}$ follows directly from the bound on $\mathcal{Q}$. Obviously, we also have $\norm{\cdot}_\ast\le 2 \norm{\cdot}$. 

Next, we apply Lemma~\ref{lmm:cubic equation} from the appendix with the choices 
\[
 \mathcal{A}[R_1,R_2]\defeq M\mathcal{S}[R_1]R_2,\qquad {X}\defeq  MD, \qquad Y\defeq  G-M. 
\]
The operator $\mathcal{B}$ in Lemma~\ref{lmm:cubic equation} is chosen as the stability operator \eqref{stability operator}. Then \eqref{upper bound in cubic lemma} is satisfied with $\lambda\defeq  N^{1/2K}$ according to \eqref{bounds on star norm} and \eqref{bounded eigendirections}. With $\delta\defeq N^{-25/4K}$ we verify \eqref{G-M expansion} directly from \eqref{R leading order}, where $\Theta= \braket{V_\mathrm{l}, G-M}$ satisfies
\begin{equation} \mu_3 \Theta^3+\mu_2 \Theta^2 -\beta \Theta = -\mu_0 + \braket{R, D} \Theta+\landauO{N^{-1/4K} \abs{\Theta}^3 + N^{20/K}\norm{D}_\ast^3}. \label{mu Theta cubic eq}\end{equation}
Here we used $\abs{\Theta}\le \norm{G-M}_\ast \lesssim N^{-10/K}$  and  $\norm{MD}_\ast \lesssim N^{1/2K}\norm{D}_\ast$. The coefficients $\mu_0, \mu_2, \mu_3$ are defined through \eqref{coefficients cubic lemma} and $R$ is given by 
\[
R \defeq  M^*(\mathcal{B}^{-1}\mathcal{Q})^*[\mathcal{S}[M^*V_\mathrm{l} V_\mathrm{r}^*]+\mathcal{S}[V_\mathrm{r}^*]M^*V_\mathrm{l}].
\]
Now we bound $ \abs{\braket{R, D} \Theta} \le N^{-1/4K} \abs{\Theta}^3 + N^{1/8K} \abs{\braket{R, D}}^{3/2}$ by Young's inequality, absorb the error terms bounded by  $N^{-1/4K} \abs{\Theta}^3$ into the cubic term, $\mu_3 \Theta^3 + \ord (N^{-1/4K} \abs{\Theta}^3) = \wt{\mu}_3 \Theta^3$, by introducing a modified coefficient $\wt{\mu}_3$ and use that $\abs{\mu_3} \sim \abs{\wt{\mu}_3} \sim 1$ for any $z \in \DD_{\mathrm{cusp}}$. Finally, we safely divide \eqref{mu Theta cubic eq} by $\wt\mu_3$ to verify \eqref{Theta cubic equation} with $\xi_1\defeq  -\beta / \wt{\mu}_3$ and $\xi_2 \defeq  \mu_2 / \wt{\mu}_3$. For the fact $\abs{\mu_3} \sim 1$ on $\DD_{\mathrm{cusp}}$ and the comparison relations \eqref{xi comparison relations} we refer to \eqref{beta asymp}--\eqref{Vl Vr inner prod asymp}. 
\end{proof}

\subsection{Probabilistic bound} We now collect bounds on the error matrix $D$ from \cite[Theorem 4.1]{MR3941370} and Section \ref{sec:Cusp fluctuation averaging}. We first  introduce the notion of \emph{stochastic domination}.
\begin{definition}[Stochastic domination]
Let $X=X^{(N)}, Y=Y^{(N)}$ be sequences of non-negative random variables. We say that $X$ is stochastically dominated by $Y$ (and use the notation $X \prec Y$) if  
\[
\P\big[X > N^\epsilon Y\big]\le C(\epsilon,\nu)N^{-\nu},\qquad N \in \N,
\]
for any $\epsilon>0, \nu \in \N$ and some  family of positive constants $C(\epsilon,\nu)$ that is uniform in $N$ and other underlying parameters (e.g.~the spectral parameter $z$ in the domain under consideration). 
\end{definition}
It can be checked (see \cite[Lemma 4.4]{MR3068390}) that $\prec$ satisfies the usual arithmetic properties, e.g.~if $X_1\prec Y_1$ and $X_2\prec Y_2$, then also  $X_1+X_2\prec Y_1 +Y_2$ and  $X_1X_2\prec Y_1 Y_2$. Furthermore, to formulate bounds on a random matrix $R$ compactly, we introduce the notations
\begin{alignat}{3}\nonumber
\abs{R}\prec \Lambda \quad &\Longleftrightarrow \quad  &\abs{R_{\vx\vy}}&\prec \Lambda \norm{\vx}\norm{\vy}&\quad  &\text{uniformly for all } {\vx},{\vy} \in \C^N,\\\nonumber
\abs{R}_{\mathrm{av}}\prec \Lambda \quad &\Longleftrightarrow \quad &\abs{\braket{ BR}}&\prec \Lambda \norm{B}&  &\text{uniformly for all } B \in \C^{N \times N}
\end{alignat}
for random matrices $R$ and a deterministic control parameter $\Lambda=\Lambda(z)$. We also introduce high moment norms 
\[ 
\norm{X}_p\defeq \Big(\E\abs{X}^p\Big)^{1/p},\qquad \norm{R}_p \defeq \sup_{\vx,\vy} \frac{\norm{\braket{\vx, R\vy}}_p}{\norm{\vx}\norm{\vy}} 
 \]
for $p\ge 1$, scalar valued random variables $X$ and random matrices $R$. To translate high moment bounds into high probability bounds and vice versa we have the following easy lemma \cite[Lemma 3.7]{1804.07744}.
\begin{lemma}\label{prec p conversion}
Let $R$ be a random matrix, $\Phi$ a deterministic control parameter such that $\Phi\ge N^{-C}$ and $\norm{R}\le N^C$ for some $C>0$, and let $K\in\N$ be a fixed integer. Then we have the equivalences
\[ \norm{R}_\ast^{K,\vx,\vy}\prec \Phi\text{ uniformly in }\vx,\vy\quad\Longleftrightarrow\quad \abs{R}\prec \Phi \quad\Longleftrightarrow\quad \norm{R}_p\le_{p,\epsilon} N^{\epsilon}\Phi\text{ for all }\epsilon>0,\,p\ge 1.\]
\end{lemma}

Expressed in terms of the $\norm{\cdot}_p$-norm we have the following high-moment bounds on the error matrix $D$. The bounds \eqref{isotropic bound on D}--\eqref{averaged bound on D} have  already been established  in \cite[Theorem 4.1]{MR3941370};
we just list them for completeness. The bounds \eqref{eq pf D moment bound}--\eqref{TGG bound}, however, are new and they
 capture the additional cancellation at the cusp and are the core novelty of the present paper. The additional smallness comes from averaging against specific weights $\vp,\vf$ from \eqref{definition of sigma}.

\begin{theorem}[High moment bound on $D$ with cusp fluctuation averaging]\label{thm pfD bound moments}
Under the assumptions of Theorem \ref{thr:Local law} for any compact set $\DD\subset\Set{z\in\C|\Im z\ge N^{-1}}$ there exists a constant $C$ such that for any $p\ge 1,\epsilon>0$, $z\in \DD$ and matrices/vectors $B,\vx,\vy$ it holds that
\begin{subequations}
\begin{alignat}{1}
\norm{\braket{\vx,D\vy}}_p &\le_{\epsilon,p} \norm{\vx}\norm{\vy} N^{\epsilon} \psi_q'\Big(1+\norm{G}_q \Big)^C \bigg(1+\frac{\norm{G}_q}{\sqrt N}\bigg)^{Cp},\label{isotropic bound on D}\\
\norm{\braket{BD}}_p &\le_{\epsilon,p} \norm{B} N^{\epsilon} \Big[\psi_q'\Big]^2\Big(1+\norm{G}_q \Big)^C \bigg(1+\frac{\norm{G}_q}{\sqrt N}\bigg)^{Cp}.\label{averaged bound on D}\\
\intertext{Moreover, for the specific weight matrix $B=\diag(\vp\vf)$ we have the improved bound}
\norm{\braket{\diag(\vp\vf)D}}_p &\le_{\epsilon,p}N^{\epsilon}  \sigma_q \Big[ \psi+\psi'_q \Big]^2 \Big(1+\norm{G}_q \Big)^C \bigg(1+\frac{\norm{G}_q}{\sqrt N}\bigg)^{Cp}, \label{eq pf D moment bound}\\
\intertext{and the improved bound on the off-diagonal component}
\label{TGG bound}
\norm{\braket{\diag(\vp\vf) [T\odot G^t]G}}_p &\le_{\epsilon,p} N^{\epsilon}\sigma_q \Big[ \psi+\psi'_q \Big]^2 \Big(1+\norm{G}_q\Big)^C\bigg(1+\frac{\norm{G}_q}{\sqrt N}\bigg)^{Cp},
\end{alignat}
\end{subequations} 
where we defined the following $z$-dependent quantities
\[ 
\psi \defeq \sqrt{\frac{\rho}{N\eta}},\quad\psi'_q \defeq \sqrt{\frac{\norm{\Im G}_q}{N\eta}},\quad \psi''_q \defeq \norm{G-M}_q,\quad \sigma_q\defeq\abs{\sigma}+\rho+\psi+\sqrt{\eta/\rho}+\psi_q'+\psi_q''
  \]
 and $q=Cp^3/\epsilon$.
\end{theorem}

Theorem \ref{thm pfD bound moments} will be proved in Section \ref{sec:Cusp fluctuation averaging}. We now translate the high moment bounds of Theorem \ref{thm pfD bound moments} into high probability bounds via Lemma \ref{prec p conversion} and use those to establish bounds on $G-M$ and the error in the cubic equation for $\Theta$. To simplify the expressions we formulate the bounds in the domain
\begin{equation} \DD_\zeta\defeq \Set{z \in \DD_\cusp| \Im z \ge N^{-1+\zeta}}. \label{D zeta def}\end{equation}
\begin{lemma}[High probability error bounds]\label{lmm:Bounds on errors}
\begin{subequations}
Fix $\zeta,c>0$ sufficiently small and suppose that $\abs{G-M}\prec \Lambda$, $\abs{\Im (G-M)} \prec \Xi$ and $\abs{\Theta}\prec \theta$ hold at fixed $z \in  \DD_\zeta$, and assume that the deterministic control parameters $\Lambda, \Xi,\theta$ satisfy $\Lambda+\Xi+\theta\lesssim N^{-c}$. Then for any sufficiently small $\epsilon>0$ it holds that
\begin{equation} \label{bound on cubic in Theta} \abs{\Theta^3+ \xi_2 \Theta^2 +\xi_1 \Theta} \prec  N^{2\epsilon}\bigg(\rho + \abs{\sigma} +\frac{\eta^{1/2}}{\rho^{1/2}}+\bigg(\frac{\rho+\Xi}{N \eta}\bigg)^{1/2}\bigg)
 \frac{\rho+\Xi}{N \eta}
+N^{-\epsilon}\theta^3, \end{equation}
as well as
\begin{equation} \label{first G-M bounds} \abs{G-M}\prec \theta + \sqrt{\frac{\rho+\Xi}{N \eta}},\qquad \abs{G-M}_{\mathrm{av}}\prec \theta + {\frac{\rho+\Xi}{N \eta}},  \end{equation}
where the coefficients $\xi_1,\xi_2$ are those from Proposition \ref{prp:Cubic equation for Theta}, and we recall that $\Theta=\braket{V_l,G-M}$.
\end{subequations} 
\end{lemma}
\begin{proof}
We translate the high moment bounds \eqref{isotropic bound on D}--\eqref{averaged bound on D} into high probability bounds using Lemma \ref{prec p conversion} and $\abs{G}\prec \norm{M} + \Lambda\lesssim 1$ to find
\begin{equation} \label{Bound on error matrix} \abs{D}\prec \sqrt{\frac{\rho+\Xi}{N \eta}} ,\qquad \abs{D}_{\mathrm{av}}\prec {\frac{\rho+\Xi}{N \eta}}. \end{equation}
In particular, these bounds together with the assumed bounds on $G-M$ guarantee the applicability of Proposition \ref{prp:Cubic equation for Theta}. Now we use \eqref{Bound on error matrix} and~\eqref{bounds on star norm} in~\eqref{G-M expansion} to get \eqref{first G-M bounds}. Here we used \eqref{bounds on star norm}, translated \smash{$\norm{\cdot}_p$}-bounds into $\prec$-bounds on \smash{$\norm{\cdot}_\ast$} and vice versa via Lemma \ref{prec p conversion}, and absorbed the \smash{$N^{1/K}$} factors into $\prec$ by using that $K$ can be chosen arbitrarily large. It remains to verify \eqref{bound on cubic in Theta}. In order to do so, we first claim that 
\begin{equation}\begin{split} \label{cusp FA} &\abs{\braket{V_\mathrm{l}, MD}}  + \abs{\braket{V_\mathrm{l},M(\mathcal{S}\mathcal{B}^{-1}\mathcal{Q}[MD])(\mathcal{B}^{-1}\mathcal{Q}[MD])}}
\\
&\prec  N^\epsilon\bigg(\abs{\sigma}+\rho +{\frac{\eta^{1/2}}{\rho^{1/2}}}+ \Lambda +\bigg(\frac{\rho+\Xi}{N \eta}\bigg)^{1/2}\bigg)\frac{\rho+\Xi}{N \eta}
+\theta^2\bigg(N^{-\epsilon}\Lambda +\bigg(\frac{\rho+\Xi}{N \eta}\bigg)^{1/2}\bigg)
\end{split}\end{equation}
for any sufficiently small $\epsilon>0$.
\begin{proof}[Proof of \eqref{cusp FA}]
We first collect two additional ingredients from \cite{1804.07752} specific to the vector case. 
\begin{enumerate}[(a)]
\item \label{prop m comp} The imaginary part $\Im\vm$ of the solution $\vm$ is comparable $\Im \vm\sim \braket{\Im\vm}=\pi\rho$ to its average in the sense \(c \braket{\Im\vm}\le \Im m_i\le C \braket{\Im \vm}\) for all \(i\) and some \(c,C>0\), and, in particular, $\vm=\Re\vm+\landauO{\rho}$.
\item \label{prop eigendirections BQ} The eigendirections $V_\mathrm{l},V_\mathrm{r}$ are diagonal and  are approximately given by
\begin{equation}\label{Vl Vr expansion}
V_\mathrm{l} = c\diag(\vf/\abs{\vm}) + \landauO{\rho+\eta/\rho},\qquad V_\mathrm{r}=c' \diag(\vf\abs{\vm})+ \landauO{\rho+\eta/\rho}
\end{equation}
for some constants $c,c'\sim1$.
\end{enumerate}
Indeed, \eqref{prop m comp} follows directly from \cite[Proposition 3.5]{1804.07752} and the approximations in \eqref{Vl Vr expansion} follow directly from \cite[Corollary 5.2]{1804.07752}. The fact that $V_\mathrm{l},V_\mathrm{r}$ are diagonal follows from simplicity of the eigendirections in the matrix case, and the fact that $M=\diag (\vm)$ is diagonal and that $\mathcal{B}$ preserves the space of diagonal matrices as well as the space of off-diagonal matrices.  On the latter $\mathcal{B}$ acts stably as $1+\ord_{\mathrm{hs}\to \mathrm{hs}}(N^{-1})$. Thus the unstable directions lie inside the space of diagonal matrices.  

We now turn to the proof of \eqref{cusp FA} and first note that, according to \eqref{prop m comp} and \eqref{prop eigendirections BQ} we have
\begin{equation} \label{Expanding M and Vl} M= \diag(\vp\abs{\vm})+\landauO{\rho}, \qquad V_\mathrm{l} = c\diag(\vf/\abs{\vm})+\landauO{\rho +\eta /\rho} \end{equation}
with errors in \(\norm{\cdot}\)-norm-sense, for some constant $c \sim1$ to see
\[
\braket{V_\mathrm{l}, MD}= c\braket{\diag(\vp \vf)D} + \landauO{\rho +\eta /\rho}\braket{\diag(\vw_1)D},
\]
where $\vw_1\in \C^N$ is a deterministic vector with uniformly bounded entries.  Since $\abs{\braket{\diag(\vw_1)D}} \prec (\rho+\Xi)/N\eta$ by \eqref{Bound on error matrix}, the bound on the first term in \eqref{cusp FA} follows together with \eqref{eq pf D moment bound} via Lemma \ref{prec p conversion}. Now we consider the second term in \eqref{cusp FA}. We split $D = D_\mathrm{d} + D_\mathrm{o}$ into its diagonal and off-diagonal components. Since $\mathcal{B}$ and $\mathcal{S}$ preserve the space of diagonal and the space of off-diagonal matrices we find
\begin{equation}\label{DD splitting}
\begin{split}
&\braket{V_\mathrm{l},M(\mathcal{S}\mathcal{B}^{-1}\mathcal{Q}[MD])(\mathcal{B}^{-1}\mathcal{Q}[MD])}\\
&\quad=  \frac{1}{N^2}\sum_{i,j} u_{ij} d_{ii} d_{jj} +\braket{V_\mathrm{l},M(\mathcal{S}\mathcal{B}^{-1}\mathcal{Q}[MD_\mathrm{o}])(\mathcal{B}^{-1}\mathcal{Q}[MD_\mathrm{o}])},
\end{split}
\end{equation}
with an appropriate deterministic matrix $u_{ij}$ having bounded entries. In particular, the cross terms vanish and the first term is bounded by
\begin{equation}\label{DD diag bound}
\abs[2]{\frac{1}{N^2}\sum_{i,j} u_{ij} d_{ii} d_{jj}} \le \max_{i}\abs{d_{ii}}\abs[2]{\frac{1}{N}\sum_j u_{ij}d_{jj}}
\prec \Bigl(\frac{\rho+\Xi}{N \eta}\Bigr)^{3/2}
\end{equation}
according to \eqref{Bound on error matrix}. By taking the off-diagonal part of \eqref{G-M expansion} and using the fact that $M$ and $V_\mathrm{r}$ and therefore also $\BO^{-1}\cQ[M\SS[V_\mathrm{r}]V_\mathrm{r}]$ are diagonal (cf.~\eqref{prop eigendirections BQ} above) we have 
\[\abs{\mathcal{B}^{-1}\mathcal{Q}[MD_\mathrm{o}]+G_\mathrm{o}} \prec \theta^3+\theta\Bigl(\frac{\rho+\Xi}{N \eta}\Bigr)^{1/2}+\frac{\rho+\Xi}{N \eta}\lesssim N^{-\epsilon}\theta^2+N^{\epsilon}\frac{\rho+\Xi}{N \eta}\]
for any $\epsilon$ such that $\theta\lesssim N^{-\epsilon}$ by Young's inequality in the last step.
Together with \eqref{Expanding M and Vl}, \eqref{Bound on error matrix} and the assumption that $\abs{G_\mathrm{o}}=\abs{(G-M)_\mathrm{o}}\prec \Lambda$ we then compute
\begin{equation*}
\begin{split}
&\braket{V_\mathrm{l},M(\mathcal{S}\mathcal{B}^{-1}\mathcal{Q}[MD_\mathrm{o}])(\mathcal{B}^{-1}\mathcal{Q}[MD_\mathrm{o}])} \\
& = c \braket{\diag(\vp\vf) (\mathcal{S}\mathcal{B}^{-1}\mathcal{Q}[MD_\mathrm{o}])(\mathcal{B}^{-1}\mathcal{Q}[MD_\mathrm{o}])} + \landauO{ \Bigl(\rho+\frac{\eta}{\rho}\Bigr)\frac{\rho+\Xi}{N\eta} }\\
& = c\braket{\diag(\vp \vf)\mathcal{S}[G_\mathrm{o}]G_\mathrm{o}} + \landauO{ \Bigl(\rho+\frac{\eta}{\rho}\Bigr)\frac{\rho+\Xi}{N\eta} + \Bigl(\Bigl(\frac{\rho+\Xi}{N\eta}\Bigr)^{1/2}+\Lambda\Bigr)\Bigl[N^{-\epsilon}\theta^2+N^{\epsilon}\frac{\rho+\Xi}{N \eta}\Bigr]}.
\end{split}
\end{equation*}
Thus the bound on the second term on the lhs.~in \eqref{cusp FA} follows together with \eqref{DD splitting}--\eqref{DD diag bound} by $\mathcal{S}[G_\mathrm{o}] = T \odot G^t$ and \eqref{TGG bound} via Lemma \ref{prec p conversion}. This completes the proof of \eqref{cusp FA}.
\end{proof}
With \eqref{Bound on error matrix} and \eqref{cusp FA} the upper bound \eqref{bound on eps} on the error $\epsilon_\ast$ of the cubic equation \eqref{Theta cubic equation} takes the same form as the rhs.~of \eqref{cusp FA} if $K$ is sufficiently large depending on $\epsilon$. By the first estimate in \eqref{first G-M bounds} we can redefine the control parameter $\Lambda$ on $\abs{G-M}$ as $\Lambda\defeq \theta +((\rho+\Xi)/N \eta)^{1/2}$ and the claim \eqref{bound on cubic in Theta} follows directly with \eqref{cusp FA}, thus completing the proof of Lemma~\ref{lmm:Bounds on errors}. 
\end{proof}

\subsection{Bootstrapping}
Now we will show that the difference $G-M$ converges to zero uniformly for all spectral parameters $z \in \DD_\zeta$ as defined in \eqref{D zeta def}. For convenience we refer to existing bounds on $G-M$ far away from the real line to establish a rough bound on $G-M$ in, say, $\DD_1$. We then iteratively lower the threshold on $\eta$ by appealing to Proposition \ref{prp:Cubic equation for Theta} and Lemma \ref{lmm:Bounds on errors} until we establish the rough bound in all of $\DD_\zeta$. As a second step we then improve the rough bound iteratively until we obtain Theorem \ref{thr:Local law}.

\begin{lemma}[Rough bound] For any $\zeta>0$ there exists a constant $c>0$ such that on the domain $\DD_\zeta$
we have the rough bound 
\begin{equation} \label{Rough bound on G-M} \abs{G-M} \prec N^{-c}. \end{equation}
\end{lemma}
\begin{proof} 
The rough bound \eqref{Rough bound on G-M} in a neighbourhood of a cusp has first been established for Wigner-type random matrices in \cite{MR3719056}. For the convenience of the reader we present a streamlined proof that is adapted to the current setting.  
The lemma is an immediate  consequence of the following statement.
Let $\zeta_\mathrm{s}>0$ be a sufficiently small \emph{step size}, depending on $\zeta$. Then  for any $\N_0\ni k\le 1/\zeta_\mathrm{s}$ on the domain $\DD_{\max\{1-k\zeta_\mathrm{s}, \zeta\}}$ we have
\begin{equation} \label{Rough bound} \abs{G-M}\prec N^{-4^{-k}\zeta}. \end{equation}
We prove \eqref{Rough bound} by induction over $k$. For sufficiently small $\zeta$ the induction start $k=0$ holds due to the local law away from the self-consistent spectrum, e.g.~\cite[Theorem 2.1]{MR3941370}. 

Now as induction hypothesis suppose that \eqref{Rough bound} holds on $\wt{\DD}_{k} \defeq  \DD_{\max\{1-k\zeta_\mathrm{s}, \zeta\}}$,  and in particular, $\abs{G}\prec 1$, $\norm{G}_p\le_{\epsilon,p}N^{\epsilon}$ for any $\epsilon,p$ according to Lemma \ref{prec p conversion}. The monotonicity of the function $\eta \mapsto \eta \norm{G(\tau+\ii\eta)}_p$ (see e.g.~\cite[proof of Prop.~5.5]{MR3941370}) implies $\norm{G}_p\le_{\epsilon,p} N^{\epsilon+\zeta_\mathrm{s}}\le N^{2 \zeta_\mathrm{s}}$ and therefore, according to Lemma \ref{prec p conversion}, that $\abs{G} \prec N^{2\zeta_\mathrm{s}}$ on $\wt{\DD}_{k+1}$. This, in turn, implies $\abs{D}\prec N^{-\zeta/3}$ on $\wt{\DD}_{k+1}$ by \eqref{isotropic bound on D} and Lemma \ref{prec p conversion}, provided $\zeta_\mathrm{s}$ is chosen small enough. We now fix $\vx,\vy$ and a large integer $K$ as the parameters of $\norm{\cdot}_\ast=\norm{\cdot}_\ast^{\vx,\vy,K}$ for the rest of the proof and omit them
from the notation but we stress that all estimates will be uniform in $\vx,\vy$. We find $\sup_{z \in \widetilde\DD_{k+1}}\norm{D(z)}_\ast\prec N^{-\zeta/3}$, by using a simple union bound and $\norm{\partial_z D}\le N^C$ for some $C>0$. Thus, for $K$ large enough, we can use \eqref{G-M expansion}, \eqref{Theta cubic equation}, \eqref{bound on eps} and \eqref{bounds on star norm} to infer
\begin{equation} \label{rough bound on cubic} 
\begin{split}
  \abs{\Theta^3+ \xi_2 \Theta^2 +\xi_1 \Theta} &\lesssim N^{1/2K}\norm{D}_\ast\prec N^{1/2K-\zeta/3},\\
\norm{G-M}_\ast&\lesssim  \abs{\Theta} + N^{1/K}\norm{D}_\ast\prec \abs{\Theta}+N^{1/K-\zeta/3} ,
\end{split}
 \end{equation}
on the event $\norm{G-M}_\ast +\norm{D}_\ast \lesssim N^{-10/K}$, and on $\wt{\DD}_{k+1}$. Now we use the following lemma \cite[Lemma 10.3]{1804.07752} to translate the first estimate in  \eqref{rough bound on cubic} into a bound on $\abs{\Theta}$. For the rest of the proof we keep $\tau=\Re z$ fixed and consider the coefficients $\xi_1,\xi_2$ and $\Theta$ as functions of $\eta$.
\begin{lemma}[Bootstrapping cubic inequality] 
\label{lmm:bootstrap cubic}
For $0<\eta_\ast<\eta^\ast<\infty$ let $\xi_1,\xi_2\colon[\eta_\ast,\eta^\ast] \to \C$ be complex valued functions and $\wt{\xi}_1,\wt{\xi}_2, d\colon[\eta_\ast,\eta^\ast]  \to \R^+ $ be continuous functions such that at least one of the following holds true:
\begin{enumerate}[(i)]
\item $\abs{{\xi}_1}\sim \wt{\xi}_1$, $\abs{{\xi}_2}\sim \wt{\xi}_2$, and $\wt{\xi}_2^3/d,\wt{\xi}_1^3/d^2,\wt{\xi}_1^2/d\wt{\xi}_2$ are monotonically increasing, and $d^2/\wt{\xi}_1^3+d\wt{\xi}_2/\wt{\xi}_1^2\ll 1$ at $\eta^\ast$,
\item \label{Bootstrapping 2} $\abs{{\xi}_1}\sim \wt{\xi}_1$, $\abs{{\xi}_2}\lesssim \wt{\xi}_1^{1/2}$, and $\wt{\xi}_1^3/d^2$ is monotonically increasing. 
\end{enumerate}
Then any continuous function $\Theta\colon [\eta_\ast,\eta^\ast]  \to \C$ that satisfies the cubic inequality $\abs[0]{\Theta^3 + {\xi}_2 \Theta^2  + {\xi}_1 \Theta} \lesssim d$ on $[\eta_\ast,\eta^\ast]$, has the property
\begin{equation}\label{cubic bootstrapping eq}
\text{If}  \quad \abs{\Theta}  \lesssim \min\bigg\{d^{1/3}, \frac{d^{1/2}}{{\wt{\xi}_2^{1/2}}},\frac{d}{\wt{\xi}_1}\bigg\}  \text{ at } \eta^\ast,
\quad\text{then}
\quad \abs{\Theta}  \lesssim \min\bigg\{d^{1/3}, \frac{d^{1/2}}{{\wt{\xi}_2^{1/2}}},\frac{d}{\wt{\xi}_1}\bigg\} \text{ on } [\eta_\ast,\eta^\ast].
\end{equation}
\end{lemma}
With direct arithmetics we can now verify that the coefficients $\xi_1,\xi_2$ in \eqref{Theta cubic equation} and the auxiliary coefficients $\wt\xi_1,\wt\xi_2$ defined in \eqref{xi tilde def} satisfy the assumptions in Lemma~\ref{lmm:bootstrap cubic} with the choice of the constant function $d=N^{-4^{-k}\zeta+\delta}$ for any $\delta>0$, by using only the information on $\xi_1,\xi_2$ given by the comparison relations \eqref{xi comparison relations}. As an example, in the regime where $\tau_0$ is a right edge and $\omega \sim \Delta$, we have $\wt{\xi}_1 \sim (\eta +\Delta)^{2/3}$ and $\wt{\xi}_2 \sim (\eta +\Delta)^{1/3}$ and both functions are monotonically increasing in $\eta$. Then Assumption \eqref{Bootstrapping 2} of Lemma~\ref{lmm:bootstrap cubic} is satisfied. All other regimes are handled similarly.

We now set $\eta^\ast\defeq N^{-k\zeta_\mathrm{s}}$ and \[\eta_\ast\defeq \inf \Set{\eta \in [N^{-(k+1) \zeta_\mathrm{s}},\eta^\ast]|\sup_{\eta'\ge \eta}\norm{G(\tau+\ii\eta')-M(\tau+\ii\eta')}_\ast \le N^{-10/K}/2}.\]
By the induction hypothesis we have $\abs{\Theta(\eta^\ast)} \lesssim d \lesssim \min\{ d^{1/3}, d^{1/2}\wt\xi_2^{-1/2},d \wt\xi_1^{-1}\}$ with overwhelming probability, so that the condition in \eqref{cubic bootstrapping eq} holds, and conclude $\abs{\Theta(\eta)} \prec d^{1/3}=N^{-(4^{-k}\zeta-\delta)/3}$ for $\eta \in [\eta_\ast,\eta^\ast]$.    For small enough $\delta>0$ the second bound in \eqref{rough bound on cubic} implies $\norm{G-M}_\ast \prec N^{-4^{k+1}\zeta}$.  By continuity and the definition of $\eta_\ast$ we conclude $\eta_\ast =N^{-(k+1) \zeta_\mathrm{s}}$,  finishing the proof of \eqref{Rough bound}. 
\end{proof}
\begin{proof}[Proof of Theorem~\ref{thr:Local law}]
The bounds within the proof hold true uniformly for $z\in\DD_\zeta$, unless explicitly specified otherwise. We therefore suppress this qualifier in the following statements. First we apply Lemma~\ref{lmm:Bounds on errors} with the choice $\Xi=\Lambda$, i.e.\ we do not treat the imaginary part of the resolvent separately. With this choice the first inequality in \eqref{first G-M bounds} becomes self-improving and after iteration shows that
\begin{equation} \label{G-M bound inside} \abs{G-M} \prec \theta + \sqrt{\frac{\rho}{N \eta}}+ \frac{1}{N \eta}, \end{equation}
and, in other words, \eqref{bound on cubic in Theta} holds with $\Xi =\theta + (\rho/N\eta)^{1/2}+ 1/N\eta$. This implies that if $\abs{\Theta}\prec\theta\lesssim N^{-c}$ for some arbitrarily small $c>0$, then 
\begin{equation} \label{Theta self-improving bound 2} \abs{\Theta^3+ \xi_2 \Theta^2 +\xi_1 \Theta} \lesssim   N^{{5}\wt{\epsilon}}d_\ast
+N^{-\wt{\epsilon}}(\theta^3 + \wt{\xi}_2\theta^2)\end{equation}
holds for all sufficiently small $\wt\epsilon$ with overwhelming probability, where we defined  
\begin{equation}\label{d ast}
d_\ast\defeq\wt{\xi}_2 \bigg(\frac{\wt \rho }{N \eta} +\frac{1}{(N \eta)^2}\bigg)  +\frac{1}{(N \eta)^3} + \bigg( \frac{\wt \rho}{N \eta}\bigg)^{3/2} .
\end{equation}
For this conclusion we used the comparison relations \eqref{xi comparison relations}, Proposition \ref{prop mde}\eqref{prop mde wt xi} as well as \eqref{rho tilde}, and the bound $\sqrt{\eta/\rho}\sim\sqrt{\eta/\wt\rho}\lesssim \wt\xi_2$. 

The bound \eqref{Theta self-improving bound 2} is a self-improving estimate on $\abs{\Theta}$ in the following sense.  For $k \in \N$ and $l \in \N \cup\{\ast\}$ let
\[
d_k \defeq \max\{N^{-k\wt{\epsilon}},N^{6\wt{\epsilon}}d_\ast\}, \qquad \theta_l\defeq  \min\bigg\{d_l^{1/3}, \frac{d_l^{1/2}}{\wt{\xi}_2^{1/2}},\frac{d_l}{\wt{\xi}_1}\bigg\}.
\]
Then \eqref{Theta self-improving bound 2} with $\abs{\Theta} \prec \theta_k$ implies that $\abs{\Theta^3+ \xi_2 \Theta^2 +\xi_1 \Theta}\lesssim N^{-\wt{\epsilon}} d_k$.
Applying Lemma~\ref{lmm:bootstrap cubic} with $d=N^{-\wt\epsilon}{d_k}$, $\eta^\ast\sim1$, $\eta_\ast=N^{\zeta-1}$ yields the improvement $\abs{\Theta} \prec \theta_{k+1}$. Here we needed to check the condition in \eqref{cubic bootstrapping eq} but at $\eta^\ast\sim 1$ we have $\wt\xi_1\sim1$, so $\abs{\Theta}\lesssim N^{-\wt\epsilon}d_k\le d_{k+1}\sim\theta_{k+1}$. After a $k$-step iteration until $N^{-k\wt\epsilon}$ becomes smaller than $N^{6\wt\epsilon}d_\ast$, we find
 $\abs{\Theta} \prec \theta_\ast$, 
where we used that $\wt{\epsilon}$ can be chosen arbitrarily small. We are now ready to prove the following bound which we, for convenience, record as a proposition.  
\begin{proposition}\label{local law uniform} For any $\zeta>0$ we have the bounds
\begin{equation} \label{Final bound} \abs{G-M}\prec \theta_\ast+ \sqrt{\frac{\rho}{N \eta}}+ \frac{1}{N \eta},\qquad \abs{G-M}_{\mathrm{av}}\prec \theta_\ast+\frac{\rho}{N \eta}+\frac{1}{(N \eta)^2} \quad\text{in}\quad\DD_\zeta,\end{equation}
where $\theta_\ast\defeq\min\{d_\ast^{1/3},d_\ast^{1/2}/\wt\xi_2^{1/2},d_\ast/\wt\xi_1\}$, and $d_\ast,\wt\rho,\wt\xi_1,\wt\xi_2$ are given in \eqref{d ast}, \eqref{rho tilde} and \eqref{xi tilde def}, respectively.
\end{proposition}
\begin{proof}
Using $\abs{\Theta}\prec\theta_\ast$ proven above, we apply \eqref{G-M bound inside} with $\theta=\theta_\ast$ to conclude the first inequality in \eqref{Final bound}. For the second inequality in \eqref{Final bound} we use the estimate on $\abs{G-M}_{\mathrm{av}}$ from \eqref{first G-M bounds} with $\theta = \theta_\ast$ and $\Xi= (\rho/N\eta)^{1/2}+ 1/N\eta$.
\end{proof}

The bound on $\abs{G-M}$ from \eqref{Final bound} implies a complete delocalisation of eigenvectors uniformly at singularities of the scDOS. The following corollary was established already in \cite[Corollary 1.14]{MR3719056} and, given \eqref{Final bound}, the proof follows the same line of reasoning.
\begin{corollary}[Eigenvector delocalisation] \label{crl:Eigenvector delocalisation}
Let $\vu \in \C^N$ be an eigenvector of $H$ corresponding to an eigenvalue $\lambda \in \tau_0+(-c,c)$ for some sufficiently small positive constant $c \sim 1$. Then for any deterministic $\vx \in \C^N$ we have
\[
\abs{\braket{\vu, \vx}} \prec \frac{1}{\sqrt{N}} \norm{\vu}\norm{\vx}.
\]
\end{corollary}

The bounds \eqref{Final bound} simplify in the regime $\eta \ge N^\zeta \eta_\mathrm{f}$ above the typical eigenvalue spacing to
 \begin{equation} \label{G-M bound inside2} \abs{G-M}\prec \sqrt{\frac{\rho}{N \eta}}+ \frac{1}{N \eta},\qquad \abs{G-M}_{\mathrm{av}}\prec \frac{1}{N \eta},
\qquad \text{for} \quad \eta \ge N^\zeta \eta_\mathrm{f} \end{equation}
using Lemma \ref{lemma tilde xi} which implies $\theta_\ast \le d_\ast/\wt{\xi}_1\le 1/N\eta$. The bound on $\abs{G-M}_{\mathrm{av}}$ is further improved in the case when $\tau_0=\ed_-$ is an edge and, in addition to $\eta\ge N^\zeta \eta_\mathrm{f}$, we assume $N^{\delta}\eta\le \omega\le \Delta/2$ for some $\delta>0$, i.e.\ if $\omega$ is well inside a gap of size $\Delta\ge N^{\delta+\zeta}\eta_\mathrm{f}$. Then we find $\Delta>N^{-3/4}$ by the definition of $\eta_\mathrm{f}=\Delta^{1/9}/N^{2/3}$ in \eqref{eta* in gap} and use Lemma \ref{lemma tilde xi} and \eqref{rho tilde}, \eqref{xi tilde def} to conclude
\begin{equation} \label{last estimate} \theta_\ast+\frac{\wt{\rho}}{N \eta}+\frac{1}{(N \eta)^2}\lesssim \frac{\wt{\xi}_2}{\wt{\xi}_1} \bigg(\frac{\wt{\rho}}{N \eta}+ \frac{1}{(N \eta)^2}\bigg)\sim
 \frac{\Delta^{1/6}}{\omega^{1/2}}\bigg(\frac{\eta}{\Delta^{1/6}\omega^{1/2}}+\frac{1}{N \eta}\bigg)\frac{1}{N \eta}\lesssim 
 \frac{N^{-\delta/2}}{N \eta} . \end{equation}
In the last bound we used $1/N\omega\le N^{-\delta}/N\eta$ and $\Delta^{1/6}/(N\eta\omega^{1/2})\le N^{-\delta/2}$. Using \eqref{last estimate} in \eqref{Final bound} yields the improvement 
 \begin{equation} \label{G-M outside}  \abs{G-M}_{\mathrm{av}}\prec \frac{N^{-\delta/2}}{N \eta},
\qquad \text{for} \quad \tau=\ed_-+\omega,\quad \Delta/2\ge\omega \ge N^{\delta} \eta \ge N^{\zeta+\delta}\eta_\mathrm{f}. \end{equation}

The bounds on $\abs{G-M}_{\mathrm{av}}$ from \eqref{G-M bound inside2} and \eqref{G-M outside}, inside and outside the self-consistent spectrum, allow us to show the uniform rigidity, Corollary~\ref{crl:Uniform rigidity}. We postpone these arguments until after we finish the proof of Theorem~\ref{thr:Local law}. The uniform rigidity implies that for $\dist(z, \supp \rho) \ge N^{\zeta}\eta_\mathrm{f}$ we can estimate the imaginary part of the resolvent via
\begin{equation} \label{bound on im G} \Im \braket{\vx, G\vx} = \sum_{\lambda}\frac{\eta\abs{\braket{\vu_\lambda, \vx}}^2}{\eta^2 + (\tau_0+\omega-\lambda)^2}
\prec \eta + \frac{1}{N}\sum_{\abs{\lambda-\tau_0}\le c}\frac{\eta}{\eta^2 + (\tau_0+\omega-\lambda)^2}\prec \rho(z), \end{equation}
for any normalised $\vx \in \C^{N}$, where $\vu_\lambda$ denotes the normalised eigenvector corresponding to $\lambda$. For the first inequality in \eqref{bound on im G} we used Corollary~\ref{crl:Eigenvector delocalisation} and for the second we applied  Corollary~\ref{crl:Uniform rigidity} that allows us to replace the Riemann sum with an integral as $[\eta^2+(\tau_0+\omega-\lambda)^2]^{1/2}=\abs{z-\lambda}\ge N^\zeta\eta_\mathrm{f}$.

Using with \eqref{bound on im G}, we apply Lemma~\ref{lmm:Bounds on errors}, repeating the strategy from the beginning of the proof. But this time we can choose the control parameter $\Xi=\rho$. In this way we find 
\begin{equation}\label{G-M outside2}
\abs{G-M} \prec \theta_{\#} + \sqrt{\frac{\rho}{N \eta}},\qquad  \abs{G-M}_{\mathrm{av}}\prec \theta_{\#}+\frac{\rho}{N \eta}, \qquad\text{for} \quad \dist(z, \supp \rho) \ge N^{\zeta}\eta_\mathrm{f},
\end{equation}
where we defined 
\[
\theta_{\#} \defeq\min\bigg\{\frac{d_\#}{\wt{\xi}_1}, \frac{d_\#^{1/2}}{\wt{\xi}_2^{1/2}}, d_\#^{1/3}\bigg\}, \qquad  d_{\#}\defeq \wt{\xi}_2\frac{\wt \rho }{N \eta}  + \bigg( \frac{\wt \rho}{N \eta}\bigg)^{3/2} .
\]
Note that the estimates in \eqref{G-M outside2} are simpler than those in \eqref{Final bound}. The reason is that the additional terms $1/N\eta$, $1/(N\eta)^2$ and $1/(N\eta)^3$ in \eqref{Final bound} are a consequence of the presence of $\Xi$ in \eqref{bound on cubic in Theta}, \eqref{first G-M bounds}. With $\Xi=\rho$ these are 
immediately absorbed into $\rho$ and not present any more.
The second term in the definition of $d_\#$ can be dropped since we still have $\wt\xi_2\gtrsim (\rho/N\eta)^{1/2}$ (this follows from Lemma \ref{lemma tilde xi} if $\eta\ge N^\zeta\eta_\mathrm{f}$, and directly from \eqref{rho tilde}, \eqref{xi tilde def} if $\omega\ge N^\zeta\eta_\mathrm{f}$). This implies $\theta_\#\lesssim d_\#^{1/2}/\wt\xi_2^{1/2}\lesssim (\rho/N\eta)^{1/2}$, so the first bound in \eqref{G-M outside2} proves \eqref{local law inside spectrum}. 

Now we turn to the proof of \eqref{average local law inside spectrum}. Given the second bound in \eqref{G-M bound inside2}, it is sufficient to consider the case when $\tau=\ed_-+\omega$ and $\eta\le\omega\le\Delta/2$ with $\omega\ge N^\zeta \eta_\mathrm{f}$. In this case Proposition \ref{prop mde} yields $\wt\xi_2\wt\rho/\wt\xi_1+\wt\rho\lesssim \eta/\omega\sim \eta/\dist(z,\supp\rho)$. Thus we have 
\[ \theta_\#+\frac{\rho}{N\eta}\lesssim \frac{d_\#}{\wt\xi_1}+\frac{\wt\rho}{N\eta}\lesssim \frac{1}{N\dist(z,\supp\rho)} \]
and therefore the second bound in \eqref{G-M outside2} implies \eqref{average local law inside spectrum}.
 This completes the proof of Theorem \ref{thr:Local law}.
\end{proof}

\subsection{Rigidity and absence of eigenvalues}
The proofs of Corollaries~\ref{crl:Uniform rigidity} and \ref{cor no eigenvalues outside} rely on the bounds on $\abs{G-M}_{\mathrm{av}}$ from \eqref{G-M bound inside2} and \eqref{G-M outside}. As before, we may restrict ourselves to the neighbourhood of a local minimum $\tau_0 \in \supp \rho$ of the scDOS which is either an internal minimum with a small value of $\rho(\tau_0)>0$, a cusp location or a right edge adjacent to a small gap of length $\Delta>0$. All other cases, namely the bulk regime and regular edges adjacent to large gaps, have been treated prior to this work \cite{MR3719056,1804.07744}. 

\begin{proof}[Proof of Corollary \ref{cor no eigenvalues outside}]
Let us denote the empirical eigenvalue distribution of $H$ by $\rho_H = \frac{1}{N} \sum_{i=1}^N \delta_{\lambda_i}$ and consider the case when $\tau_0=\ed_-$ is a right edge, $\Delta \ge N^{\delta} \eta_\mathrm{f}$ for any $\delta>0$ and $\eta_\mathrm{f}=\eta_\mathrm{f}(\ed_-)\sim \Delta^{1/9}N^{-2/3}$. Then we show that there are no eigenvalues in $\ed_-+[N^{\delta}\eta_\mathrm{f}, \Delta/2]$  with overwhelming probability. 
We apply \cite[Lemma~5.1]{MR3719056} with the choices 
\[
\nu_1\defeq\rho,\quad \nu_2\defeq\rho_H, \quad \eta_1\defeq\eta_2\defeq\epsilon\defeq N^{\zeta}\eta_\mathrm{f},\quad \tau_1\defeq\ed_-+\omega,\quad \tau_2 \defeq \ed_-+\omega +N^{\zeta}\eta_\mathrm{f},
\]
for any $\omega \in [N^{\delta}\eta_\mathrm{f}, \Delta/2]$ and some $\zeta \in (0,\delta/4)$. We use \eqref{G-M outside} to estimate the error terms $J_1, J_2$ and $J_3$ from \cite[Eq.~(5.2)]{MR3719056} by $N^{2\zeta-\delta/2-1}$ and  see that $(\rho_H-\rho)([\tau_1,\tau_2]) =\rho_H([\tau_1,\tau_2]) \prec N^{2\zeta-\delta/2-1}$, showing that with overwhelming probability the interval $[\tau_1,\tau_2]$ does not contain any eigenvalues.  A simple union bound finishes the proof of Corollary~\ref{cor no eigenvalues outside}.
\end{proof}
\begin{proof}[Proof of Corollary~\ref{crl:Uniform rigidity}]
Now we establish Corollary~\ref{crl:Uniform rigidity} around a local minimum $\tau_0 \in \supp \rho$ of the scDOS. Its proof has two ingredients. First we follow the strategy of the proof of \cite[Corollary~1.10]{MR3719056} to see that 
\begin{equation}
\label{IDOS difference}
\abs{(\rho-\rho_H)((-\infty,\tau_0+\omega]) } \prec \frac{1}{N},
\end{equation}
for any $\abs{\omega} \le c $, i.e.~we have a very precise control on $\rho_H$.
In contrast to the statement in that corollary we have a local law \eqref{G-M bound inside2} with uniform $1/N\eta$ error and thus the bound \eqref{IDOS difference} does not deteriorate close to $\tau_0$. 
We warn the reader that the standard argument inside the proof of \cite[Corollary~1.10]{MR3719056} has to be adjusted slightly to arrive at \eqref{IDOS difference}. In fact, when inside that proof the auxiliary result \cite[Lemma~5.1]{MR3719056} is used with the choice $\tau_1=-10$, $\tau_2 =\tau$, $\eta_1=\eta_2=N^{\zeta-1}$ for some $\zeta>0$, this choice should be changed to $\tau_1=-C$, $\tau_2 =\tau$, $\eta_1=N^{\zeta-1}$ and $\eta_2=N^{\zeta}\eta_{\mathrm{f}}(\tau)$, where $C>0$ is chosen sufficiently large such that $\tau_1$ lies far to the left of the self-consistent spectrum. 

The control \eqref{IDOS difference} suffices to prove Corollary~\ref{crl:Uniform rigidity} for all $\tau=\tau_0 +\omega$ except for the case when $\tau_0=\ed_-$ is an edge at a gap of length $\Delta \ge N^\zeta \eta_\mathrm{f}$ and $\omega \in [- N^\zeta \eta_\mathrm{f},0]$ for some fixed $\zeta>0$ and $\eta_\mathrm{f} = \eta_\mathrm{f}(\ed_-) \sim \Delta^{1/9}/N^{2/3}$, i.e.~except for some $N^\zeta$ eigenvalues close to the edge with arbitrarily small $\zeta>0$. In all other cases, the proof follows the same argument as the proof of \cite[Corollary~1.11]{MR3719056} using the uniform $1/N$-bound from \eqref{IDOS difference} and we omit the details here. 

The reason for having to treat the eigenvalues very close to the edge $\ed_-$ separately is that \eqref{IDOS difference} does not give information on which side of the gap these $N^\zeta$ eigenvalues are found. To get this information requires the second ingredient, the \emph{band rigidity},
\begin{equation}\label{band rigidity}
\P\big[ \rho((-\infty, \ed_-+\omega])= \rho_H((-\infty, \ed_-+\omega]) \big]\ge 1-N^{-\nu},
\end{equation}
for any $\nu \in \N$, $\Delta\ge\omega \ge N^\zeta \eta_\mathrm{f}$ and large enough $N$.  The combination of \eqref{band rigidity} and \eqref{IDOS difference} finishes the proof of Corollary~\ref{crl:Uniform rigidity}.

Band rigidity has been shown in case $\Delta$ is bounded from below in \cite{1804.07744} as part of the proof of Corollary~2.5. We will now adapt this proof to the case of small gap sizes $\Delta \ge N^{\zeta-3/4}$. Since by Corollary~\ref{cor no eigenvalues outside} with overwhelming probability there are no eigenvalues in $\ed_-+[N^\zeta \eta_\mathrm{f},\Delta/2]$, it suffices to show \eqref{band rigidity} for $\omega = \Delta/2$. As in the proof of \cite[Corollary 2.5]{1804.07744} we consider the interpolation
\[
H_t\defeq\sqrt{1-t}W+A-t\mathcal{S}M(\tau),\qquad t \in [0,1],
\]
between the original random matrix $H=H_0$ and the deterministic matrix $H_1=A-\mathcal{S}M(\tau)$, for $\tau=\ed_- +\Delta/2$. The interpolation is designed such that the solution $M_t$ of the MDE corresponding to $H_t$ is constant at spectral parameter $\tau$, i.e.~$M_t(\tau)=M(\tau)$. Let $\rho_t$ denote the scDOS of $H_t$. Exactly as in the proof from \cite{1804.07744} it suffices to show that no eigenvalue crosses the gap along the interpolation with overwhelming probability, i.e.~that for any $\nu \in \N$ we have 
\begin{equation}
\label{No ev crossing}
\P\big[ \mathfrak{a}_t \in \Spec(H_t) \text{ for some } t \in [0,1]\big] \le \frac{C(\nu)}{N^{\nu}}.
\end{equation}
Here $t \to \mathfrak{a}_t \in \R\setminus \supp \rho_t$ is some spectral parameter inside the gap, continuously depending on $t$, such that $\mathfrak{a}_0 =\tau$. In \cite{1804.07744} $\mathfrak{a}_t$ was chosen independent of $t$, but the argument remains valid with any other choice of $\mathfrak{a}_t$. We call $I_t$ the connected component of $\R\setminus \supp \rho_t$ that contains $\mathfrak{a}_t$ and denote $\Delta_t = \abs{I_t}$ the gap length. In particular, $\Delta_0=\Delta$ and $\tau \in I_t$ for all $t \in[0,1]$ by \cite[Lemma~8.1(ii)]{1804.07752}.  
For concreteness we choose $\mathfrak{a}_t$ to be the spectral parameter lying exactly in the middle of $I_t$. The $1/3$-H\"older continuity of $\rho_t$, hence $I_t$ and $\mathfrak{a}_t$ in $t$ follows from \cite[Proposition 10.1(a)]{1804.07752}. Via a simple union bound it suffices to show that for any fixed $t \in [0,1]$ we have no eigenvalue in $\mathfrak{a}_t+[-N^{-100},N^{-100}]$. 

Since $\norm{W} \lesssim 1$ with overwhelming probability, in the regime $t \ge 1-\epsilon$ for some small constant $\epsilon>0$, the matrix $H_t$ is a small perturbation of the deterministic matrix $H_1$ whose resolvent $(H_1-\tau)^{-1} =M(\tau)$ at spectral parameter $\tau$ is bounded by Assumption \ref{bdd m}, in particular $\Delta_1\gtrsim 1$. By $1/3$-H\"older continuity hence $\Delta_t\gtrsim 1$, and $\Spec(H_t) \subset \Spec(H_1)+[-C\epsilon^{1/3},C\epsilon^{1/3}]$ for some $C\sim 1$ in this regime with very high probability. Since $\Spec(H_1) \subset \supp \rho_t+[-C\epsilon^{1/3},C\epsilon^{1/3}]$ by \cite[Proposition 10.1(a)]{1804.07752} there are no eigenvalues of $H_t$ in a neighbourhood of $\mathfrak{a}_t$, proving \eqref{No ev crossing} for $t\ge 1-\epsilon$.

For $t \in [\epsilon, 1-\epsilon]$ we will now show that $\Delta_t \sim_\epsilon 1$ for any $\epsilon >0$. 
In fact, we have $\dist(\tau, \supp \rho_t)~\gtrsim_\epsilon~1$.  This is a consequence of \cite[Lemma~D.1]{1804.07752}. More precisely, we use the equivalence of (iii) and (v) of that lemma. We check (iii) and conclude the uniform distance to the self-consistent spectrum by (v). Since $M_t(\tau)=M(\tau)$ and $\norm{M(\tau)} \lesssim 1$ we only need to check that the stability operator $\mathcal{B}_t = t+(1-t)\mathcal{B}$ of $H_t$ has a bounded inverse. We write $\mathcal{B}_t = \mathcal{C}(1-(1-t) \wt{\mathcal{C}}\mathcal{F})\mathcal{C}^{-1}$ in terms of the saturated self-energy operator $\mathcal{F} = \mathcal{C}\mathcal{S}\mathcal{C}$, where $\mathcal{C}[R]\defeq\abs{M(\tau)}^{1/2}R\abs{M(\tau)}^{1/2}$ and $\wt{\mathcal{C}}[R]\defeq(\sgn M(\tau))R(\sgn M(\tau))$. Afterwards we use that $\norm{\mathcal{F}}_{\mathrm{hs} \to \mathrm{hs}}\le 1$ (cf.~\cite[Eq.~(4.24)]{MR3916109}) and $\norm[0]{\wt{\mathcal{C}}}_{\mathrm{hs}\to\mathrm{hs}}=1$ to first show the uniform bound $\norm{\mathcal{B}_t}_{\mathrm{hs} \to \mathrm{hs}} \lesssim 1/t$ and then improve the bound to $\norm{\mathcal{B}_t} \lesssim 1/t$ using the trick of expanding in a geometric series from \cite[Eqs.~(4.60)--(4.63)]{MR3916109}. This completes the argument that $\Delta_t\sim_\epsilon 1$. Now we apply \cite[Corollary~2.3]{MR3941370} to see that there are no eigenvalues of $H_t$ around $\mathfrak{a}_t$ as long as $t$ is bounded away from zero and one, proving \eqref{No ev crossing} for this regime.

Finally, we are left with the regime $t \in [0,\epsilon]$ for some sufficiently small $\epsilon>0$.  By \cite[Proposition 10.1(a)]{1804.07752} the self-consistent Green's function $M_t$ corresponding to $H_t$ is bounded even in a neighbourhood of $\tau$, whose size only depends on model parameters. In particular, Assumptions \ref{bdd moments}--\ref{bdd m} are satisfied for $H_t$ and Corollary~\ref{cor no eigenvalues outside}, which was already proved above, is applicable. Thus it suffices to show that the size $\Delta_t$ of the gap in $\supp \rho_t$ containing $\tau$ is bounded from below by $\Delta_t \ge N^{\zeta-3/4}$ for some $\zeta >0$. The size of the gap can be read off from the following relationship between the norm  of the saturated self-energy operator and the size of the gap: Let $H$ be a random matrix satisfying \ref{bdd moments}--\ref{bdd m} and $\tau$ be well inside the interior of the gap of length $\Delta \in [0,c]$ in the self-consistent spectrum for a sufficiently small $c\sim 1$. Then 
\begin{equation}
\label{F norm gap size}
1-\norm{\mathcal{F}(\tau)}_{\mathrm{hs} \to \mathrm{hs}} \sim \lim_{\eta \searrow 0}\frac{\eta}{\rho(\tau +\ii \eta )} \sim (\Delta+\dist(\tau,\supp\rho))^{1/6}\dist(\tau,\supp\rho)^{1/2} \sim \Delta^{2/3},
\end{equation}
where in the first step we used \cite[Eqs.~(4.23)--(4.25)]{MR3916109}, in the second step \eqref{rho tilde}, and in the last step that $\dist(\tau,\supp\rho)\sim\Delta$. Applying the analogue of \eqref{F norm gap size} for $H_t$ with $\mathcal{F}_t(\tau)$ and using that $\dist(\tau,\rho_t)\lesssim \Delta_t$, we obtain \smash{$1-\norm{\mathcal{F}_t(\tau)}_{\mathrm{hs}\to\mathrm{hs}}\lesssim \Delta_t^{2/3}$}. Combining this inequality with \eqref{F norm gap size} and using that $\mathcal{F}_t(\tau)=(1-t)F(\tau)$ for $t\in[0,c]$, we have \smash{$\Delta_t^{3/2} \gtrsim t +(1-t)\Delta^{2/3}$}, i.e.~$\Delta_t\gtrsim t^{3/2}+\Delta$. In particular, the gap size $\Delta_t$ never drops below $c\Delta\gtrsim N^{\zeta-3/4}$. This completes the proof of the last regime in \eqref{No ev crossing}. 
\end{proof}

\section{Cusp fluctuation averaging and proof of Theorem \ref{thm pfD bound moments}}
\label{sec:Cusp fluctuation averaging}
We will use the graphical multivariate cumulant expansion from~\cite{MR3941370}
 which automatically exploits the self-energy renormalization of \(D\) to highest order. Since the final formal statement 
 requires some custom notations, we first give a simple motivating 
 example to illustrate the type of expansion and its graphical representation.
 If \(W\) is Gaussian, then integration by parts shows that 
\begin{equation}\label{D2 example}
\begin{split}
\E \braket{D}^2 &= \sum_{\alpha,\beta}\kappa(\alpha,\beta)\E \braket{\Delta^{\alpha} G} \braket{\Delta^{\beta} G} \\
&+ \sum_{\alpha_1,\beta_1} \kappa(\alpha_1,\beta_1)\sum_{\alpha_2,\beta_2} \kappa(\alpha_2,\beta_2) \E  \braket{\Delta^{\alpha_1} G\Delta^{\beta_2} G} \braket{\Delta^{\alpha_2} G\Delta^{\beta_1}G},
\end{split}
\end{equation}
where we recall that $\kappa(\alpha, \beta)\defeq \kappa(w_\alpha,w_\beta)$ is the second cumulant of the matrix entries $w_\alpha,w_\beta$ index by double indices \(\alpha=(a,b)\), \(\beta=(a',b')\), and $\Delta^{(a,b)}$ denotes the matrix of all zeros except for an $1$ in the $(a,b)$-th entry. Since for non-Gaussian \(W\) or higher powers of \(\braket{D}\) the expansion analogous to \eqref{D2 example} 
consists of much more complicated polynomials in resolvent entries, we represent them  concisely
as the \emph{values} of certain \emph{graphs}. 
As an example, the rhs.\ of~\eqref{D2 example} is  represented simply by
\begin{equation}\label{eq graphs example}\Val\left(\plotLambda{{1},{1}}\right) + \Val\left(\plotLambda{{1,2},{2,1}}\right). \end{equation}
The graphs retain only the relevant information of the complicated expansion terms and 
 chains of estimates can be transcribed into simple
graph surgeries. Graphs also help identify critical terms that 
have to be estimated more precisely in order to obtain the improved high moment bound on \(D\). 
For example, the key cancellation mechanism behind the cusp fluctuation averaging is encoded in 
a small distinguished part of the expansion that can conveniently be identified as certain subgraphs,
called the \emph{$\sigma$-cells}, see Definition~\ref{sigma cell def} later.
 It is easy to count, estimate and manipulate $\sigma$-cells as part of a large graph, while 
following the same operations on the level of formulas would be almost intractible.  

First we review some of the basic nomenclature from \cite{MR3941370}. We consider random matrices $H=A+W$ with diagonal expectation $A$ and complex Hermitian or real symmetric zero mean random component $W$ indexed by some abstract set $J$ of size $\abs{J}=N$. We recall that Greek letters $\alpha,\beta,\dots$ stand for labels, i.e.\ double-indices from $I=J\times J$, whereas Roman letters $a,b,\dots$ stand for single indices. If $\alpha =(a,b)$, then we set $\alpha^t\defeq (b,a)$ for its transpose.
Underlined Greek letters stand for multisets of labels, whereas bold-faced Greek letters stand for tuples of labels with the counting combinatorics being their -- for our purposes -- only relevant difference. 

According to \cite[Proposition 4.4]{MR3941370} with $\NN(\alpha)=\{\alpha,\alpha^t\}$ it follows from the assumed independence 
that for general (conjugate) linear functionals $\Lambda^{(k)}$, of bounded norm $\norm[0]{\Lambda^{(k)}}=\landauO{1}$ 
\begin{subequations}\label{E prod lambda}
\begin{equation} 
\E \prod_{k\in[p]}\Lambda^{(k)}(D) = \E\prod_{l\in[p]}\bigg(1+\sum_{\alpha_l,\bm\beta_l}^{\sim(l)}\bigg)\prod_{k\in [p]} \begin{cases} \Lambda^{(k)}_{\alpha_k,\underline\beta^k}&\text{if $\sum_{\alpha_k}$}\\
\Lambda^{(k)}_{\underline\beta^k_{<k},\underline\beta^k_{>k}}&\text{else}
\end{cases} + \landauO{N^{-p}},
 \end{equation}
where we recall that
\begin{equation} \label{tilde sum def}
\sum_{\alpha_l,\bm\beta_l}^{\sim(l)} \defeq \sum_{\alpha_l\in I}\sum_{1\le m < 6p} \sum_{\bm\beta_l\in\{\alpha_l,\alpha_l^t\}^m}\frac{\kappa(\alpha_l,\underline\beta_l)}{m!} \sum_{\underline\beta_l^1\sqcup\dots\sqcup\underline\beta_l^p=\underline\beta_l}\1(\abs[0]{\underline\beta_l^l}=0\text{ if }\abs[0]{\underline\beta_l}=1)
 \end{equation}
and that
\begin{equation}\begin{split}\label{Lambda index def}
\Lambda_{\alpha_1,\dots,\alpha_k}&\defeq-(-1)^k\Lambda(\Delta^{\alpha_{1}} G\dots \Delta^{\alpha_{k}} G),\quad \Lambda_{\{\alpha_1,\dots,\alpha_m\}}\defeq\sum_{\sigma\in S_m}\Lambda_{\alpha_{\sigma(1)},\dots,\alpha_{\sigma(m)}}, \\
 \Lambda_{\alpha,\{\alpha_1,\dots,\alpha_m\}}&\defeq\sum_{\sigma\in S_m}\Lambda_{\alpha,\alpha_{\sigma(1)},\dots,\alpha_{\sigma(m)}},\quad \Lambda_{\underline\alpha,\underline\beta}\defeq \sum_{\alpha\in\underline\alpha}\Lambda_{\alpha,\underline\alpha\cup\underline\beta\setminus\{\alpha\}},\\ 
\underline\beta_{<k}^k&\defeq\bigsqcup_{j<k} \underline\beta_j^k, \quad\underline\beta_{>k}^k\defeq\bigsqcup_{j>k} \underline\beta_j^k.\end{split}\end{equation}
\end{subequations}
Some notations in \eqref{E prod lambda} require further explanation. The qualifier ``if $\sum_{\alpha_k}$'' is satisfied for those terms in which $\alpha_k$ is a summation variable when the brackets in the product $\prod_j(1+ \sum)$ are opened. The notation $\bigsqcup$ indicates the union of multisets. 

For even $p$ we apply \eqref{E prod lambda} with $\Lambda^{(k)}(D) \defeq \braket{\diag(\vf\vp)D}$ for $k\le p/2$ and $\Lambda^{(k)}(D) \defeq \overline{\braket{\diag(\vf\vp)D}}$ for $k> p/2$. This is obviously a special case of $\Lambda^{(k)}(D)=\braket{BD}$ which was considered in the so-called averaged case of \cite{MR3941370} with arbitrary $B$ of bounded operator norm since $\norm{\diag(\vf\vp)}=\norm{\vf\vp}_\infty\le C$. It was proved in \cite{MR3941370} that 
\[ \abs{\braket{\diag(\vf\vp) D}}\lesssim \frac{\rho}{N\eta}, \] 
which is not good enough at the cusp. We can nevertheless use the graphical language developed in \cite{MR3941370} to estimate the complicated right hand side of \eqref{E prod lambda}. 

\subsection{Graphical representation via double index graphs}
The graphs (or Feynman diagrams) introduced in \cite{MR3941370} encode the structure of all terms in \eqref{E prod lambda}. 
Their (directed) edges correspond to resolvents $G$, while vertices correspond to $\Delta$'s. Loop edges are allowed while parallel edges are not. Resolvents $G$ and their Hermitian conjugates $G^*$ are distinguished by 
different types of edges. Each vertex $v$ carries a label $\alpha_v$
 and we need to
sum up for all labels.  Some labels are independently summed up, these are the $\alpha$-labels in  \eqref{E prod lambda}, while the $\beta$-labels
are strongly restricted; in the independent case they can only be of the type $\alpha$ or $\alpha^t$. 
These graphs will be called ``double indexed'' graphs since the vertices are naturally equipped with labels (double indices). Here we introduced the terminology 
``double indexed'' for the graphs in \cite{MR3941370} to distinguish them from the ``single indexed'' 
graphs to be introduced later in this paper.  
 
To be more precise, 
 the graphs in \cite{MR3941370} were vertex-coloured  graphs. The colours encoded a resummation of the terms in \eqref{E prod lambda}: vertices whose labels (or their transpose) appeared in one of the cumulants in  \eqref{E prod lambda} received the same colour.  
 We then first summed up the colours and only afterwards we summed up all labels compatible with the given colouring. According to~\cite[Proposition~4.4]{MR3941370} and the expansion of the main term~\cite[Eq.~(49)]{MR3941370} for every even $p$ it holds that 
\begin{subequations}\label{double vertex graph expansion}
\begin{equation}\label{E prod double index graphs}
\E \abs{\braket{\diag(\vf\vp)D}}^p = \sum_{\Gamma\in\cG^{\text{av}(p,6p)} }\Val(\Gamma)+\landauO{N^{-p}},
\end{equation}
where $\cG^{\text{av}(p,6p)}$ is a certain finite collection of vertex coloured directed graphs with $p$ connected 
components, and $\Val(\Gamma)$, the value of the graph $\Gamma$, will be recalled below. According to \cite{MR3941370} each graph $\Gamma\in \cG^{\text{av}(p,6p)}$ fulfils the following properties:
\begin{proposition}[Properties of double index graphs]\label{facts double index graphs} There exists a finite set $\cG^{\text{av}(p,6p)}$ of double index graphs $\Gamma$ such that \eqref{double vertex graph expansion} hold. Each $\Gamma$ fulfils the following properties.
\begin{enumerate}[(a)]
\item There exist exactly $p$ connected components, all of which are oriented cycles. Each vertex has one incoming and one outgoing edge.
\item Each connected component contains at least one vertex and one edge. Single vertices with a looped edge are in particular legal connected components. 
\item Each colour colours at least two and at most $6p$ vertices. 
\item If a colour colours exactly two vertices, then these vertices are in different connected components. 
\item The edges represent the resolvent matrix $G$ or its adjoint $G^\ast$. Within each component either all edges represent $G$ or all edges represent $G^\ast$. Accordingly we call the components either  $G$ or $G^\ast$-cycles.
\item Within each cycle there is one designated edge which is represented as a wiggled line in the graph. The designated edge represents the matrix $G\diag(\vp\vf)$ in a $G$-cycle and the matrix $\diag(\vp\vf)G^\ast$ in a $G^\ast$-cycle.
\item For each colour there exists at least one component in which a vertex of that colour is connected to the matrix $\diag(\vf\vp)$. According to (f) this means that if the relevant vertex is in a $G$-cycle, then the designated (wiggled) edge is its incoming edge. If the relevant vertex is in a $G$-cycle, then the designated edge is its outgoing edge.
\end{enumerate}
\end{proposition}
If $V$ is the vertex set of $\Gamma$ and for each colour $c\in C$, $V_c$ denotes the $c$-coloured vertices then we recall that 
\begin{equation}\label{double index val}
\begin{split}
  \Val(\Gamma) &= (-1)^{\abs{V}} \Big(\prod_{c\in C} \prod_{v\in V_c} \sum_{\alpha_v} \frac{\kappa(\{\alpha_v\}_{v\in V_c})}{(\abs{V_c}-1)!}\Big) \\
  &\qquad\times\E\prod_{\text{Cyc}(v_1,\dots,v_k)\in \Gamma} \begin{cases}
    \braket{G\diag(\vf\vp)\Delta^{\alpha_{v_1}}G\dots G\Delta^{\alpha_{v_k}}} \\
    \braket{\Delta^{\alpha_{v_k}}G^\ast \dots G^\ast\Delta^{\alpha_{v_1}} \diag(\vf\vp)G^\ast}
    \end{cases}
\end{split}
\end{equation}
\end{subequations}
where the ultimate product is the product over all $p$ of the cycles in the graph. By the notation $\text{Cyc}(v_1,\dots,v_k)$ we indicate
a directed cycle with vertices $v_1,  \ldots, v_k$.  Depending upon whether a given cycle is a $G$-cycle or $G^\ast$-cycle,
it then contributes with one of the factors indicated after the last curly bracket in \eqref{double index val} with the vertex order chosen in such a way that the designated edge represents the $G\diag(\vf\vp)$ or $\diag(\vf\vp)G^\ast$ matrix. As an example illustrating \eqref{double index val} we have 
\begin{align}\label{double vertex val example}
&\Val\biggl(\plotLambda{{1,2},{1,2}}\biggr) \\\nonumber 
& = \sum_{\substack{\alpha_1,\beta_1\\\alpha_2,\beta_2}}\kappa(\alpha_1,\beta_1)\kappa(\alpha_2,\beta_2)\E\braket{G\diag(\vf\vp)\Delta^{\alpha_1}G\Delta^{\beta_2}}\braket{\Delta^{\beta_1}G^\ast \Delta^{\alpha_2}\diag(\vf\vp)G^\ast}. \end{align} 
Actually in \cite{MR3941370} the graphical representation of the graph $\Gamma$ is simplified, it does not contain all information 
encoded in the graph. First, the direction of the edges are not indicated. In the picture both cycles should be oriented in a clockwise orientation. Secondly, the type of edges are not indicated, apart from the wiggled line. In fact, the edges in the second subgraph
 stand for $G^\ast$, while those in the first subgraph stand for $G$. To translate the pictorial representation directly let the striped vertices in the first and second cycle be associated with $\alpha_1,\beta_1$ and the dotted vertices with $\alpha_2,\beta_2$. Accordingly, the wiggled edge in the first cycle stands for $G\diag(\vf\vp)$, while the wiggled edge in the second cycle stands for $\diag(\vf\vp)G^\ast$. The reason why these details were omitted in the graphical representation of a double index graph is that they do not influence the basic power counting estimate of its value used in \cite{MR3941370}.  

\subsection{Single index graphs}\label{sec:single}
In \cite{MR3941370} we operated with double index graphs that are structurally simple and 
appropriate for bookkeeping complicated correlation structures, but they 
are not suitable for detecting the additional smallness we need at the cusp. 
The contribution of the graphs in \cite{MR3941370} were estimated by a relatively simple power counting 
argument where only the number of (typically off-diagonal) resolvent elements were recorded. 
In fact, for many subleading graphs this procedure already gave a very good bound that is sufficient 
at the cusps as well. The graphs carrying the leading contribution, however, have now to be computed 
to a higher accuracy and this leads to the concept of ``single index graphs''. These are obtained by a certain 
refinement and reorganization of the double index graphs via a procedure we will call 
\emph{graph resolution} to be defined later. The main idea is to restructure the double index graph in such a way that 
instead of labels (double indices) $\alpha=(a,b)$ its vertices naturally represent single indices $a$ and $b$. Every double indexed graph will give rise to a finite 
number of resolved single index graphs. The double index graphs that require a more precise 
analysis compared with \cite{MR3941370} will be resolved to single index graphs. 
After we explain the structure of the single index graphs and the graph resolution procedure, 
double index graphs will not be used in this paper any more. Thus, unless explicitly stated otherwise, 
by graph we will mean single index graph in the rest of this paper. 

We now define the set $\cG$ of single index graphs we will use in this paper. They are directed graphs, where parallel edges and loops are allowed. Let the graph be denoted by $\Gamma$ with vertex set $V(\Gamma)$ and edge set $E(\Gamma)$. We will assign a value to each $\Gamma$ which comprises weights assigned to the vertices and specific values  assigned to the edges. Since an edge may represent different objects, we will introduce different types of edges that will be graphically distinguished by different line style. We now describe these ingredients precisely.
\subsubsection*{Vertices.} Each vertex $v\in V(\Gamma)$ is equipped with an associated index $a_v\in J$. Graphically the vertices are represented by small unlabelled bullets $\sGraph{a}$, i.e.\ in the graphical representation the actual index is not indicated. It is understood that all indices will be independently summed up
over the entire index set $J$
when we compute the value of the graph.
\subsubsection*{Vertex weights.} Each vertex $v\in V(\Gamma)$ carries some weight vector $\vw^{(v)}\in\C^J$ which is evaluated $\vw^{(v)}_{a_v}$ at the index $a_v$ associated with the vertex. We generally assume these weights to be uniformly  bounded in  $N$, i.e.\ $\sup_N\norm[0]{\vw^{(v)}}_\infty<\infty$. Visually we indicate vertex weights by incoming arrows as in $\sGraph{a[inl=$\vw$]}$. Vertices without explicitly indicated weight may carry an arbitrary bounded weight vector. We also use the notation $\sGraph{a[l1]}$ to indicate the constant $\bm 1$ vector as the weight, this corresponds to summing up the corresponding index unweighted
\subsubsection*{$G$-edges.} The set of $G$-edges is denoted by $\GE(\Gamma)\subset E(\Gamma)$. 
These edges describe resolvents and  
there are four types of $G$-edges. First of all, there are directed edges corresponding to $G$ and $G^\ast$ in the sense that a directed $G$ or $G^\ast$-edge $e=(v,u)\in E$  initiating from the vertex $v=i(e)$ and terminating in the vertex $u=t(e)$ represents the matrix elements $G_{a_va_u}$ or respectively $G^\ast_{a_va_u}$ evaluated in the indices $a_v,a_u$ associated with the vertices $v$ and $u$. Besides these two there are also edges representing $G-M$ and $(G-M)^\ast$. Distinguishing between $G$ and $G-M$, for practical purposes, is only important if it occurs in a loop. Indeed, $(G-M)_{aa}$ is typically much smaller than $G_{aa}$, while
 $(G-M)_{ab}$ basically acts just like $G_{ab}$ when $a,b$ are summed independently. Graphically we will denote the four types of $G$-edges by  
\[ 
G=\sGraph{a -- b},\quad G^\ast=\sGraph{a --[s] b},\quad G-M=\sGraph{a --[gm] b},\quad G^\ast-M^\ast=\sGraph{ a --[s,gm] b; }
\]
where all these edges can also be loops. 
The convention is that continuous lines represent $G$, dashed lines 
correspond to $G^*$, while the diamond on both types of edges indicates the subtraction of $M$ or $M^\ast$.  An edge $e\in \GE(\Gamma)$ carries
its type as its attribute, so 
as a short hand notation we can simply write $G_e$ for $G_{a_{i(e)},a_{t(e)}}$, $G^\ast_{a_{i(e)},a_{t(e)}}$, $(G-M)_{a_{i(e)},a_{t(e)}}$ and $(G-M)^\ast_{a_{i(e)},a_{t(e)}}$ depending on which type of $G$-edge $e$ represents. Due to their special role in the later estimates, we will separately bookkeep those $G-M$ or $G^\ast-M^\ast$ edges that appear looped.
We thus define the subset $\GE_{g-m}\subset \GE$ as the set of $G$-edges $e\in \GE(\Gamma)$ of type $G-M$ or $G^\ast-M^\ast$ such that
 $i(e)=t(e)$. We write $g-m$ to refer to the fact that looped edges are evaluated on the diagonal $(g-m)_{a_v}$ of $(G-M)_{a_va_v}$. 
\subsubsection*{($G$-)edge degree.}
For any vertex $v$ we define its in-degree $\deg^-(v)$ and out-degree $\deg^+(v)$ as the number of incoming and outgoing $G$-edges. Looped edges $(v,v)$ are counted for both in- and out-degree. We denote the total degree by $\deg(v)=\deg^-(v)+\deg^+(v)$. 
\subsubsection*{Interaction edges.} Besides the $G$-edges we also have interaction edges, $\IE(\Gamma)$, representing the cumulants $\kappa$. A directed interaction edge $e=(u,v)$ represents the matrix $R^{(e)}=\big(r_{ab}^{(e)}\big)_{a,b\in J}$ given by the cumulant 
\begin{equation}\label{R kappa rel}
r_{ab}^{(u,v)} = \frac{1}{(\deg(u)-1)!}\kappa( \underbrace{ab,\dots,ab}_{\text{$\deg^-(u)$ times}}, \underbrace{ba,\dots,ba}_{\text{$\deg^+(u)$ times}} )= \frac{1}{(\deg(v)-1)!}\kappa( \underbrace{ab,\dots,ab}_{\text{$\deg^+(v)$ times}}, \underbrace{ba,\dots,ba}_{\text{$\deg^-(v)$ times}} ).
\end{equation}
For all graphs \(\Gamma\in\cG\) and all interaction edges \(e=(u,v)\) we have the symmetries \(\deg^-(u)=\deg^+(v)\) and \(\deg^-(v)=\deg^+(u)\). Thus~\eqref{R kappa rel} is  compatible with exchanging the roles of $u$ and $v$. For the important case when $\deg(u)=\deg(v)=2$ it follows that the interaction from $u$ to $v$ is given by $S$ if $u$ has one incoming and one outgoing $G$-edge and $T$ if $u$ has two incoming $G$-edges, i.e.
\[
s_{ab} = \kappa(ab,ba)\qquad t_{ab}= \kappa(ab,ab).
\]
Visually we will represent interaction edges as
\[
R=\sGraph{ a --[R] b; } \qquad\text{and more specifically by}\qquad S=\sGraph{ a --[S] b; }, \quad T=\sGraph{ a --[T] b; }.
 \]
Although the interaction matrix $R^{(e)}$ is completely determined by the in- and out-degrees of the adjacent vertices $i(e),t(e)$ we still write out the specific $S$ and $T$ names because these will play a special role in the latter part of the proof. As a short hand notation we shall frequently use $R_e\defeq R^{(e)}_{a_{i(e)},a_{t(e)}}$ to denote the matrix element selected by the indices $a_{i(e)},a_{t(e)}$ associated with the initial and terminal vertex of $e$. We also note that we do not indicate the direction of edges associated with $S$ as the matrix $S$ is symmetric.
\subsubsection*{Generic weighted edges.} Besides the specific  $G$-edges and interaction edges, additionally we also allow for generic edges reminiscent of the generic vertex weights introduced above. They will be called \emph{generic weighted edges}, or \emph{weighted edges} for short. To every weighted edge $e$ we assign a weight matrix $K^{(e)}=(k^{(e)}_{ab})_{a,b\in J}$ which is evaluated 
as $k^{(e)}_{a_{i(e)},a_{t(e)}}$ when we compute the value of the graph by  summing  up all indices. To simplify the presentation we will not indicate the precise form of the weight matrix $K^{(e)}$ but only its entry-wise scaling as a function of $N$. A weighted edge presented as $\ssGraph{a --[we=$N^{-l}$] b;}$ represents an arbitrary weight matrix $K^{(e)}$ whose entries scale like $\smash{\abs[0]{k^{(e)}_{ab}}\le cN^{-l}}$. We denote the set of weighted edges by $\WeE(\Gamma)$. For a given weighted edge $e\in\WeE$ we record the entry-wise scaling of $K^{(e)}$ in an exponent  $l(e)\ge 0$ in such a way that we always have $\abs[0]{k^{(e)}_{ab}}\le c N^{-l(e)}$.
\subsubsection*{Graph value.} 
For graphs $\Gamma\in\cG$ we define their value
\begin{equation}\label{value def}
  \begin{split}
    \Val(\Gamma)&\defeq (-1)^{\abs{\GE(\Gamma)}} \bigg(\prod_{v\in V(\Gamma)}\sum_{a_v\in J} \vw^{(v)}_{a_v} \bigg) \bigg(\prod_{e\in\IE(\Gamma)} r^{(e)}_{a_{i(e)},a_{t(e)}}\bigg)\bigg(\prod_{e\in\WeE(\Gamma)} k^{(e)}_{a_{i(e)},a_{t(e)}}\bigg)\\
    &\qquad\times \E \bigg(\prod_{e\in\GE(\Gamma)} G_e\bigg),
  \end{split}
\end{equation}
which differs slightly from that in \eqref{double index val} because it applies to a different class of graphs.

\subsection{Single index resolution}
There is a natural mapping from double indexed graphs to a collection of single indexed graphs 
that encodes the rearranging of the terms in \eqref{double index val} when the summation over labels $\alpha_v$ 
is reorganized into summation over single indices. Now we describe this procedure.
\begin{definition}[Single index resolution]\label{def single v res}
By the \emph{single index resolution} of a double vertex graph we mean the collection of single index graphs obtained through the following procedure. 
\begin{enumerate}[(i)] 
\item For each colour, the identically coloured vertices of the double index graph are mapped into a pair 
of vertices of the single index graph. 
\item The pair of vertices in the single index graph stemming from a fixed colour is connected by 
an interaction edge in the single index graph. 
\item Every (directed) edge of the double index graph is naturally mapped to a $G$-edge of the single index graph. 
While mapping equally coloured vertices $x_1,\dots,x_k$ in the double index graph to vertices $u,v$ connected by an interaction edge $e=(u,v)$ there are $k-1$ binary choices of whether we map the incoming edge of $x_j$ to an incoming edge of $u$ and the outgoing edge of $x_j$ to an outgoing edge of $v$ or vice versa. In this process we are free to consider the mapping of $x_1$ (or any other vertex, for that matter) as fixed by symmetry of $u\leftrightarrow v$.
\item If a wiggled $G$-edge is mapped to an edge from $u$ to $v$, then $v$ is equipped with a weight of $\vp\vf$. If a wiggled $G^\ast$-edge is mapped to an edge from $u$ to $v$, then $u$ is equipped with a weight of $\vp\vf$. All vertices with no weight specified in this process are equipped with the constant weight $\bm 1$.
\end{enumerate}
We define the set $\cG(p)\subset \cG$ as the set of all graphs obtained from the double index graphs $\cG^{\text{av}(p,6p)}$ 
via the single index resolution procedure.
\end{definition}

\begin{remark}\label{remark:single}~
\begin{enumerate}[(i)]
\item We note some ingredients described in Section \ref{sec:single} for a typical graph in $\cG$
will be absent for  graphs $\Gamma\in\cG(p)\subset\cG$. 
 For example, $\WeE(\Gamma)=\GE_{g-m}(\Gamma)=\emptyset$ for all $\Gamma\in \cG(p)$.
\item We also remark that loops in double index graphs are never  mapped into loops in single index graphs
along the single index resolution. 
 Indeed, double index loops are always mapped to edges parallel to the interaction edge of the corresponding vertex.
\end{enumerate}
\end{remark}

A few simple facts immediately follow from the the single index  construction in Definition \ref{def single v res}. From (i) it is clear that the number of vertices in the single index graph is twice the number of colours of the double index graph. From (ii) it follows that the number of interaction edges in the single index graph equals the number of colours of the double index graph. Finally, from (iii) it is obvious that if for some colour $c$ there are $k=k(c)$ vertices in the double index graph with colour $c$, then the resolution of this colour gives rise to $2^{k(c)-1}$ single indexed graph. Since these resolutions are done independently for each colour, we obtain that the number of single index 
graphs originating from one double index graph is 
\[ 
\prod_c 2^{k(c)-1} 
 \]
Since the number of double index graph in $\cG^{\text{av}(p,6p)}$ is finite, so is the number of graphs in $\cG(p)$.

Let us present an example of single index resolution applied to the graph from \eqref{double vertex val example} where we, for the sake of transparency, label all vertices and edges. $\Gamma$ is a graph consisting of one 2-cycle on the vertices $x_1,y_2$ and one 2-cycle on the vertices $x_2,y_1$ as in
\begin{equation}\label{labeled edges example}
\plotLambdas{"$x_1$"[fill=gray!50] ->[bl,edge node={node[above=-1pt] { $e_1$ }}] "$y_2$" ->[bend left,edge node={node[above=-11pt] { $e_2$ }}] "$x_1$"; "$y_1$"[fill=gray!50] ->[bl,edge node={node[above=-1pt] { $e_3$ }}] "$x_2$" ->[bend left,edge node={node[above=-11pt] { $e_4$ }}] "$y_1$";  }
\end{equation}
with $x_1,y_1$ and $x_2,y_2$ being of equal colour (i.e.\ being associated to labels connected through cumulants). In order to explain steps (i)-(iii) of the construction we first neglect that some edges may be wiggled, but we restore the orientation of the edges in the picture. We then fix the mapping of $x_i$ to pairs of vertices $(u_i,v_i)$ for $i=1,2$ in such a way that the incoming edges of $x_i$ are incoming at $u_i$ and the outgoing edges from $x_i$ are outgoing from $v_i$. It remains to map $y_i$ to $(u_i,v_i)$ and for each $i$ there are two choices of doing so that we obtain the four possibilities 
\[
\begin{split}
  &\Biggl\{ 
  \plotLambdas{ "$v_1$ \nodepart{lower} $u_1$"[circle split, draw,fill=gray!50] ->[bl,edge node={node[above=-1pt] { $e_1$ }}] "$u_2$ \nodepart{lower} $v_2$"[circle split, draw] ->[bend left,edge node={node[above=-11pt] { $e_2$ }}] "$v_1$ \nodepart{lower} $u_1$"[circle split, draw,fill=gray];} 
  \plotLambdas{ "$v_1$ \nodepart{lower} $u_1$"[circle split, draw,fill=gray!50] ->[bl,edge node={node[above=-1pt] { $e_3$ }}] "$u_2$ \nodepart{lower} $v_2$"[circle split, draw] ->[bend left,edge node={node[above=-11pt] { $e_4$ }}] "$v_1$ \nodepart{lower} $u_1$"[circle split, draw,fill=gray];  },\quad
  \plotLambdas{ "$v_1$ \nodepart{lower} $u_1$"[circle split, draw,fill=gray!50] ->[bl,edge node={node[above=-1pt] { $e_1$ }}] "$v_2$ \nodepart{lower} $u_2$"[circle split, draw] ->[bend left,edge node={node[above=-11pt] { $e_2$ }}] "$v_1$ \nodepart{lower} $u_1$"[circle split, draw,fill=gray];}  
  \plotLambdas{ "$v_1$ \nodepart{lower} $u_1$"[circle split, draw,fill=gray!50] ->[bl,edge node={node[above=-1pt] { $e_3$ }}] "$u_2$ \nodepart{lower} $v_2$"[circle split, draw] ->[bend left,edge node={node[above=-11pt] { $e_4$ }}] "$v_1$ \nodepart{lower} $u_1$"[circle split, draw,fill=gray];  },\\
  &\qquad \plotLambdas{ "$v_1$ \nodepart{lower} $u_1$"[circle split, draw,fill=gray!50] ->[bl,edge node={node[above=-1pt] { $e_1$ }}] "$u_2$ \nodepart{lower} $v_2$"[circle split, draw] ->[bend left,edge node={node[above=-11pt] { $e_2$ }}] "$v_1$ \nodepart{lower} $u_1$"[circle split, draw,fill=gray];}  
  \plotLambdas{ "$u_1$ \nodepart{lower} $v_1$"[circle split, draw,fill=gray!50] ->[bl,edge node={node[above=-1pt] { $e_3$ }}] "$u_2$ \nodepart{lower} $v_2$"[circle split, draw] ->[bend left,edge node={node[above=-11pt] { $e_4$ }}] "$u_1$ \nodepart{lower} $v_1$"[circle split, draw,fill=gray];  },\quad \plotLambdas{ "$v_1$ \nodepart{lower} $u_1$"[circle split, draw,fill=gray!50] ->[bl,edge node={node[above=-1pt] { $e_1$ }}] "$v_2$ \nodepart{lower} $u_2$"[circle split, draw] ->[bend left,edge node={node[above=-11pt] { $e_2$ }}] "$v_1$ \nodepart{lower} $u_1$"[circle split, draw,fill=gray];}  
  \plotLambdas{ "$u_1$ \nodepart{lower} $v_1$"[circle split, draw,fill=gray!50] ->[bl,edge node={node[above=-1pt] { $e_3$ }}] "$u_2$ \nodepart{lower} $v_2$"[circle split, draw] ->[bend left,edge node={node[above=-11pt] { $e_4$ }}] "$u_1$ \nodepart{lower} $v_1$"[circle split, draw,fill=gray];  }
  \Biggr\},
\end{split}
\]
which translates to 
\begin{equation}\label{single vertex labeled examples}
\Set{
\sGraph{u1[label = $u_1$] --[T] v1[label = $v_1$]; u2[label=$u_2$] --[T] v2[label=$v_2$]; v1 --[bl,draw=gray,edge node={node {$e_1$}}] u2; v2 --[bl,draw=gray,edge node={node {$e_2$}}] u1; v2 --[br,draw=gray,edge node={node {$e_4$}}] u1; v1 --[br,draw=gray,edge node={node {$e_3$}}] u2; },
\sGraph{u1[label = $u_1$] --[T] v1[label = $v_1$]; u2[label=$u_2$] --[S] v2[label=$v_2$]; v1 --[bl,draw=gray,edge node={node {$e_1$}}] v2; u2 --[bl,draw=gray,edge node={node {$e_2$}}] u1; v2 --[br,draw=gray,edge node={node {$e_4$}}] u1; v1 --[br,draw=gray,edge node={node {$e_3$}}] u2; }, 
\sGraph{u1[label = $u_1$] --[S] v1[label = $v_1$]; u2[label=$u_2$] --[T] v2[label=$v_2$]; v1 --[br,draw=gray,edge node={node {$e_1$}}] u2; v2 --[br,draw=gray,edge node={node {$e_2$}}] u1; v2 --[br,draw=gray,edge node={node {$e_4$}}] v1; u1 --[br,draw=gray,edge node={node {$e_3$}}] u2; },
\sGraph{u1[label = $u_1$] --[S] v1[label = $v_1$]; u2[label=$u_2$] --[S] v2[label=$v_2$]; v1 --[br,draw=gray,edge node={node {$e_1$}}] v2; u2 --[br,draw=gray,edge node={node {$e_2$}}] u1; v2 --[br,draw=gray,edge node={node {$e_4$}}] v1; u1 --[br,draw=gray,edge node={node {$e_3$}}] u2; }
}  
\end{equation}
in the language of single index graphs where the $S,T$ assignment agrees with \eqref{R kappa rel}. Finally we want to visualize step (iv) in the single index resolution in our example. Suppose that in \eqref{labeled edges example} the edges $e_1$ and $e_2$ are $G$-edges while $e_3$ and $e_4$ are $G^\ast$ edges with $e_2$ and $e_4$ being wiggled (in agreement with \eqref{double vertex val example}). According to (iv) it follows that the terminal vertex of $e_2$ and the initial vertex of $e_4$ are equipped with a weight of $\vp\vf$ while the remaining vertices are equipped with a weight of $\bm 1$. The first graph in \eqref{single vertex labeled examples} would thus be equipped with the weights 
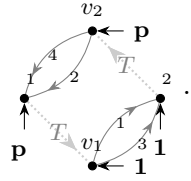
\[
\sGraph{u1[label = $u_1$,bpf] --[T] v1[label = $v_1$,r1]; u2[label=$u_2$,b1] --[T] v2[label=$v_2$,rpf]; v1 --[bl,draw=gray,edge node={node {$e_1$}}] u2; v2 --[bl,draw=gray,edge node={node {$e_2$}}] u1; v2 --[br,draw=gray,edge node={node {$e_4$}}] u1; v1 --[br,draw=gray,edge node={node {$e_3$}}] u2; }.
\]

\subsubsection*{Single index graph expansion.}
With the value definition in \eqref{value def} it follows from Definition \ref{def single v res} that
\begin{equation}\label{pfD val exp}
\E \abs{\braket{\diag(\vf\vp) D}}^p = N^{-p} \sum_{\Gamma\in \cG(p)} \Val(\Gamma) + \landauO{N^{-p}}.
\end{equation}
We note that in contrast to the value definition for double index graphs \eqref{double vertex graph expansion}, where each average in \eqref{double index val} contains
an $1/N$ prefactor, the single index graph value \eqref{value def} does not include the $N^{-p}$ prefactor. 
We chose this convention in this paper   mainly because the exponent $p$ in the
prefactor  $N^{-p}$ cannot be easily read off from the single index graph itself, whereas in the double index graph $p$ is simply the number of connected components.\\

\noindent We now collect some simple facts about the structure of these graphs in $\cG(p)$ which directly follow from the corresponding properties of the double index graphs listed in Proposition \ref{facts double index graphs}.  
\begin{fact}\label{number of edges} The interaction edges $\IE(\Gamma)$ form a perfect matching of $\Gamma$, in particular $\abs{V}=2\abs{\IE}$. Moreover, $1\le \abs{\IE(\Gamma)}\le p$ and therefore the number of vertices in the graph is even and satisfies $2\le \abs{V(\Gamma)}\le 2p$. Finally, since for $(u,v)\in\IE$ we have $\deg^-(u) = \deg^+(v)$  and $\deg^-(v)=\deg^+(u)$, consequently also $\deg(e)\defeq\deg(u)=\deg(v)$. The degree furthermore satisfies the bounds $2\le \deg(e)\le 6p$ for each $e\in\IE(\Gamma)$.
\end{fact}
\begin{fact}\label{fact weights}
The weights associated with the vertices are some non-negative powers of $\vf\vp$ in such a way that the total power of all $\vf\vp$'s is exactly $p$. The trivial zeroth power, i.e.\ the constant
weight $\bm 1$ is allowed.
Furthermore, the $\vf\vp$ weights are distributed in such a way that at least one non-trivial $\vf\vp$ weight is associated with each interacting edge $(u,v)=e\in\IE(\Gamma)$. 
\end{fact}

\subsection{Examples of graphs}\label{sec example graphs}
We now turn to some examples explaining the relation of the double index graphs from \cite{MR3941370} and single index graphs. We note that the single index graphs actually contain more information because they specify edge direction, specify weights explicitly and differentiate between $G$ and $G^\ast$ edges. 
These information were not necessary for the power counting arguments used in \cite{MR3941370}, but for the improved estimates
they will be crucial. 

We start with the graphs representing the following simple equality following from $\kappa(\alpha,\beta)=\E w_\alpha w_\beta$
\[  
\begin{split}
 & N^2\E\sum_{\alpha,\beta} \kappa(\alpha,\beta) \braket{\diag(\vf\vp) \Delta^{\alpha} G} \braket{G^\ast \Delta^\beta \diag(\vf\vp)^\ast} \\
 &\quad = \sum_{a,b} s_{ab} (pf)_a^2 \E G_{ba} G^\ast_{ab} + \sum_{a,b} t_{ab} (pf)_a (pf)_b \E G_{ba} G^\ast_{ba}     
\end{split}\]
which can be represented as
\[  
N^2\Val\left(\plotLambda{{1},{1}}\right)=\Val\left(\sGraph{a[inl=$(\vp\vf)^2$] --[s,bl] b[r1] --[bl] a; a --[S] b;}\right) + \Val\left(\sGraph{a --[T] b; b[rpf] --[bl] a[lpf];  b --[s,br] a;}\right).
 \]

We now turn to the complete graphical representation for the second moment in the case of Gaussian entries, 
\begin{equation}\begin{split}\label{ED2 Gauss}
&\E\abs{\braket{\diag(\vf\vp)D}}^2 = \E \braket{\diag(\vf\vp)D} \braket{D^\ast \diag(\vf\vp)} = \Val\left(\plotLambda{{1},{1}}\right) + \Val\left(\plotLambda{{1,2},{2,1}}\right)\\
&= \sum_{\alpha,\beta} \kappa(\alpha,\beta) \braket{\diag(\vf\vp) \Delta^{\alpha} G} \braket{G^\ast \Delta^\beta \diag(\vf\vp)^\ast} \\
&\qquad+ \sum_{\alpha_1,\beta_1}\sum_{\alpha_2,\beta_2} \kappa(\alpha_1,\beta_1)\kappa(\alpha_2,\beta_2) \braket{\diag(\vf\vp) \Delta^{\alpha_1} G\Delta^{\beta_2} G} \braket{G^\ast \Delta^{\beta_1} G^\ast \Delta^{\alpha_2}\diag(\vf\vp)^\ast},
\end{split}\end{equation}
where we again stress that the double index graphs hide the specific weights and the fact that one of the connected components actually contains $G^\ast$ edges. In terms of single index graphs, the rhs.~of \eqref{ED2 Gauss} can be represented as the sum over the values of the six graphs 
\begin{equation}\begin{split}\label{ED2 graphs} 
N^2 \E \abs{\braket{\diag(\vf\vp)D}}^2&= \Val\left(\sGraph{a[inl=$(\vp\vf)^2$] --[s,bl] b[r1] --[bl] a; a --[S] b;}\right) + \Val\left(\sGraph{a --[T] b; b[rpf] --[bl] a[lpf];  b --[s,br] a;}\right) \\ 
&\quad+\Val\left(\sGraph{ a1[bpf] --[S] b1[r1]; a2[b1] --[S] b2[lpf]; b1 --[bl] b2 --[s,bl] b1; a2 --[bl] a1 --[s,bl] a2;  } \right)+
\Val\left(\sGraph{ a1[bpf] --[T] b1[r1] -- b2[bpf]; a2[l1] --[S] b2; a2 -- a1;  b1 --[s] a2; b2 --[s] a1; } \right)\\
&\quad+
\Val\left(\sGraph{ a1[bpf] --[S] b1[r1] -- a2[b1]; b2[lpf] -- a1 --[s] a2; a2 --[T] b2 --[s] b1;   } \right) + 
\Val\left(\sGraph{ a1[bpf] --[T] b1[r1] --[bl] a2[b1]; a2 --[T] b2[lpf] --[bl] a1; b1 --[s,br] a2; b2 --[s,br] a1;   } \right) \end{split}
\end{equation}
The first two graphs were already explained above. The additional four graphs come from the second term 
in the rhs.~of \eqref{ED2 Gauss}. Since $\kappa(\alpha_1, \beta_1)$ is non-zero only if $\alpha_1=\beta_1$ or
$\alpha_1=\beta_1^t$, there are four possible choices of relations among the $\alpha$ and $\beta$ 
labels in  the two kappa factors. For example, the first graph in the second line of \eqref{ED2 graphs}  corresponds to 
 the choice $\alpha_1^t = \beta_1$, $\alpha_2^t=\beta_2$. Written out explicitly with summation over single
 indices, this value is given by
 \[ 
 \sum_{a_1,b_1}\sum_{a_2,b_2} (pf)_{a_1} (pf)_{b_2} s_{a_1b_1} s_{a_2b_2} \E G_{a_2a_1} G_{b_1 b_2} G^\ast_{a_1a_2} G^\ast_{b_2b_1}
  \]
 where in the picture the left index corresponds to $a_1$, the top index to $b_2$, the right one to $a_2$ and the bottom one to $b_1$.
 
We conclude this section by providing an example of a graph with some degree higher than two which only occurs in the non-Gaussian situation and might contain looped edges. For example, in the expansion of $N^2\E \abs{\braket{\diag(\vf\vp) D}}^2$ in the non-Gaussian setup there is the term
\[ 
\begin{split}
  &\sum_{\substack{a_1,b_1\\a_2,b_2}} r_{a_1b_1} s_{a_2b_2} \E \braket{\diag(\vf\vp) \Delta^{a_1b_1}G \Delta^{b_1a_1}G \Delta^{b_2a_2}G}\braket{G^\ast\Delta^{b_1a_1} G^\ast \Delta^{a_2b_2}\diag(\vf\vp)^\ast} \\
  &\qquad=\Val\left(\sGraph{ a1[bpf] --[R] b1[r1]; b2[bpf] --[s] b1; a1 --[bl,s] a2[r1] --[S] b2; a2 --[bl]  a1; a1 -- b2; b1 --[glb] b1; }\right),
\end{split}
 \]
where $r_{ab}=\kappa(ab,ba,ba)/2$ and $s_{ab}=\kappa(ab,ba)$, in accordance with \eqref{R kappa rel}.

\subsection{Simple Estimates on \texorpdfstring{$\Val(\Gamma)$}{Val(Gamma)}}
In most cases we aim only at estimating the value of a graph instead of precisely computing it. The simplest power counting estimate on
 \eqref{value def} uses that the matrix elements of $G$ and those of the generic weight matrix  $K$ are bounded by an $\landauO{1}$ constant, while the matrix elements of $R^{(e)}$ are bounded by $N^{-\deg(e)/2}$.  Thus the naive estimate on \eqref{value def} is 
\begin{equation}
\abs{\Val(\Gamma)} \lesssim \Big(\prod_{v\in V(\Gamma)} N\Big)\Big(\prod_{e\in\IE(\Gamma)} N^{-\deg(e)/2}\Big) = \prod_{e\in\IE(\Gamma)} N^{2-\deg(e)/2} \le \prod_{e\in\IE(\Gamma)} N\le N^{p} \label{naive estimate} 
\end{equation}
where we used that the interaction edges form a perfect matching and that $\deg(e)\ge 2$, $\abs{\IE(\Gamma)}\le p$. The somewhat informal notation $\lesssim$ in \eqref{naive estimate} hides a technical subtlety. The resolvent entries $G_{ab}$ are indeed bounded by an $\landauO{1}$ constant in the sense of very high moments but not almost surely. We will make bounds like the one in \eqref{naive estimate} rigorous in a high moments sense in Lemma \ref{lemma ward moment}.

The estimate \eqref{naive estimate} ignores the fact that typically only the diagonal resolvent matrix elements of $G$ are of $\landauO{1}$, the off-diagonal matrix elements are much smaller. This is manifested in the \emph{Ward-identity}
\begin{subequations}
\begin{equation}\label{ward identity}
  \sum_{a\in J} \abs{G_{ab}}^2 = (G^\ast G)_{bb} = \frac{(G-G^\ast)_{bb}}{2i\eta}= \frac{\Im G_{bb}}{\eta}.
\end{equation}
Thus the sum of off-diagonal resolvent elements $G_{ab}$ is usually smaller than its naive size of order $N$, at least in the regime
$\eta\gg N^{-1}$. This is quantified by the so called Ward estimates 
\begin{equation}
\sum_{a\in J} \abs{G_{ab}}^2 = N\frac{\Im G_{bb}}{N\eta} \lesssim N \psi^2,\qquad \sum_{a\in J} \abs{G_{ab}} \lesssim N\psi, \qquad \psi\defeq \left(\frac{\rho}{N\eta}\right)^{1/2}.\label{Ward}
\end{equation}
\end{subequations}
Similarly to \eqref{naive estimate} the inequalities $\lesssim$ in 
\eqref{Ward} are meant in a power counting sense ignoring that the entries of $\Im G$ might not be bounded by $\rho$ almost surely but only in some high moment sense.

As a consequence of \eqref{Ward} we can gain a factor of $\psi$ for each off-diagonal (that is, connecting two separate vertices) $G$-factor, but clearly only for at most two $G$-edges per adjacent vertex. Moreover, this gain can obviously only be used once for each edge and not twice, separately when summing up the indices at both adjacent vertices. As a consequence a careful counting of the total number of $\psi$-gains is necessary, see \cite[Section 4.3]{MR3941370} for details. \\

\noindent\textbf{Ward bounds for the example graphs from Section \ref{sec example graphs}.} 
From the single index graphs drawn in \eqref{ED2 graphs} we can easily obtain the known bound $\E\abs{\braket{\diag(\vf\vp)D}}^2\lesssim \psi^4$. Indeed, the last four graphs contribute a combinatorial factor of $N^4$ from the summations over four single indices and a scaling factor of $N^{-2}$ from the size of $S,T$. Furthermore, we can gain a factor of $\psi$ for each $G$-edge through Ward estimates and the bound follows. Similarly, the first two graphs contribute a factor of $N=N^{2-1}$ from summation and $S/T$ and a factor of $\psi^2$ from the Ward estimates, which overall gives $N^{-1}\psi^2\lesssim \psi^4$. As this example shows, the bookkeeping of available Ward-estimates is important and we will do so systematically in the following sections.

\subsection{Improved estimates on \texorpdfstring{$\Val(\Gamma)$}{Val(Gamma)}: Wardable edges}
For the sake of transparency we briefly recall the combinatorial argument used in \cite{MR3941370}, which also provides the starting point for the refined estimate in the present paper. Compared to \cite{MR3941370}, however, we phrase the counting argument directly in the language of the single index graphs. We only aim to gain from the $G$-edges adjacent to vertices of degree two or three; for vertices of higher degree the most naive estimate $\abs{G_{ab}}\lesssim 1$ is already sufficient as demonstrated in  \cite{MR3941370}. We collect the vertices of degree two and three in the set $V_{2,3}$ and collect the $G$-edges adjacent to $V_{2,3}$ in the set $E_{2,3}$. In \cite[Section 4.3]{MR3941370} a specific \emph{marking procedure} on the $G$-edges of the graph is introduced that has the following properties. For each $v\in V_{2,3}$ we  put a \emph{mark}  on at most two adjacent $G$-edges in such a way that those edges can be estimated via \eqref{Ward} while performing the $a_v$ summation. In this case we say that the mark comes from the $v$-perspective. An edge may have two marks coming from the perspective of each of its adjacent vertices. Later, marked edges will be estimated via \eqref{Ward} while summing up $a_v$.  After doing this for all of $V_{2,3}$ we call an edge in $E_{2,3}$ \emph{marked effectively} if it either \emph{(i)} has two marks, or \emph{(ii)} has one mark and is adjacent to only one vertex from $V_{2,3}$. While subsequently using \eqref{Ward} in the summation of $a_v$ for $v\in V_{2,3}$ (in no particular order) on the marked edges (and estimating the remaining edges adjacent to $v$ trivially) we can gain at least as many factors of $\psi$ as there are \emph{effectively marked} edges. Indeed, this follows simply from the fact that \emph{effectively marked} edges are never estimated trivially during the procedure just described, no matter the order of vertex summation. 

\begin{fact}\label{Wardable edges}
For each $\Gamma\in\cG(p)$ there is a marking of edges adjacent to vertices of degree at most $3$ such that there are at least $\sum_{e\in \IE(\Gamma)} (4-\deg(e))_+$ effectively marked edges. 
\end{fact}
\begin{proof}
On the one hand we find from Fact \ref{number of edges} (more specifically, from the equality $\deg(e)=\deg(u)=\deg(v)$ for $(u,v)=e\in\IE(\Gamma)$) that 
\begin{equation}\label{E23 size}
\abs{E_{2,3}}\ge \sum_{v\in V_{2,3}} \frac{1}{2}\deg(v)=\sum_{e\in\IE(\Gamma),\deg(e)\in\{2,3\}} \deg(e).
\end{equation}
On the other hand it can be checked that for every pair $(u,v)=e\in\IE(\Gamma)$
 with $\deg(e)=2$ all $G$-edges adjacent to $u$ or $v$  can be marked from the $u,v$-perspective. Indeed, this is a direct consequence of Proposition \ref{facts double index graphs}(d): Because the two vertices in the double index graph being resolved to $(u,v)$ cannot be part of the same cycle it follows that all of the (two, three or four) $G$-edges adjacent to the vertices with index $u$ or $v$ are not loops (i.e.\ do not represent diagonal resolvent elements). Therefore they can be bounded by using \eqref{Ward}. Similarly, it can be checked that for every edge $(u,v)=e\in\IE(\Gamma)$ with $\deg(e)=3$ at most two  $G$-edges adjacent to  $u$ or $v$ can remain unmarked from the $u,v$-perspective. By combining these two observations it follows that at most 
\begin{equation}\label{unmarked edges}
\sum_{e\in\IE(\Gamma),\deg(e)\in\{2,3\}} (2\deg(e)-4)   
\end{equation}
edges in $E_{2,3}$ are \emph{ineffectively marked} since those are counted as unmarked from the perspective of one of its vertices.
Subtracting \eqref{unmarked edges} from \eqref{E23 size} it follows that in total at least 
\[ 
\sum_{e\in\IE(\Gamma)} (4-\deg(e))_+ = \sum_{e\in\IE(\Gamma),\deg(e)\in\{2,3\}} (4-\deg(e))
 \]
edges are marked effectively, just as claimed.
\end{proof}

In \cite{MR3941370} it was sufficient to estimate the value of each graph in $\cG(p)$ by subsequently estimating all effectively marked edges using \eqref{Ward}. For the purpose of improving the local law at the cusp, however, we need to introduce certain operations on the graphs of $\cG(p)$ which allow to estimate the graph value to a higher accuracy. It is essential that during those operations we keep track of the number of edges we estimate using \eqref{Ward}. Therefore we now introduce a more flexible way of recording these edges. We first recall a basic definition \cite{MR0266812} from graph theory. 

\begin{definition}
For $k\ge 1$ a graph $\Gamma=(V,E)$ is called \emph{$k$-degenerate} if any induced subgraph has minimal degree at most $k$. 
\end{definition}

It is well known that being $k$-degenerate is equivalent to the following sequential property\footnote{This equivalent property is commonly known as having a \emph{colouring number} of at most $k+1$, see e.g.~\ \cite{MR0193025}.}. We provide a short proof for convenience.

\begin{lemma}\label{lemma equiv coloring deg}
A graph $\Gamma=(V,E)$ is $k$-degenerate if and only if there exists an ordering of vertices $\{v_1,\dots,v_n\}=V$ such that for each $m\in[n]\defeq\{1,\dots,n\}$ it holds that 
\begin{equation}\deg_{\Gamma[\{v_1,\dots,v_m\}]}(v_m)\le k\label{vertex ordering}\end{equation}
where for $V'\subset V$, $\Gamma[V']$ denotes the induced subgraph on the vertex set $V'$. 
\end{lemma}
\begin{proof}
Suppose the graph is $k$-degenerate and let $n\defeq \abs{V}$. Then there exists some vertex $v_n\in V$ such that $\deg(v_n)\le k$ by definition. We now consider the subgraph induced by $V'\defeq V\setminus\{v_n\}$ and, by definition, again find some vertex $v_{n-1}\in V'$ of degree $\deg_{\Gamma[V']}(v_{n-1})\le k$. Continuing inductively we find a vertex ordering with the desired property. 

Conversely, assume there exists a vertex ordering such that \eqref{vertex ordering} holds for each $m$. Let $V'\subset V$ be an arbitrary subset and let $m\defeq \max\Set{l\in[n]|v_l\in V'}$. Then it holds that 
\[ 
\deg_{\Gamma[V']}(v_m)\le \deg_{\Gamma[\{v_1,\dots,v_m\}]}(v_m)\le k
 \]
and the proof is complete.
\end{proof}

The reason for introducing this graph theoretical notion is that it is equivalent to the possibility of estimating edges effectively using \eqref{Ward}. A subset $\GE'$ of $G$-edges in $\Gamma\in\cG$ can be fully estimated using \eqref{Ward} if and only if there exists a vertex ordering such that we can subsequently remove vertices in such a way that in each step at most two edges from $\GE'$ are removed. Due to Lemma \ref{lemma equiv coloring deg} this is the case if and only if $\Gamma'=(V,\GE')$ is $2$-degenerate. For example, the graph $\Gamma_{\text{eff}}=(V,\GE_{\text{eff}})$ induced by the effectively marked $G$-edges $\GE_{\text{eff}}$ is a $2$-degenerate graph. Indeed, each effectively marked edge is adjacent to at least one vertex which has degree at most $2$ in $\Gamma_{\text{eff}}$: Vertices of degree 2 in $(V,\GE)$ are trivially at most of degree $2$ in $\Gamma_{\text{eff}}$, and vertices of degree $3$ in $(V,\GE)$ are also at most of degree $2$ in $\Gamma_{\text{eff}}$ as they can only be adjacent to $2$ effectively marked edges. Consequently any induced subgraph of $\Gamma_{\text{eff}}$ has to contain some vertex of degree at most $2$ and thereby $\Gamma_{\text{eff}}$ is $2$-degenerate.

\begin{definition}
For a graph $\Gamma=(V,\GE\cup\IE\cup\WeE)\in\cG$ we call a subset of $G$-edges $\WE\subset\GE$ \emph{Wardable} if the subgraph $(V,\WE)$ is $2$-degenerate. 
\end{definition}

\begin{lemma}\label{lemma wardable set}
For each $\Gamma\in\cG(p)$ there exists  a Wardable subset $\WE\subset\GE$ of size 
\begin{equation}\abs{\GE_W}= \sum_{e\in\IE} (4-\deg(e))_+.\label{wardable edge eq}\end{equation} 
\end{lemma}
\begin{proof}
This follows immediately from Fact \ref{Wardable edges}, the observation that $(V,\GE_{\text{eff}})$ is $2$-degenerate and the fact that sub-graphs of $2$-degenerate graphs remain $2$-degenerate.
\end{proof}
For each  $\Gamma\in\cG(p)$ we choose a Wardable subset $\WE(\Gamma)\subset\GE(\Gamma)$ satisfying \eqref{wardable edge eq}.  At least one such set is guaranteed to exist by the lemma.
 For graphs with several possible such sets, we arbitrarily choose one, and consider it permanently assigned to $\Gamma$. Later we will introduce certain operations on graphs $\Gamma\in\cG(p)$ which produce families of derived graphs $\Gamma'\in\cG\supset\cG(p)$. During those operations the chosen Wardable subset $\WE(\Gamma)$ will be modified in order to produce eligible sets of Wardable edges $\WE(\Gamma')$ and we will select one among those to define the Wardable subset of $\Gamma'$. We stress that the relation
 \eqref{wardable edge eq} on the Wardable set is required only for $\Gamma\in\cG(p)$ but not for the derived graphs $\Gamma'$. 

We now give a precise meaning to the vague bounds of \eqref{naive estimate}, \eqref{Ward}. We define the $N$-exponent, $\Nexp(\Gamma)$, of a graph $\Gamma=(V,\GE\cup\IE\cup\WeE)$ as the effective $N$-exponent in its value-definition, i.e.\ as \[ 
\Nexp(\Gamma)\defeq \abs{V} - \sum_{e\in\IE} \frac{\deg(e)}{2} - \sum_{e\in\WeE} l(e).
\]
We defer the proof of the following technical lemma to the appendix.
\begin{lemma}\label{lemma ward moment} For any \(c>0\) there exists some \(C>0\) such that the following holds. Let $\Gamma=(V,\GE\cup\IE\cup\WeE)\in\cG$ be a graph with Wardable edge set $\WE\subset\GE$ and at most $\abs{V}\le c p$ vertices and at most $\abs{\GE}\le c p^2$ $G$-edges. Then for each $0<\epsilon<1$ it holds that 
\begin{subequations}\label{eqs Val WEst bound}
\begin{equation}\label{eq ward moment}
\abs{\Val(\Gamma)}\le_\epsilon N^{\epsilon p} \big(1+\norm{G}_q\big)^{Cp^2} \WEst(\Gamma),
\end{equation}
where
\begin{equation}\label{Ward est}
\WEst(\Gamma)\defeq N^{\Nexp(\Gamma)} \big(\psi+\psi_q'\big)^{\abs{\WE}} \big(\psi+\psi'_q+\psi_q''\big)^{\abs{\GE_{g-m}}}, \qquad q\defeq C p^3/\epsilon.
\end{equation}
\end{subequations}
\end{lemma}
\begin{remark}~
\begin{enumerate}[(i)]
\item We consider $\epsilon$ and $p$ as fixed within the proof of Theorem \ref{thm pfD bound moments} and therefore do not explicitly carry the dependence of them in quantities like $\WEst$.
\item We recall that the factors involving $\GE_{g-m}$ and $\WeE$ do not play any role for graphs $\Gamma\in\cG(p)$ as those sets are empty in this restricted class of graphs (see Remark \ref{remark:single}).
\item Ignoring the difference between $\psi$ and $\psi_q'$, $\psi_q''$ and the irrelevant order $\landauO{N^{p\epsilon }}$ factor in \eqref{eqs Val WEst bound}, the reader should think of \eqref{eqs Val WEst bound} as the heuristic inequality
\[
\abs{\Val(\Gamma)}\lesssim N^{\Nexp(\Gamma)} \psi^{\abs{\WE}+\abs{\GE_{g-m}}}.
 \]
Using Lemma \ref{lemma wardable set}, $N^{-1/2}\lesssim \psi\lesssim 1$, $\abs{V}=2\abs{\IE}\le 2p$ and $\deg(e)\ge 2$ (from Fact \ref{number of edges}) we thus find 
\begin{equation}\label{non-improved moment bound heuristic}
\begin{split}
  N^{-p} \abs{\Val(\Gamma)} &\lesssim N^{\abs{\IE}-p} \prod_{e\in\IE} N^{1-\deg(e)/2} \psi^{(4-\deg(e))_+}\\
  &\lesssim \psi^{2\abs{\IE}-2p} \prod_{e\in\IE} \psi^{\deg(e)-2+(4-\deg(e))_+}\le \psi^{2p}
\end{split}  
\end{equation}
for any $\Gamma=(V,\GE\cup\IE)\in\cG(p)$.
\end{enumerate}
\end{remark}

\subsection{Improved estimates on \texorpdfstring{$\Val(\Gamma)$}{Val(Gamma)} at the cusp: \texorpdfstring{$\sigma$}{sigma}-cells}
\begin{definition}\label{sigma cell def}
For $\Gamma\in\cG$ we call an interaction edge $(u,v)=e\in\IE(\Gamma)$ a $\sigma$-cell if the following four properties 
hold: (i) $\deg(e)=2$, (ii) there are no $G$-loops adjacent to $u$ or $v$, (iii) precisely one of $u,v$ carries a weight of $\vp\vf$ while the other carries a weight of $\bm 1$, and (iv), $e$ is not adjacent to any other non $\GE$-edges. Pictorially, possible $\sigma$-cells are given by
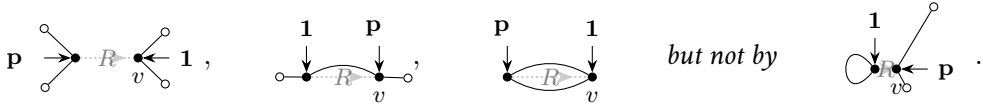
\[ 
\ssGraph{ u[lb=$u$,lpf] --[R] v[lb=$v$,r1]; x[o] -- u -- y[o]; z[o] -- v --w[o]; },\qquad
\ssGraph{ u[lb=$u$,t1] --[R] v[lb=$v$,right=5pt,tpf]; u --[bl] v; z[o] -- u; w[o,right=5pt] -- v; },\qquad
\ssGraph{ u[lb=$u$,tpf] --[R] v[lb=$v$,t1]; u --[bl] v; u --[br] v; }\qquad\text{but not by}\qquad
\ssGraph{ u[lb=$u$,t1] --[R] v[lb=$v$,rpf]; u --[gll] u; z[o] -- v -- w[o]; }.
 \]
For $\Gamma\in\cG$ we denote the number of $\sigma$-cells in $\Gamma$ by $\sigma(\Gamma)$.
\end{definition}
Next, we state a simple lemma, estimating $\WEst(\Gamma)$ of the graphs in the restricted class $\Gamma\in\cG(p)$.

\begin{lemma}\label{lemma sigma cell counting}
For each $\Gamma=(V,\IE\cup\GE)\in\cG(p)$ it holds that 
\[ 
N^{-p}\abs{\WEst(\Gamma)} \le_p \Big(\sqrt{\eta/\rho}\Big)^{p-\sigma(\Gamma)}(\psi+\psi'_q)^{2p} \prod_{\substack{e\in\IE\\\deg(e)\ge 4}} N^{2-\deg(e)/2}.
 \]
\end{lemma}
\begin{proof} We introduce the short-hand notations $\IE_k\defeq\set{e\in\IE|\deg(e)=k}$ and $\IE_{\ge k}\defeq\bigcup_{l\ge k}\IE_l$. Starting from \eqref{Ward est} and Lemma \ref{lemma wardable set} we find 
\[ 
\begin{split}
  &N^{-p} \abs{\WEst(\Gamma)} \\
  &\le N^{-(p-\abs{\IE})}\Bigg(\prod_{e\in\IE_2} (\psi+\psi'_q)^2\Bigg) \Bigg(\prod_{e\in\IE_3} \frac{\psi+\psi'_q}{\sqrt N}\Bigg) \Bigg(\prod_{e\in\IE_{\ge 4}} \frac{1}{N}\Bigg) \Bigg(\prod_{e\in\IE_{\ge 4}} N^{2-\deg(e)/2}\Bigg).
\end{split}
 \]
Using $N^{-1/2}= \psi\sqrt{\eta/\rho} \le C\psi$ it then follows that 
\begin{equation}\label{WEst sigma cell est 1}
\begin{split}
 & N^{-p}\abs{\WEst(\Gamma)} \\
 &\le_p \bigg[\frac{\eta}{\rho} \psi^2\bigg]^{p-\abs{\IE}}\Bigg(\prod_{e\in\IE_2} (\psi+\psi_q')^2\Bigg) \Bigg(\prod_{e\in\IE_{\ge 3}} \sqrt{\frac{\eta}{\rho}}(\psi+\psi_q')^2\Bigg)\Bigg(\prod_{e\in\IE_{\ge 4}} N^{2-\deg(e)/2}\Bigg).
\end{split}
\end{equation}
It remains to relate \eqref{WEst sigma cell est 1} to the number $\sigma(\Gamma)$ of $\sigma$-cells in $\Gamma$. Since each interaction edge of degree two which is not a $\sigma$-cell has an additional weight $\vp\vf$ attached to it, it follows from Fact \ref{fact weights} that $\abs{\IE_2}-\sigma(\Gamma)\le p - \abs{\IE}$. Therefore, from \eqref{WEst sigma cell est 1} and $\eta/\rho\le C$ we have that 
\[ 
\begin{split}
& N^{-p} \abs{\WEst(\Gamma)} \\
&\le_p \Big[\sqrt{\eta/\rho}(\psi+\psi'_q)^2\Big]^{p-\abs{\IE}+\abs{\IE_{\ge 3}}+\abs{\IE_2}-\sigma(\Gamma)} \Big[(\psi+\psi'_q)^2\Big]^{\sigma(\Gamma)} \Bigg(\prod_{e\in\IE_{\ge 4}} N^{2-\deg(e)/2}\Bigg),
\end{split}
 \]
proving the claim.
\end{proof}

Using Lemma \ref{lemma ward moment} and $\sqrt{\eta/\rho}\le \sigma_q$, the estimate in Lemma \ref{lemma sigma cell counting} has improved the previous bound \eqref{non-improved moment bound heuristic} by a factor $\sigma_q^{p-\sigma(\Gamma)}$ 
(ignoring the irrelevant factors). In order to prove \eqref{eq pf D moment bound}, we thus need to remove the $-\sigma(\Gamma)$ 
from this exponent, in other words, we need to show that from each $\sigma$-cell we can multiplicatively 
gain a factor of $\sigma_q$. This is the content of the following proposition. 
\begin{proposition}\label{sigma cell prop}
Let \(c>0\) be any constant and $\Gamma\in \cG$ be a single index graph with at most $cp$ vertices and $cp^2$ edges with a $\sigma$-cell $(u,v)=e\in\IE(\Gamma)$. Then there exists a finite collection of graphs $\{\Gamma_\sigma\}\sqcup \cG_\Gamma$ with at most one additional vertex and at most $6p$ additional $G$-edges such that
\begin{equation}
\begin{split}
\label{Ward est sigma cell removal}
\Val(\Gamma)&= \sigma\Val(\Gamma_\sigma)+\sum_{\Gamma'\in\cG_\Gamma}\Val(\Gamma') + \landauO{N^{-p}},\\
\WEst(\Gamma_\sigma) &= \WEst(\Gamma),\qquad \WEst(\Gamma')\le_p \sigma_q\WEst(\Gamma), \quad \Gamma'\in\cG_\Gamma 
\end{split}
\end{equation}
and all graphs $\Gamma_\sigma$ and $\Gamma'\in\cG_\Gamma$ have exactly one $\sigma$-cell less than $\Gamma$.
\end{proposition}
Using Lemma \ref{lemma ward moment} and Lemma \ref{lemma sigma cell counting} 
together with the repeated application of Proposition \ref{sigma cell prop} we are ready to present the proof of Theorem \ref{thm pfD bound moments}.

\begin{proof}[Proof of Theorem \ref{thm pfD bound moments}] We remark that the isotropic local law \eqref{isotropic bound on D} and the averaged local law \eqref{averaged bound on D} are verbatim as in \cite[Theorem 4.1]{MR3941370}. We therefore only prove the improved bound \eqref{eq pf D moment bound}--\eqref{TGG bound} in the remainder of the section. We recall \eqref{pfD val exp} and partition the set of graphs $\cG(p)=\cG_0(p)\cup\cG_{\ge1}(p)$ into those graphs $\cG_0(p)$ with no $\sigma$-cells and those graphs $\cG_{\ge 1}(p)$ with at least one $\sigma$-cell. For the latter group we then use Proposition \ref{sigma cell prop} for some $\sigma$-cell to find 
\begin{equation}\label{eq pfd expansion first step}
\begin{split}
\E\abs{\braket{\diag(\vp\vf)D}}^p &= N^{-p}\sum_{\Gamma\in\cG_0(p)}\Val(\Gamma) +\landauO{N^{-2p}}\\
&\quad +N^{-p}\sum_{\Gamma\in\cG_{\ge 1}(p)}\left(\sigma\Val(\Gamma_\sigma)+ \sum_{\Gamma'\in\cG_\Gamma}\Val(\Gamma') \right),
\end{split}
\end{equation}
where the number of $\sigma$-cells is reduced by $1$ for $\Gamma_\sigma$ and each $\Gamma'\in\cG_\Gamma$ as compared to $\Gamma$. 
We note that the Ward-estimate $\WEst(\Gamma)$ from Lemma \ref{lemma sigma cell counting} together with Lemma \ref{lemma ward moment} is already sufficient for the graphs in $\cG_0(p)$. For those graphs $\cG_1(p)$ with exactly one $\sigma$-cell the expansion in \eqref{eq pfd expansion first step} is sufficient because $\sigma\le\sigma_q$ and, according to \eqref{Ward est sigma cell removal}, each $\Gamma'\in\cG_{\Gamma}$ has a Ward estimate which is already improved by $\sigma_q$. For the other graphs we iterate the expansion from Proposition \ref{sigma cell prop} until no sigma cells are left. 

It only remains to count the number of $G$-edges and vertices in the successively derived graphs to make sure that Lemma \ref{lemma ward moment} and Proposition \ref{sigma cell prop} are applicable and that the last two factors in \eqref{eq pf D moment bound} come out as claimed. Since every of the $\sigma(\Gamma)\le p$ applications of Proposition \ref{sigma cell prop} creates at most $6p$ additional $G$-edges and one additional vertex, it follows that $\abs{\GE(\Gamma)}\le C'p^2$, $\abs{V}\le C'p$ also in any successively derived graph. Finally, it follows from the last factor in Lemma \ref{lemma sigma cell counting} that for each $e\in\IE$ with $\deg(e)\ge 5$ we gain additional factors of $N^{-1/2}$. Since $\abs{\IE}\le p$, we easily conclude that if there are more than $4p$  $G$-edges, then each of them comes with an additional gain of $N^{-1/2}$. 
Now \eqref{eq pf D moment bound} follows immediately after taking the $p$-th root. 

We turn to the proof of \eqref{TGG bound}. We first write out
\[ 
\braket{\diag(\vp\vf) [T\odot G^t]G} = \frac{1}{N} \sum_{a,b} (pf)_a t_{ab} G_{ba} G_{ba}
 \]
and therefore can, for even $p$, write the $p$-th moment as the value
\[ 
\E\abs{\braket{\diag(\vp\vf) [T\odot G^t]G}}^p = N^{-p}\Val(\Gamma_0)
 \]
of the graph $\Gamma_0=(V,\GE\cup\IE)\in\cG$ which is given by $p$ disjoint $2$-cycles as 
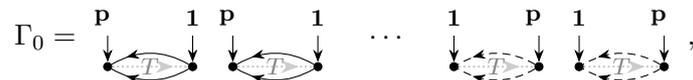
\[ 
\Gamma_0=\ssGraph{ u[tpf] --[T] v[t1]; v --[bl,g] u; v --[br,g] u; }\ssGraph{ u[tpf] --[T] v[t1]; v --[bl,g] u; v --[br,g] u; } \quad\cdots\quad \ssGraph{ u[t1] --[T] v[tpf]; v --[bl,s] u; v --[br,s] u; } \ssGraph{ u[t1] --[T] v[tpf]; v --[bl,s] u; v --[br,s] u; },
 \]
where there are $p/2$ cycles of $G$-edges and $p/2$ cycles of $G^\ast$ edges. It is clear that $(V,\GE)$ is $2$-degenerate and since $\abs{\GE}=2p$ it follows that 
\[ 
\WEst(\Gamma_0) \le N^p(\psi+\psi_q')^{2p}.
 \]
On the other hand each of the $p$ interaction edges in $\Gamma_0$ is a $\sigma$-cell and we can use Proposition \ref{sigma cell prop} $p$ times to obtain \eqref{TGG bound} just as in the proof of \eqref{eq pf D moment bound}.
\end{proof}

\subsection{Proof of Proposition \ref{sigma cell prop}}
It follows from the MDE that 
\[
G=M-M\SS[M]G- MWG=M-G\SS[M]M-GWM,
\]
which we use to locally expand a term of the form $G_{xa}G^\ast_{ay}$ for fixed $a,x,y$ further. To make the computation local we allow for an arbitrary random function $f=f(W)$, which in practice encodes the remaining $G$-edges in the graph. A simple cumulant expansion shows
\begin{equation} \begin{split} \label{BGGast exp}
\sum_{b}B_{ab}\E G_{x b} G^{\ast}_{b y} f &=\E M_{xa} G^\ast_{ay} f - \sum_{k=2}^{6p} \sum_{b}\sum_{\bm\beta\in I^k} \kappa(ba,\underline\beta)m_a \E\partial_{\bm\beta}\Big[ G_{xb} G^{\ast}_{a y} f \Big]  + \landauO{N^{-p}} \\
& \; + \sum_b s_{ba} m_a \E\Big[G_{xa}(g-m)_{b} G^\ast_{ay}+G_{xb} \overline{(g-m)}_a G^\ast_{by}- G_{xb} G^\ast_{ay}\partial_{ab} \Big]f\\
&\; +\sum_b t_{ba}m_a \E\Big[G_{xb}(G-M)_{ab} G^\ast_{ay}+ G_{xb} G^\ast_{ab}G^\ast_{ay} - G_{xb} G^\ast_{ay}\partial_{ba}\Big]f
\end{split}\end{equation}
where $\partial_\alpha\defeq \partial_{w_\alpha}$ and introduced the stability operator $B\defeq 1-\diag(\abs{\vm}^2) S$. The stability operator $B$ appears from rearranging the equation obtained from the cumulant expansion to express
the quantity $\E G_{x b} G^{\ast}_{b y} f$.  In our graphical representation, the stability operator is 
a special edge that we can also express as  
\begin{equation}\label{B graphical rep}
\val{\ssGraph{ a[o,lb=$x$] ->[B=$B$] b[o,lb=$y$] }} = \val{\ssGraph{ a[o,lb=$x$] --[eq] b[o,lb=$y$]; }} - \val{\ssGraph{ a[inl=$\abs{\vm}^2$,o,lb=$x$] --[S] b[o,lb=$y$]; }}. 
\end{equation}
An equality like \eqref{B graphical rep} is meant locally in the sense that the pictures only represent subgraphs of the whole graph with the empty, labelled vertices symbolizing those vertices which connect the subgraph to its complement. Thus \eqref{B graphical rep} holds true for every fixed graph extending $x,y$ consistently in all three graphs. The doubly drawn edge in \eqref{B graphical rep} means that the external vertices $x,y$ are identified with each other and the associated indices are set equal via a $\delta_{a_x,a_y}$ function. Thus \eqref{B graphical rep} should be understood as the equality
\begin{equation}\label{subgraph picture}
\val{\ssGraph{ x ->[B=$B$] y; x -- a[wh]; x -- b[wh] ; y -- c[wh]; inner[draw,starburst,starburst point height=2pt] //{x,y};  }} 
= \val{\ssGraph{ x; x -- a[wh]; x -- b[wh] ; x -- c[wh]; inner[draw,starburst,starburst point height=2pt] //{x};  }} 
-\val{\ssGraph{ x[inl=] --[S] y; x -- a[wh]; x -- b[wh] ; y -- c[wh]; inner[draw,starburst,starburst point height=2pt] //{x,y};  }} 
\end{equation}
where the outside edges incident at the merged vertices $x,y$ are reconnected to one common vertex in the middle graph. For example, in the picture \eqref{subgraph picture} the vertex $x$ is connected to the rest of the graph by two edges, and the vertex $y$ by one. 

In order to represent \eqref{BGGast exp} in terms of graphs we have to define a notion of \emph{differential edge}. First, we define a \emph{targeted differential edge} represented by an interaction edge with a red $\partial$-sign written on top and a red-coloured \emph{target $G$-edge} to denote the collection of graphs 
\begin{equation}\label{pd def}
\ssGraph{u[label=$u$,o] --[pdr] v[label=$v$,o]; x[lb=$x$,o] --[g,r] y[lb=$y$,o];} \defeq \Set{\ssGraph{x[label=below:$x$,o] --[g] u[label=$u$,o] --[R] v[label=$v$,o] --[g] y[lb=$y$,o];},  \ssGraph{u[label=$u$,o] --[R] v[label=$v$,o]; u --[g] y[lb=$y$,o]; x[lb=$x$,o] --[g] v;}}, \qquad \ssGraph{u[label=$u$,o] --[pdr] v[label=$v$,o]; x[lb=$x$,o] --[g,glr,r] x; } \defeq \Set{\ssGraph{u[label=$u$,o] --[R] v[label=$v$,o]; x[lb=$x$,o] --[g] v; u --[g] x;},\ssGraph{u[label=$u$,o] --[R] v[label=$v$,o]; x[lb=$x$,o] --[g] u; v --[g] x; }}.
\end{equation}
The second picture in \eqref{pd def} shows that the target $G$-edge may be a loop; the definition remains the same.
This definition extends naturally to $G^\ast$ edges and is exactly the same for $G-M$ edges (note that this is compatible with the usual notion of derivative as $M$ does not depend on $W$). Graphs with the differential signs should be viewed only 
as an intermediate simplifying picture but they really mean the collection of graphs indicated in the right hand side of \eqref{pd def}. They represent the identities 
\[ 
\begin{split}
  \sum_{\alpha} \kappa(uv,\alpha)\partial_{uv} G_{xy} &= - s_{uv} G_{xv} G_{uy} - t_{uv} G_{xu} G_{vy}, \\
  \sum_{\alpha} \kappa(uv,\alpha) \partial_{uv} G_{xx} &= - s_{uv} G_{xv} G_{ux} - t_{uv} G_{xu} G_{vx}
\end{split}
 \]
In other words we introduced these graphs only to temporary encode
expressions with derivatives (e.g.~second term in the rhs.~of \eqref{BGGast exp}) \emph{before} the differentiation is 
actually performed. We can then further define the action of an \emph{untargeted differential edge} according the Leibniz rule as the collection of graphs with the differential edge being targeted on all $G$-edges of the graph one by one (in particular not only those in the displayed subgraph), i.e.\ for example 
\begin{equation}\label{der edge Leibniz}
\ssGraph{u[label=$u$,o] --[pd] v[label=$v$,o]; x[label=below:$x$,o] --[g] y[lb=$y$,o] --[g] z[lb=$z$,o];}  \defeq \ssGraph{u[label=$u$,o] --[pdr] v[label=$v$,o]; x[label=below:$x$,o] --[g,r] y[lb=$y$,o] --[g] z[lb=$z$,o];}  \bigsqcup \ssGraph{u[label=$u$,o] --[pdr] v[label=$v$,o]; x[label=below:$x$,o] --[g] y[lb=$y$,o] --[g,r] z[lb=$z$,o];} \bigsqcup \dots.
\end{equation}
Here the union is a union in the sense of multisets, i.e.\ allows for repetitions in the resulting set (note that also this is compatible with the usual action of derivative operations). The $\sqcup\dots$ symbol on the rhs.~of \eqref{der edge Leibniz} indicates that the targeted edge cycles through \emph{all} $G$-edges in the graph, not only the ones in the subgraph.  
For example, if there are $k$ $G$-edges in the graph, then the picture \eqref{der edge Leibniz} represents a collection of $2k$ graphs arising from performing the differentiation
\begin{equation*}
\begin{split}
&\sum_\alpha \kappa(uv,\alpha) \partial_{uv} \big[G_{xy}G_{yz} f\big]\\ 
&= \sum_\alpha \kappa(uv,\alpha)  \big[\partial_{uv} G_{xy}\big]G_{yz}f+\sum_\alpha \kappa(uv,\alpha)  G_{xy}\big[\partial_{uv}G_{yz} \big]f+\sum_\alpha \kappa(uv,\alpha)  G_{xy}G_{yz}\big[\partial_{uv}f \big] \\
&= - s_{uv} \big[G_{xv}G_{uy}G_{yz}f + G_{xy} G_{yv}G_{uz}f + G_{xy} G_{yz} (\partial_{vu}f)\big]\\
&\quad - t_{uv} \big[G_{xu}G_{vy}G_{yz}f + G_{xy} G_{yu}G_{vz}f+G_{xy} G_{yz} (\partial_{uv}f)\big],
\end{split}
\end{equation*}
where $f=f(W)$ represents the value of the $G$-edges outside the displayed subgraph. 

Finally we introduce the notation that a differential edge which is targeted on all $G$-vertices except for those in the displayed subgraph. This differential edge targeted on the outside will be denoted by $\widehat \partial$. 

Regarding the value of the graph, we define the value of a collection of graphs as the sum of their values. We note that this definition is for the collection of graphs encoded by the differential edges also consistent with the usual differentiation.

Written in a graphical form \eqref{BGGast exp} reads 
\begin{equation} \begin{split}
\label{BGGast graph exp}
&\val{\ssGraph{x[o,lb=$x$] --[g] b[r1] --[s] y[o,lr=$y$]; a[o,lr=$a$] ->[B=$B$] b;}} = \val{\ssGraph{ x[o,lb=$x$] --[eq] a[tm,o,lb=$a$] --[s] y[o,lb=$y$]; }} - \sum_{k=2}^{6p}\val{\ssGraph{x[o,lb=$x$] --[g] b[t1] --[pdn=k] a[tm,o,lb=$a$] --[s] y[o,lb=$y$]; }} + \landauO{N^{-p}}   \\
&\quad+ \val{\ssGraph{ x[o,lt=$x$] --[g] a[tm,o,lr=$a$] --[s] y[lr=$y$,o]; a --[S] b[l1] --[g,glb,gm] b; }} + \val{\ssGraph{ x[o,lt=$x$] --[g] b --[s] y[lr=$y$]; b[r1] --[S] a[lm,o,lr=$a$] --[s,glb,gm] a; }} + \val{\ssGraph{ x[lb=$x$,o] --[g] b[t1]; a[tm,o,lb=$a$] --[g,bl,gm] b; b --[bl,T] a; a --[s] y[lb=$y$,o];  }}\\
&\quad + \val{\ssGraph{ x[o,lb=$x$] --[g] b[t1]; a[tm,o,lb=$a$] --[s,bl] b; b --[bl,T] a; a --[s] y[lb=$y$,o];  }}- \val{\ssGraph{x[o,lb=$x$] --[g] b[t1] --[pdh] a[tm,lb=$a$,o] --[s] y[lb=$y$,o]; }},
\end{split} 
\end{equation}
where the ultimate graph encodes the ultimate terms in the last two lines of \eqref{BGGast exp}.

We worked out the example for  the resolution of the quantity $\E G_{x a} G^{\ast}_{a y} f$, but very similar formulas hold
if the order of the fixed indices $(x,y)$ and the summation index $a$ changes in the resolvents, as well as 
for other combinations of the complex conjugates. In graphical language this corresponds to changing the
arrows of the two  $G$-edges adjacent to $a$, as well as their types.   In other words, 
equalities like the one in \eqref{BGGast graph exp} hold true for other any degree two vertex but the stability operator changes slightly: In total there are $16$ possibilities, four for whether the two edges are incoming or outgoing at $a$ and another four for whether the edges are of type $G$ or of type $G^\ast$. The general form for the stability operator is 
\begin{equation}\label{B def}
B\defeq 1-\diag(\vm^{\#_1}\vm^{\#_2})R,
\end{equation}
where $R=S$ if there is one incoming and one outgoing edge, $R=T$ if there are two outgoing edges and $R=T^t$ otherwise, and where $\#_1,\#_2$ represent complex conjugations if the corresponding edges are of $G^\ast$ type. Thus for, for example, the stability operator in $a$ for $G_{xa}^\ast G_{ya}^\ast$ is $1-\diag(\overline{\vm}^2)T^t$.  Note that the stability operator at vertex with degree two
is exclusively determined by the type and orientation of the two $G$-edges adjacent to $a$. In the sequel the letter $B$ will refer to the appropriate stability operator, we will not distinguish their 9 possibilities ($R=S,T,T^t$ and $\vm^{\#_1}\vm^{\#_2}=\abs{\vm}^2,\vm^2,\overline{\vm}^2$) in the notation. 
\begin{lemma}\label{B insertion Lemma}
Let \(c>0\) be any constant, $\Gamma\in \mathcal{G}$ be a single index graph with at most $cp$ vertices and $cp^2$ edges and let $a\in V(\Gamma)$ be a vertex of degree $\deg(a)=2$ not adjacent to a \(G\)-loop. The insertion of the stability operator 
$B$ \eqref{B def} at $a$ as in \eqref{BGGast graph exp} produces a finite set of graphs with at most one additional vertex and $6p$ additional edges, denoted by  $\cG_\Gamma$, 
such that 
\[ 
\Val(\Gamma)=\sum_{\Gamma'\in\cG_\Gamma} \val{\Gamma'} +\landauO{N^{-p}}, 
 \]
and all of them have a Ward estimate 
\[ \WEst(\Gamma')\le_p \big(\rho+\psi+\eta/\rho+\psi_q'+\psi_q''\big)\WEst(\Gamma)\le_p \sigma_q \WEst(\Gamma),\qquad \Gamma'\in\cG_\Gamma. \]
Moreover all $\sigma$-cells  in $\Gamma$, except possibly a $\sigma$-cell adjacent to $a$, remain $\sigma$-cells also in each $\Gamma'$.
\end{lemma}
\begin{proof}
As the proofs for all of the 9 cases of $B$-operators are almost identical we prove the lemma for the case \eqref{BGGast graph exp} for definiteness. Now we compare the value of the graph 
\[ \Gamma\defeq\ssGraph{x[lb=$x$,o] --[g] a[lb=$a$,o] --[s] y[lb=$y$,o]; } \] 
with the graph in the lhs.~of \eqref{BGGast graph exp}, i.e.
when the stability operator $B$ is attached to the vertex $a$. We remind the reader that the displayed graphs only show a certain subgraph of the whole graph.  The goal is to show that $\WEst\left(\Gamma' \right) \le \big(\rho+\psi+\eta/\rho+\psi_q'+\psi_q''\big) \WEst (\Gamma)$ for each graph $\Gamma'$ occurring on the rhs.~of \eqref{BGGast graph exp}. The forthcoming reasoning is based on comparing the quantities $\abs{V}$, $\abs{\WE}$, $\abs{\GE_{g-m}}$ and $\sum_{e\in\IE} \deg(e)/2$ defining the Ward estimate $\WEst$ from \eqref{Ward est} of the graph $\Gamma$ and the various graphs $\Gamma'$ occurring on the rhs.~of \eqref{BGGast graph exp}.
\begin{enumerate}[(a)]
\item We begin with the first graph and claim that 
\[
\WEst\left(\ssGraph{ x[o,lb=$x$] --[eq] a[tm,o,lb=$a$] --[s] y[o,lb=$y$]; }\right) \le \frac{1}{N\psi^2} \WEst(\Gamma) =\frac{\eta}{\rho}\WEst(\Gamma).
 \]
Due to the double edge which identifies the $x$ and $a$ vertices it follows that $\abs{V(\Gamma')}=\abs{V(\Gamma)}-1$. The degrees of all interaction edges remain unchanged when going from $\Gamma$ to $\Gamma'$. As the $2$-degenerate set of Wardable edges $\WE(\Gamma')$ we choose $\WE(\Gamma)\setminus N(a)$, i.e.\ the $2$-degenerate edge set in the original graph except for the edge-neighbourhood $N(a)$ of $a$, i.e.\ those edges adjacent to $a$. As a subgraph of $(V,\WE(\Gamma))$ it follows that $(V\setminus\{a\},\WE(\Gamma'))$ is again $2$-degenerate. Thus $\abs{\WE(\Gamma)}\ge\abs{\WE(\Gamma')}\ge \abs{\WE(\Gamma)}-2$ and the claimed bound follows since $\abs{\GE_{g-m}(\Gamma')}=\abs{\GE_{g-m}(\Gamma)}$ and 
\[ 
\frac{\WEst(\Gamma')}{\WEst(\Gamma)} = \frac{1 }{N (\psi+\psi_q')^{\abs{\WE(\Gamma)}-\abs{\WE(\Gamma')}} }\le \frac{1}{N\psi^2}.
\]
\item Next, we consider the third and fourth graph and claim that 
\[
\WEst\left(\ssGraph{ x[o,lt=$x$] --[g] a[tm,o,lr=$a$] --[s] y[lr=$y$,o]; a --[S] b[l1] --[g,glb,gm] b; }\right)+\WEst\left(\ssGraph{ x[o,lt=$x$] --[g] b[lt=$b$] --[s] y[lr=$y$,o]; b[r1] --[S] a[lm,o,lr=$a$] --[s,glb,gm] a; }\right) =2(\psi+\psi_q'+\psi''_q) \WEst(\Gamma) .
 \]
Here there is one more vertex (corresponding to an additional summation index), $\abs{V(\Gamma')}=\abs{V(\Gamma)}+1$, whose effect in \eqref{Ward est} is compensated by one additional interaction edge $e$ of degree $2$. Hence the $N$-exponent $n(\Gamma)$ remains unchanged. In the first graph we can simply choose $\WE(\Gamma')=\WE(\Gamma)$, whereas in the second graph we choose $\WE(\Gamma')=\WE(\Gamma)\setminus\{(x,a),(a,y)\}\cup\{(x,b),(b,y)\}$ which is $2$-degenerate as a subgraph of a $2$-degenerate graph together with an additional vertex of degree $2$. Thus in both cases we can choose $\WE(\Gamma')$ (if necessary, by removing excess edges from $\WE(\Gamma')$ again) in such a way that $\abs{\WE(\Gamma')}=\abs{\WE(\Gamma)}$ but the number of $(g-m)$-loops is increased by $1$, i.e.\ $\abs{\GE_{g-m}(\Gamma')}=\abs{\GE_{g-m}(\Gamma)}+1$. 
\item Similarly, we claim for the fifth and sixth graph that
\[ 
\WEst\left(\ssGraph{ x[lb=$x$,o] --[g] b[t1,lb=$b$]; a[tm,o,lb=$a$] --[g,bl,gm] b; b --[bl,T] a; a --[s] y[lb=$y$,o];  }\right)+\WEst\left(\ssGraph{ x[o,lb=$x$] --[g] b[t1,lb=$b$]; a[tm,o,lb=$a$] --[s,bl] b; b --[bl,T] a; a --[s] y[lb=$y$,o];  }\right) = 2(\psi+\psi_q') \WEst(\Gamma).
 \]
There is one more vertex whose effect in \eqref{Ward est} is compensated by one more interaction edge of degree $2$, whence the number $N$-exponent remains unchanged. The number of Wardable edges can be increased by one by setting $\WE(\Gamma')$ to be a suitable subset of $\WE(\Gamma)\setminus\{(x,a),(a,y)\}\cup\{(x,b),(a,b),(a,y)\}$ which is $2$-degenerate as the subset of a $2$-degenerate graph together with two vertices of degree $2$. The number of $(g-m)$-loops remains unchanged.
\item For the last graph in \eqref{BGGast graph exp}, i.e.\ where the derivative targets an outside edge, we claim that 
\[ 
\WEst\left(\ssGraph{X[o,lb=$x$] --[g] b[t1] --[pdh] a[tm,lb=$a$,o] --[s] y[lb=$y$,o]; }\right) \le_p (\psi+\psi'_q+\psi''_q) \WEst(\Gamma).
 \]
Here the argument on the lhs., $\Gamma'$, stands for a whole collection of graphs but we essentially only have to consider two types: The derivative edge either hits a $G$-edge or a $(g-m)$-loop, i.e.\ 
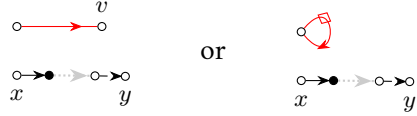
\[ \ssGraph{u[lt=$u$,o] --[r,g] v[lt=$v$,o]; x[o,lb=$x$] --[g] b --[pdr] a[lb=$a$,o] --[s] y[lb=$y$,o];}\qquad \text{or}\qquad \ssGraph{u[lt=$u$,o] --[r,glr,g,gm] u; x[o,lb=$x$] --[g] b --[pdr] a[lb=$a$,o] --[s] y[lb=$y$,o];}\]
which encodes the graphs
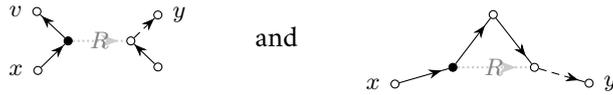
\[ \ssGraph{b --[R] a[lt=$a$,o]; x[o,ll=$x$] --[g] b --[g] v[ll=$v$,o]; u[lr=$u$,o] --[g] a --[s] y[lr=$y$,o]; }\qquad \text{and}\qquad\ssGraph{b --[R] a[lb=$a$,o]; x[o,ll=$x$] --[g] b --[g] u[lt=$u$,o] --[g] a --[s] y[lr=$y$,o]; }\]
as well as the corresponding transpositions (as in \eqref{pd def}). In both cases the $N$-size of $\WEst$ remains constant since the additional vertex is balanced by the additional degree two interaction edge. In both cases all four displayed edges can be included in $\WE(\Gamma')$. So $\abs{\WE}$ can be increased by $1$ in the first case and by $2$ in the second case while the number of $(g-m)$-loops remains constant in the first case is decreased by $1$ in the second case. The claim follows directly in the first case and from
\[ 
\frac{\WEst(\Gamma')}{\WEst(\Gamma)}=\frac{(\psi+\psi'_q)^2}{\psi+\psi'_q+\psi''_q} \le \psi+\psi'_q +\psi''_q
 \]
in the second case.

\item It remains to consider the second graph in the rhs.~of \eqref{BGGast graph exp} with the higher derivative edge. We claim that for each $k\ge 2$ it holds that
\[ 
\WEst\left(\ssGraph{x[o,lb=$x$] --[g] b[t1] --[pdn=k] a[tm,o,lb=$a$] --[s] y[o,lb=$y$]; }\right) \le_p (\psi+\psi_q') \WEst(\Gamma).
 \]
We prove the claim by induction on $k$ starting from $k=2$. For any $k\ge 2$ we write $\partial^k = \partial^{k-1}\partial$. 
For the action of the last derivative we distinguish three cases: 
(i) action on an edge adjacent to the derivative edge, (ii) action on a non-adjacent $G$-edge and 
(iii) an action on a non-adjacent $(g-m)$-loop. Graphically this means 
\begin{equation}\label{k-1 red graphs}
\ssGraph{x[o,lb=$x$] --[r,g] b --[pdkr=k-1] a[lb=$a$,o,right=10pt] --[s] y[lb=$y$,o,right=10pt];},\quad 
\ssGraph{u[lt=$u$,o] --[r,g] v[lt=$v$,o]; x[o,lb=$x$] --[g] b --[pdkr=k-1] a[lb=$a$,o,right=10pt] --[s] y[lb=$y$,o,right=10pt];}\quad \text{or}\quad 
\ssGraph{u[lt=$u$,o] --[r,glr,g,gm] u; x[o,lb=$x$] --[g] b --[pdkr=k-1] a[lb=$a$,o,right=10pt] --[s] y[lb=$y$,o,right=10pt];}.\end{equation}
We ignored the case where the derivative acts on $(a,y)$ since it is estimated identically to the first graph. 
We also neglected the possibility that the derivative acts on a $g$-loop, as this is estimated exactly as the last 
graph and the result is even better since no $(g-m)$-loop is destroyed. After performing the last derivative in \eqref{k-1 red graphs} 
we obtain the following graphs $\Gamma'$ 
\begin{equation}\label{k-1 graphs} 
\ssGraph{b[lt=$b$] --[pdk=k-1] a[lr=$a$,o,right=10pt] --[s] y[o,ll=$y$,right=10pt]; b --[gll,g] b; x[o,ll=$x$,right=10pt] --[g] a},\quad 
\ssGraph{x[o,lb=$x$] --[g] b[lb=$b$] --[pdk=k-1] a[lb=$a$,o,right=10pt] --[s] y[o,lb=$y$,right=10pt]; a --[g,bl] b; },\quad 
\ssGraph{b --[pdk=k-1] a[lr=$a$,o,right=10pt]; x[o,ll=$x$] --[g] b[lb=$b$] --[g] v[ll=$v$,o]; u[lr=$u$,o,right=10pt] --[g] a --[s] y[lr=$y$,o,right=10pt]; }\quad \text{and}\quad
\ssGraph{b[lt=$b$] --[pdk=k-1] a[lt=$a$,o,right=10pt]; x[o,lt=$x$] --[g] b --[g] u[lt=$u$,o,above right=5pt and 5pt] --[g] a --[s] y[lt=$y$,o,right=10pt]; }\end{equation}
where we neglected the transposition of the third graph with $u,v$ exchanged because this is equivalent with regard to the counting argument. First, we handle the second, third and fourth graphs in \eqref{k-1 graphs}. In all these cases the set $\WE(\Gamma')$ 
is defined simply by adding all edges drawn in \eqref{k-1 graphs} to the set $\WE(\Gamma)\setminus \{(x,a), (a,y)\}$. 
The new set remains $2$-degenerate since all these new edges are adjacent to vertices of degree $2$. Compared to the original graph, $\Gamma$, we thus have increased $\abs{\WE}+\abs{\GE_{g-m}}$ by at least $1$. 

We now continue with the first graph in \eqref{k-1 graphs}, where we explicitly expand the action of another derivative (notice that this is the only graph where $k\ge 2$ is essentially used). We distinguish four cases, depending on whether the derivative acts on (i) the $b$-loop, (ii) an adjacent edge, (iii) a non-adjacent edge or (iv) a non-adjacent $(g-m)$-loop, i.e.\ graphically we have 
\begin{equation}\label{k-2 graphs red}
\ssGraph{b[lt=$b$] --[pdkr=k-2] a[lr=$a$,o,right=20pt of b] --[s] y[o,ll=$y$,below=10pt of a]; b --[gll,g,r] b; x[o,ll=$x$] --[g] a},\quad 
\ssGraph{b[lt=$b$] --[pdkr=k-2] a[lr=$a$,o,right=20pt of b] --[s] y[o,ll=$y$,below=10pt of a]; b --[gll,g] b; x[o,ll=$x$,right=11pt] --[g,r] a},\quad 
\ssGraph{u[lt=$u$,o] --[r,g] v[lt=$v$,o]; b[lt=$b$] --[pdkr=k-2] a[lr=$a$,o,right=20pt of b] --[s] y[o,ll=$y$,below=10pt of a]; b --[gll,g] b; x[o,ll=$x$,right=11pt] --[g] a}\quad\text{and}\quad
\ssGraph{u[lr=$u$,o,above=10pt] --[r,glt,g,gm] u; b[lb=$b$] --[pdkr=k-2] a[lr=$a$,o,right=7pt of b] --[s] y[o,lr=$y$,below=10pt of a]; b --[gll,g] b; x[o,lr=$x$,above=7pt of a] --[g] a}.
\end{equation}
After performing the indicated derivative, the encoded graphs $\Gamma'$ are 
\begin{equation}\label{k-2 graphs}
\ssGraph{b[lt=$b$] --[pdk=k-2] a[lr=$a$,o,right=12pt] --[s] y[o,ll=$y$,right=12pt]; b --[gll,g] b; x[o,ll=$x$,right=12pt] --[g] a; b --[br,g] a;},\quad 
\ssGraph{b[lb=$b$] --[pdk=k-2] a[lb=$a$,o,right=11pt] --[s] y[o,lt=$y$,right=11pt]; b --[glt,g] b; a --[glt,g] a; x[o,lt=$x$] --[g] b;},\quad 
\ssGraph{ b[lt=$b$] --[pdk=k-2] a[lr=$a$,o,right=11pt] ; b --[glb,g] b; a --[g] v[right=8pt,o,ll=$v$]; u[o,lt=$u$] --[g] b; a --[s] y[left=4pt,o,ll=$y$]; x[o,lt=$x$,above right=15pt and 8pt] --[g] a}\quad\text{and}\quad
\ssGraph{b[lt=$b$] --[pdk=k-2] a[lt=$a$,o,right=12pt] --[s] y[o,lr=$y$,right=12pt]; b --[glb,g] b; x[o,lr=$x$,right=12pt] --[g] a; b --[g] u[lb=$u$,o,above right=4pt and 10pt] -- a;},
\end{equation}
where we again neglected the version of the third graph with $u,v$ exchanged. We note that both the first and the second graph in \eqref{k-2 graphs red} produce the first graph in \eqref{k-2 graphs}. Now we define how to get the set $\WE(\Gamma')$ from $\WE(\Gamma)\setminus\{ (x,a), (a,y)\}$ for each case. In the first graph of \eqref{k-2 graphs} we add all three non-loop edges to $\WE(\Gamma')$, in the second graph we add both non-loop edges, in the third and fourth graph we add the non-looped edge adjacent to $b$ as well as any two non-looped edges adjacent to $a$. Thus, compared to the original graph the number $\abs{\WE}+\abs{\GE_{g-m}}$ is at least preserved. On the other hand the $N$-power counting is improved by $N^{-1/2}$. Indeed, there is one additional vertex $b$,  yielding a factor $N$, which is compensated by the scaling factor $N^{-3/2}$ from the interaction edge of degree $3$. 

To conclude the inductive step we note that additional derivatives (i.e.\ the action of $\partial^{k-2}$) can only decrease the Ward-value of a graph. Indeed, any single derivative can at most decrease the number $\abs{\WE(\Gamma)}+\abs{\GE_{g-m}}$ by $1$ by either differentiating a $(g-m)$-loop or differentiating an edge from $\WE$. Thus the number $\abs{\WE}+\abs{\GE_{g-m}}$ is decreased by at most $k-2$ while the number $\abs{\GE_{g-m}}$ is not increased. In particular, by choosing a suitable subset of Wardable edges, we can define $\WE(\Gamma')$ in such a way that $\abs{\WE}+\abs{\GE_{g-m}}$ is decreased by exactly $k-2$. But at the same time each derivative provides a gain of $cN^{-1/2}\le \psi\le \psi+\psi_q'$ since the degree of the interaction edge is increased by one. Thus we have 
\[ 
\frac{\WEst(\Gamma')}{\WEst(\Gamma)} \le_p (\psi+\psi_q')^{k-1 +\abs{\WE(\Gamma')}+\abs{\GE_{g-m}(\Gamma')}-\abs{\WE(\Gamma)}-\abs{\GE_{g-m}(\Gamma)} } = \psi+\psi_q',
 \]
just as claimed. \qedhere
\end{enumerate}
\end{proof}

Lemma \ref{B insertion Lemma} shows that the insertion of the $B$-operator reduces the Ward-estimate by at least $\rho$.
However,  this insertion does not come for free since the inverse 
\[ 
B^{-1} = (1-\diag(\vm^{\#_1}\vm^{\#_2})R)^{-1}
 \]
is generally not a uniformly bounded operator. For example, it follows from  \eqref{Dyson equation} that
\[
\Im \vm =  \eta \abs{\vm}^2 + \abs{\vm}^2 S\Im \vm
\]
and therefore $(1-\diag(\abs{\vm}^2) S)^{-1}$ is singular for small $\eta$ with $\Im \vm$ being the unstable direction. It turns out, however, that $B$ is invertible on the subspace complementary to some bad direction $\vb^{(B)}$. 
At this point we distinguish two cases. If $B$ has a uniformly bounded inverse, i.e.\ if $\norm{B^{-1}}_{\infty \to \infty}\le C$ for some constant $C>0$, then we set $P_B\defeq 0$. Otherwise we define $P_B$ as the spectral projection operator onto the eigenvector $\vb^{(B)}$ of $B$ corresponding to the eigenvalue $\beta$ with smallest modulus:
\begin{equation}\label{PB QB defs}
P_B\defeq \frac{\braket{\vl^{(B)},\cdot}}{\braket{\vl^{(B)},\vb^{(B)}}} \vb^{(B)},\qquad Q_B\defeq 1-P_B,
\end{equation}
where $\braket{\vv,\vw}\defeq N^{-1} \sum_a \overline{v_a} w_a$ denotes the normalized inner product and $\vl^{(B)}$ is the corresponding left eigenvector,  $(B^*-\beta)\vl^{(B)} = 0$. 
\begin{lemma}\label{lemma B inverse}
For all 9 possible $B$-operators in \eqref{B def} it holds that
\begin{equation}
\label{upper bound on B inverse}
\norm{B^{-1} Q_B}_{\infty\to\infty} \le C < \infty
\end{equation}
for some constant $C>0$, depending only on model parameters.
\end{lemma}
\begin{proof} First we remark that it is sufficient to prove the bound \eqref{upper bound on B inverse} on $B^{-1}Q_B$ as an operator on $\C^N$ with the Euclidean norm, i.e.\ $\norm{B^{-1} Q_B} \le C$. For this insight we refer to \cite[Proof of (5.28) and (5.40a)]{1506.05095}. Recall that $R =S$, $R=T$ or $R=T^t$, depending on which stability operator we consider (cf.~\eqref{B def}).
We begin by considering the complex hermitian symmetry class and the cases $R=T$ and $R=T^t$. We will now see that in this case $B$ has a bounded inverse and thus $Q_B =1$. Indeed, we have  
\[
 \norm{B^{-1}} \lesssim \frac{1}{1-\norm[0]{F^{(R)}}},
\]
where $F^{(R)}\vw\defeq \abs{\vm} R (\abs{\vm}\vw)$. The fullness Assumption \ref{fullness} in \eqref{fullness cplx} implies that $\abs{t_{ij}} \le (1-c)s_{ij}$ for some constant $c>0$ and thus $\norm[0]{F^{(R)}} \le (1-c)\norm{F^{(S)}}\le 1-c$ for $R=T,T^t$. Here we used $\norm{F^{(S)}}\le 1$, a general property of the saturated self-energy matrix $F^{(S)}$ that was first established in \cite[Lemma 4.3]{MR3684307}  (see also \cite[Eq.~(4.24)]{MR3916109} and \cite[Eq.~(4.5)]{1804.07752}). Now we turn to the case $R=S$ for both the  real symmetric and complex hermitian symmetry classes. In this case $B$ is the restriction to diagonal matrices of an operator $\mathcal{T}: \C^{N \times N} \to \C^{N \times N}$, where $\mathcal{T} \in \{\mathrm{Id}-M^*\mathcal{S}[\cdot]M,\mathrm{Id}-M\mathcal{S}[\cdot]M,\mathrm{Id}-M^*\mathcal{S}[\cdot]M^*\}$. All of these operators were covered in \cite[Lemma 5.1]{1804.07752} and thus \eqref{upper bound on B inverse} is a consequence of that lemma.  Recall that the flatness \eqref{flatness of cal S} of $\mathcal{S}$ ensured the applicability of the lemma.
\end{proof}
We will insert the identity $1= P_B+BB^{-1}Q_B$,  and we will perform an explicit calculation for the $P_B$ component,
while using the boundedness of $B^{-1}Q_B$ in the other component. We are thus left with studying the  effect of inserting $B$-operators
 and suitable projections into a $\sigma$-cell. To include all possible cases with regard to edge-direction and edge-type (i.e.\ $G$ or $G^\ast$), in the pictures below we neither indicate directions of the $G$-edges nor their type but implicitly allow all possible assignments. We recall that both the $R$-interaction edge as well as the \emph{relevant} $B$-operators (cf.~\eqref{B def}) are completely determined by the type of the four $G$-edges as well as their directions. To record the type of the inserted $B$, $P_B$, $Q_B$ operators we call those inserted on the rhs.~of the $R$-edge $B'$, $P_B'$ and $Q_B'$ in the following graphical representations.
 Pictorially we start first decompose the $\sigma$-cell subgraph of some graph $\Gamma$ as
\begin{equation}\label{sigma cell resolution PQ}
\begin{split}
  \Val(\Gamma)&=\val{\ssGraph{ a[lpf] --[R] b[r1]; x[o,lr=$x$] -- a -- y[o,lr=$y$]; z[o,ll=$z$] -- b -- w[o,ll=$w$]; }}\\ 
  &= \val{\ssGraph{ c[tpf] --[R] b[r1]; c ->[B=$P_B$] a[l1]; x[o,lr=$x$] -- a -- y[o,lr=$y$]; z[o,ll=$z$] -- b -- w[o,ll=$w$]; }} + \val{\ssGraph{ c[tpf] --[R] b[r1]; c ->[B=$Q_B$] a[l1]; x[o,lr=$x$] -- a -- y[o,lr=$y$]; z[o,ll=$z$] -- b -- w[o,ll=$w$]; }},
\end{split}
\end{equation}
where we allow the vertices $x,y$ to agree with $z$ or $w$. 
With formulas, the insertion  in \eqref{sigma cell resolution PQ} means the following identity
\[ 
\sum_{ab} (pf)_a G_{ya} G_{xa} R_{ab}  G_{bw} G_{bz} = \sum_{abc} (pf)_c G_{ya} G_{xa} \big(P_{ac} + Q_{ac} \big) R_{cb}  G_{bw} G_{bz}
 \]
since $P_{ac} + Q_{ac} =\delta_{ac}$.
We first consider with the second graph in \eqref{sigma cell resolution PQ}, whose treatment is independent of the specific weights,
so we already removed the weight information. We insert the $B$ operator as 
\[ 
 \val{\ssGraph{ c --[R] b; c ->[B=$Q_B$] a; x[o,lr=$x$] -- a -- y[o,lr=$y$]; z[o,ll=$z$] -- b -- w[o,ll=$w$]; }} =  
 \val{\ssGraph{ e --[R] b; e ->[B=$Q_B$] d[left=10pt] ->[B=$B^{-1}$] c[left=20pt] ->[B=$B$] a[left=20pt]; x[o,lr=$x$,left=20pt] -- a -- y[o,lr=$y$,left=20pt]; z[o,ll=$z$] -- b -- w[o,ll=$w$]; }} =  
 \val{\ssGraph{ c --[we=$N^{-1}$] b[right=10pt]; c ->[B=$B$] a; x[o,lr=$x$] -- a -- y[o,lr=$y$]; z[o,ll=$z$,right=10pt] -- b -- w[o,ll=$w$,right=10pt]; }}
 \]  
and notice that due to Lemma \ref{lemma B inverse} the matrix $K=(B^{-1})^t Q_B^t R$, 
assigned to the weighted edge in the last graph, is entry-wise $\abs{k_{ab}}\le cN^{-1}$ bounded (the transpositions compensate for the opposite orientation
 of the participating edges). It follows from Lemma \ref{B insertion Lemma} that 
\begin{equation}\label{Q resolution}
\val{\ssGraph{ c --[R] b; c ->[B=$Q_B$] a; x[o,lr=$x$] -- a -- y[o,lr=$y$]; z[o,ll=$z$] -- b -- w[o,ll=$w$]; }}=
\val{\ssGraph{ c --[we=$N^{-1}$] b[right=4pt of c]; c ->[B=$B$] a; x[o,lr=$x$] -- a -- y[o,lr=$y$]; z[o,ll=$z$,right=4pt] -- b -- w[o,ll=$w$,right=4pt]; }} = \sum_{\Gamma'\in \cG_\Gamma} \Val(\Gamma') + \landauO{N^{-p}}, 
\end{equation}
where all $\Gamma'\in\cG_\Gamma$ satisfy $\WEst(\Gamma')\le_p \sigma_q \WEst(\Gamma)$ and all $\sigma$-cells in $\Gamma$ except for the currently expanded one remain $\sigma$-cells in $\Gamma'$. We note that it is legitimate to compare the Ward estimate of $\Gamma'$ with that of $\Gamma$ because with respect to the Ward-estimate there is no difference between $\Gamma$ and the modification of $\Gamma$ in which the $R$-edge is replaced by a generic $N^{-1}$-weighted edge.

We now consider the first graph in \eqref{sigma cell resolution PQ} and repeat the process of inserting projections $P_B'+Q_B'$
to the other side of the $R$-edge to find 
\begin{equation}\label{sigma cell res second PQ}
\val{\ssGraph{ c[tpf] --[R] b[r1]; c ->[B=$P_B$] a[l1]; x[o,lr=$x$] -- a -- y[o,lr=$y$]; z[o,ll=$z$] -- b -- w[o,ll=$w$]; }}= \val{\ssGraph{ c[tpf] --[R] d[t1] ->[B=$P_B'$] b; c ->[B=$P_B$] a; x[o,lr=$x$] -- a -- y[o,lr=$y$]; z[o,ll=$z$] -- b -- w[o,ll=$w$]; }} + \val{\ssGraph{ c --[R] d ->[B=$Q_B'$] b; c ->[B=$P_B$] a; x[o,lr=$x$] -- a -- y[o,lr=$y$]; z[o,ll=$z$] -- b -- w[o,ll=$w$]; }}, 
\end{equation}
where we already neglected those weights which are of no importance to the bound. The argument for the second graph in \eqref{sigma cell res second PQ} is identical to the one we used in \eqref{Q resolution} and we find 
another finite collection of graphs $\cG'_\Gamma$ such that
\begin{equation}
\val{\ssGraph{ c --[R] d ->[B=$Q_B'$] b; c ->[B=$P_B$] a; x[o,lr=$x$] -- a -- y[o,lr=$y$]; z[o,ll=$z$] -- b -- w[o,ll=$w$]; }} = \val{\ssGraph{ a --[we=$N^{-1}$] c[right=4pt] ->[B=$B'$] b[right=4pt]; x[o,lr=$x$] -- a -- y[o,lr=$y$]; z[o,ll=$z$,right=4pt] -- b -- w[o,ll=$w$,right=4pt]; }} = \sum_{\Gamma'\in\cG_\Gamma'} \val{\Gamma'} + \landauO{N^{-p}},
\end{equation}
where the weighted edge carries the weight matrix $K=P_{B}^t R Q_{B'} B'^{-1}$,   which is according to Lemma \ref{lemma B inverse} indeed scales like $\abs{k_{ab}}\le cN^{-1}$. The graphs $\Gamma'\in\cG_\Gamma'$ also satisfy $\WEst(\Gamma')\le_p \sigma_q \WEst(\Gamma)$ and all $\sigma$-cells in $\Gamma$ except for the currently expanded one remain $\sigma$-cells in $\Gamma'$. 

It remains to consider the first graph in \eqref{sigma cell res second PQ}  in the situation where $B$ does not have a bounded inverse. We compute the weight matrix of the $P_B^t R P_B'$ 
interaction edge as  
\[
\begin{split}
  P_B^t \diag(\vp\vf) R P_B' &= \left(\frac{\braket{\overline{\vb^{(B)}},\cdot}}{\braket{\overline{\vb^{(B)}},\overline{\vl^{(B)}}}}\overline{\vl^{(B)}}\right)\left[ \diag(\vp\vf) R \frac{\braket{\vl^{(B')},\cdot}}{\braket{\vl^{(B')},\vb^{(B')}}}\vb^{(B')} \right]\\
  &= \frac{\braket{\vb^{(B)}\vp\vf (R \vb^{(B')}) }}{\braket{\overline{\vb^{(B)}},\overline{\vl^{(B)}}}}\frac{\braket{\vl^{(B')},\cdot}\overline{\vl^{(B)}}}{\braket{\vl^{(B')},\vb^{(B')}}}
\end{split}
\]
which we separate into the scalar factor 
\[\frac{\braket{\vb^{(B)}\vp\vf (R \vb^{(B')}) }\braket{\vl^{(B')},\overline{\vl^{(B)}}}}{\braket{\overline{\vb^{(B)}},\overline{\vl^{(B)}}}\braket{\vl^{(B')},\vb^{(B')}}}\]
and the weighted edge 
\begin{equation}
\label{definition of K}
K=\frac{\braket{\vl^{(B')},\cdot}\overline{\vl^{(B)}}}{\braket{\vl^{(B')},\overline{\vl^{(B)}}}}\end{equation} 
which scales like $\abs{k_{ab}}\le cN^{-1}$ since \(\vl\) is \(\ell^2\)-normalised and delocalised. Thus we can write 
\begin{equation}\label{PP resolution}
\val{\ssGraph{ c[tpf] --[R] d[t1] ->[B=$P_B$] b; c ->[B=$P_B$] a; x[o,lr=$x$] -- a -- y[o,lr=$y$]; z[o,ll=$z$] -- b -- w[o,ll=$w$]; }} = \frac{\braket{\vb^{(B)}\vp\vf (R \vb^{(B')}) }\braket{\vl^{(B')},\overline{\vl^{(B)}}}}{\braket{\overline{\vb^{(B)}},\overline{\vl^{(B)}}}\braket{\vl^{(B')},\vb^{(B')}}} 
\val{\ssGraph{ a --[we=$N^{-1}$] b[right=5pt]; x[o,lr=$x$] -- a -- y[o,lr=$y$]; z[o,ll=$z$,right=5pt] -- b -- w[o,ll=$w$,right=5pt]; }}.
\end{equation}
Note that the $B$ and $B'$ operators are not completely independent: According to Fact \ref{number of edges} it follows that for an interaction edge $e=(u,v)$ associated with the matrix $R$ the number of incoming $G$-edges in $u$ is the same as the number of outgoing $G$-edges from $v$, and vice versa. Thus, according to \eqref{B def}, the $B$-operator at $u$ comes with an $S$ if and only if the $B'$-operator at $v$ comes also with an $S$. Furthermore, if the $B$-operator comes with an $T$, then the $B'$-operator comes with an $T^t$, and vice versa. The distribution of the conjugation operators to $B,B'$ in \eqref{B def}, however, can be arbitrary. We now use the fact that the scalar factor in \eqref{PP resolution} can be estimated by $\abs{\sigma}+\rho+\eta/\rho$ (cf.~Lemma~\ref{lmm:expansion of coefficient}). 
Summarising the above arguments, from \eqref{sigma cell resolution PQ}--\eqref{PP resolution}, the proof of Proposition \ref{sigma cell prop} is complete. 

\section{Cusp universality}\label{sec: univ}
The goal of this section is the proof of cusp universality in the sense of Theorem \ref{thr:cusp universality}. Let $H$ be the original Wigner-type random matrix with expectation $A\defeq \E H$ and variance matrix $S=(s_{ij})$ with  $s_{ij}\defeq\E \abs{h_{ij}-a_{ij}}^2$ and $T=(t_{ij})$ with $t_{ij}\defeq \E (h_{ij}-a_{ij})^2$. We consider the Ornstein Uhlenbeck process $\set{\widetilde H_t|   t\ge 0}$ starting from $\wt H_0=H$, i.e.
\begin{equation} \label{OU flow}
\diff \wt H_t = - \frac{1}{2}(\wt H_t-A)\diff t+\Sigma^{1/2}[\diff B_t], \qquad \Sigma[R]\defeq \E W\Tr (W R)
\end{equation}
which preserves expectation and variance.
In our setting of deformed Wigner-type matrices the covariance operator $\Sigma : \C^{N \times N} \to \C^{N \times N}$ is given by
\[
\Sigma[R] \defeq S \odot R+ T\odot R^t.
\]
 The OU process effectively adds a small Gaussian component to $\wt H_t$ along the flow in the sense that $\tilde H_t=A+e^{-t/2} (H-A) + \wt U_t$ in distribution with $\wt U_t$ being and independent centred Gaussian matrix with covariance $\Cov{\wt U}= (1-e^{-t/2})\Sigma$. Due to the fullness Assumption \ref{fullness} there exist small $c,t_\ast$ such that $\wt U_t$ can be decomposed as $\wt U_t=\sqrt{ct} U+U'_t$ with $U\sim \mathrm{GUE}$ and $U_t'$ Gaussian and independent of $U$ for $t\le t_\ast$. Thus there exists a Wigner-type matrix   $ H_t$  such that 
\begin{equation}\label{repr}
\begin{split}
\wt H_t &= H_t + \sqrt{ct} U, \qquad \SS_t=  \SS - ct\SS^{\mathrm{GUE}}, \qquad \E  H_t = A, \\ 
U&\sim\text{GUE}, \qquad \SS^{\mathrm{GUE}}[R] \defeq \braket{R} = \frac{1}{N}\Tr R
\end{split}
\end{equation} 
with $U$ independent of $ H_t$. Note that we do not define $H_t$ as a stochastic process and we will use
the representation \eqref{repr}  only for one carefully chosen $t=N^{-1/2+\epsilon}$.
We note that $H_t$ satisfies the assumption of our local law from Theorem \ref{thr:Local law}. It thus follows that $G_t\defeq (H_t-z)^{-1}$ 
is well approximated by the solution $M_t=\diag(M_t)$ to the MDE
\begin{equation*}
-M_t^{-1} = z-A + \SS_t[M_t]. \qquad \rho_t(E)\defeq \lim_{\eta\searrow 0}\frac{\Im\braket{M_t(E+i\eta)}}{\pi}.
\end{equation*}
In particular, by setting $t=0$, $M_0$ well approximates the resolvent of the original matrix $H$ and $\rho_0=\rho$ is its self-consistent density.
Note that the Dyson equation of $\wt H_t$ and hence its solution as well  are independent of $t$, since they are entirely determined by the first
and second moments of $\wt H_t$ that are the same  $A$ and $S$ for any $t$. Thus the resolvent of $\wt H_t$
 is well approximated by the same $M_{0}$ and the self-consistent density of $\wt H_t$ is given by $\rho_0=\rho$ for any $t$.
While $H$ and $\wt H_t$ have  identical self-consistent data, structurally they differ in a key point: $\wt H_t$ has a small 
Gaussian component. Thus the  correlation kernel of the local eigenvalue statistics has a contour integral representation
using a version of the Br\'ezin-Hikami formulas, see Section~\ref{sec: contour integral}.

The contour integration analysis requires a Gaussian component of size at least $ct\gg N^{-1/2}$
and  a very precise description of the eigenvalues of $H_t$
just above the scale of the eigenvalue spacing. 
This  information will come from the optimal rigidity, Corollary~\ref{crl:Uniform rigidity}, and the precise
shape of the self-consistent density of states of $H_t$.
 The latter will be analysed in Section \ref{sec: free conv} 
where we describe the evolution of the density near the cusp under  an additive GUE perturbation  $\sqrt{s} U$. 
We need to construct $H_t$ with a small gap carefully so that after a relatively long time $s=ct$  the matrix $H_t+ \sqrt{ct} U$ 
develops a cusp exactly at the right location.
In fact, we the process has two  scales in the shifted variable $\nu = s-ct$ that indicates the time relative to  the cusp formation.
It turns out that
the \emph{locations} of the edges typically  move linearly with $\nu$,
while the \emph{length of the gap} itself scales like $(-\nu)_+^{3/2}$, i.e.~it varies much slower
and we need to fine-tune the evolution of both.

To understand this tuning process, we fix $ t= N^{-1/2+\epsilon}$ and we consider the matrix flow $s\to H_t(s)\defeq H_t + \sqrt{s}U$ for any $s\ge 0$ and
not just for $s=ct$.  It is well known that the corresponding self-consistent densities are given by the 
semicircular flow.
Equivalently,  these densities  can be described by the 
 free convolution of $\rho_t$ with a scaled semicircular distribution $\rho_{\textrm{sc}}$. In short, the self-consistent
 density of $H_t(s)$ is given by   $\rho^\fc_s \defeq \rho_t \boxplus \sqrt s \rho_{\textrm{sc}}$, where  we omitted $t$ from the notation  $\rho^\fc_s$
since  we consider $t$ fixed.
  In particular we have $\rho^\fc_0=\rho_t$, the density of $H_t$  and $\rho^\fc_{ct}=\rho$, the density of  $\wt H_t= H_t+\sqrt{ct} U$ as well as that of $H$.  Hence, as a preparation to the contour integration, in Section~\ref{sec: free conv} we need to  describe the cusp formation along the semicircular flow.
 Before going into details, we describe the  strategy.

Since in the sequel the densities $\rho^\fc_s$ and their local minima and gaps will play an important role, we introduce the convention that properties of the original density $\rho$ will always carry $\rho$ as a superscript for the remainder of Section \ref{sec: univ}. In particular, the points $\cu,\ed_\pm,\mi$ and the gap size $\Delta$ from \eqref{gamma def eqs} and Theorem \ref{thr:cusp universality} will from now on be denoted by $\cu^\rho,\ed_\pm^\rho, \mi^\rho$ and $\Delta^\rho$. In particular a superscript of $\rho$ never denotes a power. 

\subsubsection*{Proof strategy.} First we consider case (i) when   
 $\rho$, the self-consistent density associated with $H$, has an exact cusp
at the point $\cu^\rho\in\R$. Note that  $\cu^\rho$ is also a cusp point of the self-consistent density of $\wt H_t$ for any $t$.

 We set $t\defeq N^{-1/2+\epsilon}$. Define the functions
 \[\Delta(\nu) \defeq(2\gamma)^2 (\nu/3)^{3/2}\qquad\text{and}\qquad \rho^{\min}(\nu)\defeq \gamma^2\sqrt{\nu}/\pi\] 
 for any $\nu\ge 0$. For $s<ct$ denote the gap in the support of $\rho^\fc_s$ close to $\cu^\rho$ by $[\ed_s^-,\ed_s^+]$ and its length by $\Delta_s\defeq \ed_s^+-\ed_s^-$. In Section~\ref{sec: free conv} we will prove that if $\rho$ has an exact cusp in $\cu^\rho$ as in \eqref{gamma def}, then $\rho^\fc_s$ has a gap of size $\Delta_s\approx \Delta(ct-s)$, and, in particular, $\rho_t=\rho^\fc_0$ has a gap of size $\Delta_0\approx\Delta(ct)\sim t^{3/2}$, 
  only depending on $c,t$ and $\gamma$. The distance of $\cu^\rho$ from the gap is $\approx \text{const}\cdot t$. This
 overall shift will be relatively  easy to handle, but notice that it must be tracked very precisely
since the gap changes much slower than its location. For $s>ct$ with $s-ct=\landauo{1}$ we will similarly prove that $\rho^\fc_s$ has no gap anymore close to $\cu^\rho$ but a unique local minimum in $\mi_s$ of size $\rho^\fc_s(\mi_s)\approx\rho^{\min}(s-ct)$.

Now we consider the case where $\rho$ has no exact cusp but a small gap of size $\Delta^\rho>0$. We parametrize this gap length via a parameter $t^\rho>0$ defined by $\Delta^\rho=\Delta(t^\rho)$. It follows from the associativity \eqref{fc assoc} of the free convolution that $\rho_t$ has a gap of size $\Delta_0\approx \Delta(ct+t^\rho)$. 
  
  Finally, the third case is  where $\rho$ has a local minimum of size $\rho(\mi^\rho)$. We parametrize it as $\rho(\mi^\rho)=\rho^{\min}(t^\rho)$ with $0<t^\rho<ct$ then it follows that $\rho_t$ has a gap of size $\Delta_0\approx \Delta(ct-t^\rho)$.  

   Note that these conclusions follow purely from the considerations in Section \ref{sec: free conv} for exact cusps and the associativity of the free convolution. We note that in both almost cusp cases $t^\rho$ should be interpreted as a time (or reverse time) to the cusp formation.

In the final part of the proof in Sections \ref{sec: contour integral}--\ref{sec: contour deformation} we will write the correlation kernel of $ H_t+\sqrt{ct} U$ as a contour integral purely in terms of the mesoscopic shape parameter $\gamma$ and the gap size $\Delta_0$ of the density $\rho_t$ associated with $H_t$. If $\Delta_0\approx\Delta(ct)$, then the gap closes after time $s\approx ct$ and we obtain a
 Pearcey kernel with parameter $\alpha=0$. If
$\Delta_0\approx \Delta(ct+t^\rho)$ and $t^\rho\sim N^{-1/2}$, then the gap does not quite close at time $s=ct$ 
and we obtain a Pearcey kernel with $\alpha>0$, while for $\Delta_0\approx \Delta(ct-t^\rho)$ with $t^\rho\sim N^{-1/2}$ the gap
 after time $s=ct$  is transformed into a tiny local minimum and we obtain a Pearcey kernel with $\alpha<0$. The precise value of $\alpha$ in terms of $\Delta^\rho$ and $\rho(\mi^\rho)$ are given in \eqref{eq pearcey param choice}.
 Note that as an input to the contour integral analysis, 
 in all three cases we use the local law only for $H_t$, i.e.~in a situation when there is a small gap in the support of $\rho_t$, given by $\Delta_0$ defined as above in each case.

\subsection{Free convolution near the cusp}\label{sec: free conv}
In this section we quantitatively investigate the free semi-circular flow before and after the formation of cusp. We first establish the exact rate at which a gap closes to form a cusp, and the rate at which the cusp is transformed into a non-zero local minimum. We now suppose that $\rho^\ast$ is a general density with a small spectral gap $[\ed_-^\ast,\ed_+^\ast]$ whose Stieltjes transform $m^\ast$ can be obtained from solving a Dyson equation. Let $\rho_{\mathrm{sc}}(x)\defeq \sqrt{(4-x^2)_+}/2\pi$ be the density of the semicircular distribution and let $s\ge 0$ be a time parameter. The free semicircular convolution $\rho_s^\fc$ of $\rho^\ast$ with $\sqrt{s}\rho_\mathrm{sc}$ is then defined implicitly via its Stieltjes transform
\begin{subequations}
\begin{equation}\label{free conv def}
m_s^\fc(z) = m^\ast(\xi_s(z))= m^\ast(z+ s m_s^\fc(z)),\qquad \xi_s(z)\defeq z + s m_s^\fc(z), \qquad z,m_s^\fc(z)\in\HC.
\end{equation}
It follows directly from the definition that $s\mapsto m_s^\fc$ is \emph{associative} in the sense that
\begin{equation}\label{fc assoc}
m_{s+s'}^\fc(z) = m_s(z+ s' m_{s+s'}^\fc(z)), \qquad s,s'\ge 0.
\end{equation}
\end{subequations}
Figure \ref{rhos deltas plot} illustrates the quantities in the following lemma. 
We state the lemma for scDOSs from arbitrary data pairs $(A_\ast,\SS_\ast)$ satisfying the conditions in \cite{1804.07752}, i.e.
\begin{equation} \label{general flatness}
\norm{A_\ast}\le C, \qquad c\braket{R} \le \SS_\ast[R]\le C \braket{R} 
\end{equation}
for any self-adjoint $R=R^\ast$ and some constants $c,C>0$.
\begin{figure}
\centering
\begin{subfigure}[t]{.45\textwidth}
\setlength{\figurewidth}{\textwidth}
\setlength{\figureheight}{4cm}
\input{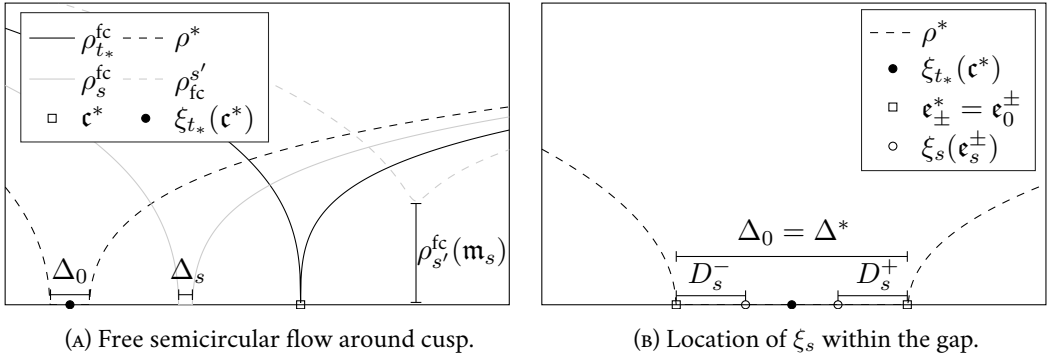}
\caption{Free semicircular flow around cusp.}
\label{rhos deltas plot}
\end{subfigure}  
\begin{subfigure}[t]{.45\textwidth}
\setlength{\figurewidth}{\textwidth}
\setlength{\figureheight}{4cm}
\input{xis}   
\caption{Location of $\xi_s$ within the gap.}
\label{rhos xi plot}
\end{subfigure}  
\caption{Figure \ref{rhos deltas plot} illustrates the evolution of $\rho^\fc_s$ along the semicircular flow at two times $0<s<t_\ast<s'$ before and after the cusp. We recall that $\rho^\ast=\rho^\fc_0$ and $\rho=\rho^\fc_{t_\ast}$. Figure \ref{rhos xi plot} shows the points $\xi_s(\ed_s^\pm)$ as well as their distances to the edges $\ed_0^\pm$.}
\label{rho xi figure} 
\end{figure} 
\begin{lemma}\label{general lemma fc 1}
Let $\rho^\ast$ be the density of a Stieltjes transform $m^\ast=\braket{M_\ast}$ associated with some Dyson equation 
\[-1=(z-A_\ast+\SS_\ast[M_\ast])M_\ast,\]
with $(A_\ast,\SS_\ast)$ satisfying \eqref{general flatness}. Then there exists a small constant $c$, depending only on the constants in Assumptions \eqref{general flatness} such that the following statements hold true. Suppose that $\rho^\ast$ has an initial gap $[\ed_-^\ast,\ed_+^\ast]$ of size $\Delta^\ast = \ed_+^\ast-\ed_-^\ast \le c$. Then there exists some critical time $t_\ast \lesssim (\Delta^\ast)^{2/3}$ such that $m_{t_\ast}^\fc$ has exactly one exact cusp in some point $\cu^\ast$ with $\abs[0]{\cu^\ast-\ed_\pm^\ast}\lesssim t_\ast$, and that $\rho_{t_\ast}^\fc$ is locally around $\cu^\ast$ given by \eqref{gamma def} for some $\gamma>0$. Considering the time evolution $[0,2t_\ast]\ni s\mapsto m_s^\fc$ we then have the following asymptotics.
\begin{subequations} 
\begin{enumerate}[(i)]
\item \emph{After the cusp.}\label{lemma fc 1 min case} For $t_\ast<s\le 2t_\ast$, $\rho_s^\fc$ has a unique non-zero local minimum in some point $\mi_s$ such that
\begin{equation}\label{eq min case}
\rho_s^\fc(\mi_s) = \frac{\sqrt{s-t_\ast}\gamma^2}{\pi}[1+\landauO[0]{(s-t_\ast)^{1/2}}], \qquad \abs{\mi_s-\cu^\ast+(s-t_\ast)\Re m_s^\fc(\mi_s)} \lesssim (s-t_\ast)^{3/2+1/4}.
\end{equation}
Furthermore, $\mi_s$ can approximately be found by solving a simple equation, namely there exists $\wt\mi_s$ such that
\begin{equation}\label{eq min case b}
\wt\mi_s -\cu^\rho +(s-t_\ast )\Re m^\fc_s(\wt\mi_s)=0,\qquad \abs[0]{\mi_s-\wt\mi_s}\lesssim (s-t_\ast)^{3/2+1/4},\qquad \rho^\fc_{s}(\wt\mi_s)\sim \sqrt{s-t_\ast}.
\end{equation}
\item \emph{Before the cusp.} For $0\le s<t_\ast$, the support of $\rho_s^\fc$ has a spectral gap $[\ed_s^-,\ed_s^+]$ of size $\Delta_{s}\defeq \ed_s^+-\ed_s^-$ near $\cu^\ast$ which satisfies
\begin{equation} \label{Delta approx eq}
\Delta_s = (2\gamma)^2 \Big(\frac{t_\ast-s}{3}\Big)^{3/2}[1+\landauO[0]{(t_\ast-s)^{1/3}}].
\end{equation} 
In particular we find that the initial gap $\Delta^\ast=\Delta_0$ is related to $t_\ast$ via $\Delta^\ast= (2\gamma)^2 (t_\ast/3)^{3/2}[1+\landauO[0]{(t_\ast-s)^{1/3}}]$.\label{lemma fc 1 gap case}
\end{enumerate}
\end{subequations}
\end{lemma}
\begin{proof}
Within the proof of the lemma we rely on the extensive shape analysis from \cite{1804.07752}. We are doing so not only for the density $\rho^\ast=\rho_0^\fc$ and its Stieltjes transform, but also for $\rho_s^\fc$ and its Stieltjes transform $m_s^\fc$ for $0\le s\le 2t_\ast$. The results from \cite{1804.07752} also apply here since $m_s^\fc(z)=\braket{M_\ast(\xi_s(z))}$ can also be realized as the solution
\[-M_\ast(\xi_s(z))^{-1} = z + s\braket{M_\ast(\xi_s(z))}-A_\ast + \SS_\ast[M_\ast(\xi_s(z))] = z -A_\ast + (\SS_\ast+s\SS^{\mathrm{GUE}})[M_\ast(\xi_s(z))] \]
to the Dyson equation with perturbed self-energy $\SS_\ast+ s\SS^{\mathrm{GUE}}$. Since $t_\ast\lesssim 1$ it follows that the shape analysis from \cite{1804.07752} also applies to $\rho_s^\fc$ for any $s\in[0,2t_\ast]$. 

We begin with part \eqref{lemma fc 1 min case}. Set $\nu\defeq s-t_\ast$, then for $0\le\nu\le t_\ast$ we want to find $x_\nu$ such that $\Im m^\fc_s$ has a local minimum in $\mi_s\defeq \cu^\ast+x_\nu$ near $\cu^\ast$, i.e.
\[ x_\nu\defeq \argmin_x \Im m^\fc_s(\cu^\ast+x), \qquad \abs{x_\nu}\lesssim \nu.
\]
First we show that $x_\nu$ with these properties exists and is unique by using 
the extensive shape analysis in \cite{1804.07752}. Uniqueness directly follows from \cite[Theorem 7.2(ii)]{1804.07752}.
For the existence, we set 
\[
 a_\nu(x)\defeq \Im m_{\fc}^s(\cu^\ast+x), \quad b_\nu(x) \defeq \Re m^\fc_s(\cu^\ast+x), \quad a_\nu\defeq a_\nu(x_\nu), \quad b_\nu\defeq b_\nu(x_\nu). \]
Set $\delta\defeq K\nu$ with a large constant $K$. Since $a_0(x)=\Im m_{t_\ast}(\cu^\ast+x)\sim \abs{x}^{1/3}$, we have 
$a_0(\pm \delta)\sim \delta^{1/3}$ and $a_0(0)=0$. Recall from \cite[Proposition 10.1(a)]{1804.07752} that the map $s\mapsto m^\fc_s$ is $1/3$-H\"older continuous.
It then follows  that $a_\nu(\pm \delta)\sim \delta^{1/3} + \landauO{\nu^{1/3}}$, while $a_\nu(0)\lesssim\nu^{1/3}$.
Thus $a_\nu$ necessarily has a local minimum in $(-\delta,\delta)$ if $K$ is sufficiently large. This shows the existence of a local minimum with 
$\abs{x_\nu}\lesssim K\nu\sim \nu$.

We now study the function $f_\nu(x)=x+\nu b_\nu(x)$ in a small neighbourhood around $0$. From \cite[Eqs.~(7.62),(5.43)--(5.45)]{1804.07752} it follows that 
\begin{equation} \label{bprime} 
\begin{split}
b_\nu'(x) &= \Re \frac{c_1(x)+\landauO{a_\nu(x)}}{-\ii c_2(x) a_\nu(x) +   a_\nu(x)^2+\landauO{a_\nu(x)^3}} +\landauO{1}\\
&= \frac{c_1(x)}{c_2(x)^2 +   a_\nu(x)^2} + \landauO{\frac{1}{c_2(x)+a_\nu(x)}} 
\end{split}  \end{equation}
whenever $a_\nu(x)\ll 1$, with appropriate real functions\footnote{We have $c_1=\pi/\psi$, $c_2=2\sigma/\psi$ with the notations $\psi, \sigma$ in \cite{1804.07752},
where $\psi\sim 1$ and $\abs{\sigma}\ll 1$ near the almost cusp, but we refrain from using these letters in the present context to avoid confusions.} $c_1(x)\sim 1$ and $ c_2(x)\ge0$. Moreover, $\abs{c_2(0)}\ll 1$ since $\cu^\ast$ is an almost cusp point for $m_s^\fc$ for any $s\in[0,2t_\ast]$. Thus it follows that $b_\nu'(x)>0$ whenever $a_\nu(x)+c_2(x)\ll1$. Due to the $1/3$-H\"older continuity\footnote{See \cite[Lemma 5.5]{1804.07752} for the $1/3$-H\"older continuity of quantities $\psi,\sigma$ in the definition of $c_2$.} of both $a_\nu(x)$ and $c_2(x)$
and $a_\nu(0)+ \abs{c_2(0)}\ll 1$, it follows that $b_\nu'(x)>0$ whenever $\abs{x}\ll 1$. We can thus conclude that $f_\nu$ satisfies $f_\nu'\ge 1$ in some $\landauo{1}$-neighbourhood of $0$. As $\abs{f_\nu(0)}\lesssim\nu$ we can conclude that there exists a root $\wt x_\nu$, $f_\nu(\wt x_\nu)=0$ of size 
$\abs[0]{\wt x_\nu}\lesssim \nu$. With $\wt\mi_s\defeq \cu^\ast+\wt x_\nu$ we have thus shown the first equality in \eqref{eq min case b}. 

Using \eqref{gamma def}, 
we now expand the defining equation
\[ a_\nu(x) = \Im  m_{t_\ast}^\fc(\cu^\ast+x+\nu b_\nu(x) +\ii \nu a_\nu(x) ) \]
for the free convolution in the regime for those $x$ sufficiently close to $\wt x_\nu$ such that $\abs[0]{x+\nu b_\nu(x)}\lesssim \nu a_\nu(x)$ to find 
\[
\begin{split}
a_\nu(x) &=\frac{\sqrt 3 \gamma^{4/3}}{2\pi} \nu a_\nu(x) \int_\R \frac{\abs{\lambda}^{1/3}+\landauO{\abs{\lambda}^{2/3}}}{(\lambda-x-\nu b_\nu(x))^2+ (\nu a_\nu(x))^2}\diff \lambda \\
&= \frac{\sqrt 3 \gamma^{4/3}}{2\pi} \int_\R \frac{ (\nu a_\nu(x))^{1/3} \abs{\lambda}^{1/3}}{(\lambda-[x+\nu b_\nu(x)]/\nu a_\nu(x))^2+1}\diff \lambda + \landauO{(\nu a_\nu(x))^{2/3}}\\
&= (\nu a_\nu(x))^{1/3} \gamma^{4/3} \left[ 1+ \frac{1}{9}\left(\frac{x+\nu b_\nu(x)}{\nu a_\nu(x)}\right)^2 + \landauO{\left(\frac{x+\nu b_\nu(x)}{\nu a_\nu(x)}\right)^4 + (\nu a_\nu(x))^{1/3}}\right],
\end{split}
\]
i.e.
\begin{equation}
a_\nu(x) = \nu^{1/2}\gamma^2 \left[ 1+ \frac{1}{9}\left(\frac{x+\nu b_\nu(x)}{\nu a_\nu(x)}\right)^2 + \landauO{\left(\frac{x+\nu b_\nu(x)}{\nu a_\nu(x)}\right)^4 + (\nu a_\nu(x))^{1/3}}\right]^{3/2}.
\label{vsigma exp2}
\end{equation}
Note that \eqref{vsigma exp2} implies that $\nu a_\nu(\wt x_\nu)\sim \nu^{3/2}$, i.e.~the last claim in \eqref{eq min case b}. We now pick some large $K$ and note that from \eqref{vsigma exp2} it follows that $a_\nu(\wt x_\nu \pm K \nu^{7/4})> a_\nu(\wt x_\nu)$. Thus the interval $[\wt x_\nu - K \nu^{7/4}, \wt x_\nu + K \nu^{7/4}]$ contains a local minimum of $a_\nu(x)$, but by
the uniqueness this must then be $x_\nu$.
We thus have $\abs[0]{x_\nu-\wt x_\nu}\le K \nu^{7/4}$, proving the second claim in \eqref{eq min case b}. By 1/3-H\"older continuity of $a_\nu(x)$ and by $a_\nu(\wt x_\nu)\sim \nu^{1/2}$ from \eqref{vsigma exp2}, we conclude
that $a_\nu = a_\nu(x_\nu)\sim \nu^{1/2}$ as well. Using that $\wt x_\nu + \nu b_\nu(\wt x_\nu)=0$ and $b_\nu'\lesssim 1/\nu$ from \eqref{bprime} and $a_\nu(x)\gtrsim \sqrt{\nu}$, we conclude that $\abs{x_\nu +\nu b_\nu(x_\nu)}\lesssim \nu^{7/4}$, i.e.~the second claim in \eqref{eq min case}. Plugging this information back into \eqref{vsigma exp2}, we thus find $a_\nu=\gamma^2\sqrt\nu(1+\landauO{\nu^{1/2}})$ and have also proven the first claim in \eqref{eq min case}.

We now turn to part \eqref{lemma fc 1 gap case}. It follows from the analysis in \cite{1804.07752} that $\rho^\fc_s$ exhibits either a small gap, a cusp or a small local minimum close to $\cu^\ast$. It follows from \eqref{lemma fc 1 min case} that a cusp is transformed into a local minimum, and a local minimum cannot be transformed into a cusp along the semicircular flow. Therefore it follows that the support of $\rho^\fc_s$ has a gap of size $\Delta_s=\ed_s^+-\ed_s^-$ between the edges $\ed_s^\pm$. Evidently $\ed_{t_\ast}^-=\ed_{t_\ast}^+=\cu^\ast$, $\ed_0^+-\ed_0^-=\Delta_0$, $\ed_0^\pm=\ed_\pm^\ast$ and for $s>0$ we differentiate \eqref{free conv def} to obtain 
\begin{equation}\label{Mp 1/s}
\frac{(m^\fc_s)'(z)}{1+s(m^\fc_s)'(z)} = m_\ast'(z+sm^\fc_s (z))\quad\text{and conclude}\quad m_\ast'(\xi_s(\ed_s^\pm))=1/s
\end{equation}
by considering the $z\to \ed_s^\pm$ limit and the fact that $\rho^\fc_s$ has a square root at edge (for $s<t_\ast$) hence $(m^\fc_s)'$ blows up at this point. Denoting the $\diff/\diff s$ derivative by dot, from 
\begin{equation*}
\frac{\diff}{\diff s} m^\fc_s(\ed_s^\pm) = m_\ast'(\xi_s(\ed_s^\pm))\left(\dot \ed_s^\pm + m^\fc_s(\ed_s^\pm) +s\frac{\diff}{\diff s} m^\fc_s(\ed_s^\pm)  \right) = \frac{\dot \ed_s^\pm + m^\fc_s(\ed_s^\pm)}{s}+ \frac{\diff}{\diff s} m^\fc_s(\ed_s^\pm)
\end{equation*}
we can thus conclude that $\dot \ed_s^\pm=-m^\fc_s(\ed_s^\pm)$. This implies that the gap as a whole moves with linear speed (for non-zero $m^\fc_s(\ed_s^\pm)$), and, in particular, the distance of the gap of $\rho^\ast$ to $\cu^\ast$ is an order of magnitude larger than the size of the gap. It follows that the size $\Delta_{s}\defeq \ed_s^+-\ed_s^-$ of the gap of $\rho^\fc_s$ satisfies 
\[\dot\Delta_{s}=m^\fc_s(\ed_s^-)-m^\fc_s(\ed_s^+)=\int_\R \Big[\frac{1}{x-\ed_s^-}-\frac{1}{x-\ed_s^+}\Big]\rho^\fc_s(x)\diff x=-\Delta_{s}\int_\R \frac{\rho^\fc_s(x)}{(x-\ed_s^-)(x-\ed_s^+)}\diff x.\] 
We now use the precise shape of $\rho^\fc_s$ close to $\ed_s^\pm$ according to \eqref{gamma def edge} which is given by 
\begin{equation}\label{edge shape formula}
\rho^\fc_s(\ed_s^\pm\pm x)=\frac{\sqrt{3}(2\gamma)^{4/3}\Delta_{s}^{1/3}}{2\pi}\left((1+\landauO[0]{(t_\ast-t)^{1/3}})\Psi_{\text{edge}}(x/\Delta_{s})+\landauO{\Delta_s^{1/3}\Psi^2_{\text{edge}}(x/\Delta_{s})}\right),
\end{equation}
where $\Psi_\mathrm{edge}$ defined in \eqref{Psi edge} exhibits the limiting behaviour 
\[\lim_{\Delta\to 0}\Delta^{1/3}\Psi_{\text{edge}}(x/\Delta) = \abs{x}^{1/3}/2^{4/3}.\] Using \eqref{edge shape formula}, we compute
\begin{equation}\label{diff eq Deltas}
\begin{split}
  \dot\Delta_{s}&=-(1+\landauO[0]{(t_\ast-s)^{1/3}})\frac{\sqrt 3(2\gamma)^{4/3}\Delta_{s}^{1/3}}{\pi} \int_0^\infty \frac{\Psi_{\text{edge}}(x)}{x(1+x)}\diff x \\
  &= -\gamma^{4/3}(2\Delta_{s})^{1/3}\left[1+\landauO[0]{(t_\ast-s)^{1/3}+\Delta_{s}^{1/3}}\right],
\end{split}
\end{equation}
where the $(1+\landauO[0]{(t_\ast-s)^{1/3}})$ factor in \eqref{edge shape formula} encapsulates two error terms; both are due to the fact that the shape factor $\gamma_s$ of $\rho_s^\fc$ from \eqref{gamma def edge} is not exactly the same as $\gamma$, i.e.~the one for $s=t_\ast$. To track this error in $\gamma$ we go back to \cite{1804.07752}. First, $\abs{\sigma}$ in \cite[Eq.~(7.5a)]{1804.07752} is of size $(t_\ast-s)^{1/3}$ by the fact that $\sigma$ vanishes at $s=t_\ast$ and is $1/3$-H\"older continuous according to {\cite[Lemma 10.5]{1804.07752}}. Secondly, according to {\cite[Lemma 10.5]{1804.07752}} the shape factor $\Gamma$ (which is directly related to $\gamma$ in the present context) is also $1/3$-H\"older continuous and therefore we know that the shape factors of $\rho^\ast$ at $\ed_0^\pm$ are at most multiplicatively perturbed by a factor of $(1+\landauO[0]{(t_\ast-s)^{1/3}})$. By solving the differential equation \eqref{diff eq Deltas} with the initial condition $\Delta_{t_\ast}=0$, the claim \eqref{Delta approx eq} follows.
\end{proof}

Besides the asymptotic expansion for gap size and local minimum we also require some quantitative control on the location of $\xi_{t_\ast}(\cu^\ast)$, as defined in \eqref{free conv def}, and some slight perturbations thereof within the spectral gap $[\ed_-^\ast,\ed_+^\ast]$ of $\rho^\ast$. We remark the the point $\xi^\ast\defeq\xi_{t_\ast}(\cu^\ast)$ plays a critical role for the contour integration in Section \ref{sec: contour integral} since it will be the critical point of the phase function. 
From \eqref{Delta approx eq} we recall that the gap size scales as $t_\ast^{3/2}$ which makes it natural to compare distances on that scale. In the regime where $t'\ll t_\ast$ all of the following estimates thus identify points very close to the centre of the initial gap.

\begin{lemma}\label{general lemma fc 2}
Suppose that we are in the setting of Lemma \ref{general lemma fc 1}. We then find that $\xi_{t_\ast}(\cu^\ast)$ is very close to the centre of $[\ed_-^\ast,\ed_+^\ast]$ in the sense that
\begin{subequations}
\begin{equation}\label{xi cusp ineq}
\abs[2]{\xi_{t_\ast}(\cu^\ast) - \frac{\ed_+^\ast+\ed_-^\ast}{2}} \lesssim t_\ast^{3/2+1/3}.
\end{equation}
Furthermore, for $0\le t'\le t_\ast$ we have that 
\begin{equation}\label{xi almost cusp ineq}
\begin{split}
  \abs[2]{\xi_{t_\ast-t'}\Big(\frac{\ed_{t_\ast-t'}^++\ed_{t_\ast-t'}^-}{2}\Big) - \frac{\ed_+^\ast+\ed_-^\ast}{2}} &\lesssim t_\ast^{3/2+1/9},\\ 
  \abs[2]{\xi_{t_\ast+t'}\left(\mi_{t_\ast+t'}\right) - \frac{\ed_+^\ast+\ed_-^\ast}{2}} &\lesssim t_\ast^{3/2}(t_\ast^{1/12}+(t'/t_\ast)^{1/2}).
\end{split}
\end{equation}
\end{subequations}
\end{lemma} 
\begin{proof}
We begin with proving \eqref{xi cusp ineq}. For $s<t_\ast$ we denote the distance of $\xi_s(\ed_s^\pm)$ to the edges $\ed_0^\pm$ by $D_s^\pm\defeq \pm(\ed_0^\pm - \xi_s(\ed_s^\pm))$, cf.~Figure \ref{rhos xi plot}. We have, by differentiating $m_\ast'(\xi_s(\ed_s^\pm))=1/s$ from \eqref{Mp 1/s} that 
\begin{equation}\label{Ddot eq}
\dot D_{s}^\pm=\mp \frac{\diff}{\diff s}\xi_s(\ed_s^\pm),\qquad-\frac{1}{s^2}= m''_\ast(\xi_s(\ed_s^\pm)) \frac{\diff }{\diff s}\xi_s(\ed_s^\pm)
\end{equation}
and by differentiating \eqref{free conv def},
\begin{equation*}
(m^\fc_s)'= m_\ast'(\xi_s) \xi_s',\qquad \xi_s'(m^\fc_s)''= m_\ast''(\xi_s)(\xi_s')^3+(m^\fc_s)'\xi_s'',\qquad m_\ast''(\xi_s)=\frac{(m^\fc_s)''}{(1+s (m^\fc_s)')^3}.
\end{equation*}
We now consider $z=\ed_s^\pm+\ii\eta$ with $\eta\to0$ and compute  from \eqref{edge shape formula}, for any $s<t_\ast$, 
\begin{equation*}
\begin{split}
\lim_{\eta\searrow 0}\sqrt\eta(m^\fc_s)'(z)&=\lim_{\eta\searrow 0}\sqrt\eta\int_\R \frac{\rho^\fc_s(x)}{(x-z)^2}\diff x =\lim_{\eta\searrow 0}\frac{\sqrt{3\eta}(2\gamma)^{4/3}\Delta_{s}^{1/3}}{2\pi}\int_0^\infty \frac{\Psi_\text{edge}(x/\Delta_{s})}{(x-\ii\eta)^2}\diff x\\
&= \frac{(2\gamma)^{4/3}}{2\sqrt{3}\Delta_{s}^{1/6}\pi}\int_0^\infty \frac{x^{1/2}}{(x-\ii)^2}\diff x = \frac{(2\gamma)^{4/3}\sqrt \ii}{4\sqrt{3}\Delta_{s}^{1/6}}
\end{split}
\end{equation*}
and
\begin{equation*}
\begin{split}
\lim_{\eta\searrow 0}\eta^{3/2}(m^\fc_s)''(z)&=\lim_{\eta\searrow 0}\eta^{3/2}2\int_\R \frac{\rho^\fc_s(x)}{(x-z)^3}\diff x = \lim_{\eta\searrow 0}\frac{\sqrt 3\eta^{3/2}(2\gamma)^{4/3}\Delta_{s}^{1/3}}{\pi}\int_0^\infty \frac{\Psi_\text{edge}(x/\Delta_{s})}{(x-\ii\eta)^3}\diff x\\
&=\frac{(2\gamma)^{4/3}}{\sqrt{3} \Delta_{s}^{1/6}\pi}\int_0^\infty \frac{x^{1/2}}{(x-\ii)^3}\diff x = \frac{(2\gamma)^{4/3}\ii^{3/2}}{8\sqrt{3}\Delta_{s}^{1/6}}.
\end{split}
\end{equation*}
Here we used that fact that the error terms in \eqref{edge shape formula} become irrelevant in the $\eta\to0$ limit. We conclude, together with \eqref{Ddot eq}, that
\begin{equation*}
\begin{split}
m_\ast''(\xi_s(\ed_s^\pm)) &= \pm\frac{3(2\Delta_{s})^{1/3}}{s^3\gamma^{8/3}},\\
\dot D_s^\pm &=\pm (s^2 m_\ast''(\xi_s(\ed_s^\pm)))^{-1}=\frac{s\gamma^{8/3}}{3(2\Delta_{s})^{1/3}}= \frac{s\gamma^2}{2\sqrt3\sqrt{t_\ast-s}}[1+\landauO[0]{t_\ast^{1/3}}].
\end{split}
\end{equation*}
Since $D_0^-=D_0^+=0$ and $\dot D_s^-\approx \dot D_s^+ $ 
 it follows that, to leading order, $D_{s}^+\approx D_{s}^-$ and more precisely
\[ D_s^\pm = \gamma^2 \frac{2t_\ast^{3/2} -s \sqrt{t_\ast-s}-2t_\ast\sqrt{t_\ast-s}}{3^{3/2}}[1+\landauO[0]{t_\ast^{1/3}}].\]
In particular it follows that $\abs{\ed_0^\pm-\xi_{t_\ast}(\cu^\ast)}= [1+\landauO[0]{t_\ast}^{1/3}]2\gamma^2 t_\ast^{3/2} /3^{3/2}$. Together with the $s=0$ case from \eqref{Delta approx eq} we thus find
\[ \abs[2]{\xi_{t_\ast}(\cu^\ast)-\frac{\ed_+^\ast+\ed_-^\ast}{2}} \lesssim t_\ast^{3/2+1/3}=t_\ast^{11/6},\]
proving \eqref{xi cusp ineq}.

We now turn to the proof of \eqref{xi almost cusp ineq} where we treat the small gap and small non-zero minimum separately. We start with the first inequality. We observe that \eqref{xi cusp ineq} in the setting where $(\rho^\ast,t_\ast)$ are replaced by $(\rho_{t_\ast-t'}^\fc,t')$ implies
\begin{equation}\label{eq xi t'} \abs[2]{\cu^\ast + t' m_{t_\ast}^\fc(\cu^\ast) - \frac{\ed_{t_\ast-t'}^++\ed_{t_\ast-t'}^-}{2}} \le (t')^{11/6}.\end{equation}
Furthermore, we infer from the definition of $\xi$ and the associativity \eqref{fc assoc} of the free convolution that 
\[  \xi_{t_\ast-t'}\Big( \cu^\ast + t' m_{t_\ast}^\fc(\cu^\ast) \Big) = \cu^\ast + t' m_{t_\ast}^\fc(\cu^\ast) + (t_\ast-t') m_{t_\ast-t'}^\fc\Big( \cu^\ast + t' m_{t_\ast}^\fc(\cu^\ast) \Big) = \xi_{t_\ast}(\cu^\ast) 
\]
and can therefore estimate
\[ 
\begin{split}
   &\abs[2]{\xi_{t_\ast-t'}\Big(\frac{\ed_{t_\ast-t'}^++\ed_{t_\ast-t'}^-}{2}\Big) - \xi_{t_\ast}(\cu^\ast) } = \abs[2]{\xi_{t_\ast-t'}\Big(\frac{\ed_{t_\ast-t'}^++\ed_{t_\ast-t'}^-}{2}\Big) - \xi_{t_\ast-t'}\Big( \cu^\ast + t' m_{t_\ast}^\fc(\cu^\ast) \Big) } \\
   & \lesssim (t')^{11/6} + t_\ast (t')^{11/18} \lesssim t_\ast^{29/18},
\end{split}
\]
just as claimed. In the last step we used \eqref{eq xi t'} and the fact that 
\begin{equation}\label{continuity xi}
\abs{\xi_s(a)-\xi_s(b)} \lesssim \abs{a-b} + s \abs{a-b}^{1/3},
\end{equation}
which directly follows from the definition of $\xi$ and the $1/3$-H\"older continuity of $m_s^\fc$. 

Finally, we address the second inequality in \eqref{xi almost cusp ineq} and appeal to Lemma \ref{general lemma fc 1}\eqref{lemma fc 1 min case} to establish the existence of $\wt\mi_{t_\ast+t'}$ such that 
\begin{equation}\label{wt mi cu aux eq}
\cu^\ast-\wt\mi_{t_\ast+t'} = t' \Re m_{t_\ast+t'}^\fc(\wt\mi_{t_\ast+t'}).
\end{equation}
It thus follows from \eqref{eq min case b} that $\abs[0]{\wt\mi_{t_\ast+t'}-\mi_{t_\ast+t'}}\lesssim (t')^{7/4}$ and therefore from \eqref{continuity xi} that 
\[\abs[0]{\xi_{t_\ast+t'}(\wt\mi_{t_\ast+t'})-\xi_{t_\ast+t'}(\mi_{t_\ast+t'})}\lesssim (t')^{7/4}+ t_\ast (t')^{7/12}\lesssim t_\ast^{19/12}.\]
Using \eqref{wt mi cu aux eq} twice, as well as the associativity \eqref{fc assoc} of the free convolution and $\Im m_{t_\ast}^\fc(\cu^\ast)=0$ we then further compute
\begin{equation}\label{xi mtilde c}
\begin{split}
&\xi_{t_\ast+t'}(\wt\mi_{t_\ast+t'})-\xi_{t_\ast }(\cu^\ast) = \wt\mi_{t_\ast+t'} + (t_{\ast}+t') m_{t_\ast+t'}^\fc(\wt\mi_{t_\ast+t'})-\cu^\ast - t_\ast m_{t_\ast}^\fc(\cu^\ast) \\
&= t_\ast \Re \Big[ m_{t_\ast}^\fc(\cu^\ast +\ii t' \Im m_{t_\ast+t'}^\fc(\wt\mi_{t_\ast+t'}))- m_{t_\ast}^\fc(\cu^\ast) \Big] +\ii (t_{\ast}+t')\Im m_{t_\ast+t'}^\fc(\wt\mi_{t_\ast+t'}).
\end{split}
\end{equation}
By H\"older continuity we can, together with \eqref{xi cusp ineq} and $\Im m_{t_\ast+t'}(\wt\mi_{t_\ast+t'}) \sim (t')^{1/2}$ from \eqref{eq min case b}, conclude that 
\begin{equation*}
\begin{split}
\abs[2]{\xi_{t_\ast+t'}\left(\mi_{t_\ast+t'}\right) - \frac{\ed_+^\ast+\ed_-^\ast}{2}} &\lesssim \abs{\xi_{t_\ast+t'}\left(\mi_{t_\ast+t'}\right) - \xi_{t_\ast+t'}\left(\wt\mi_{t_\ast+t'}\right)}+\abs{\xi_{t_\ast+t'}\left(\wt\mi_{t_\ast+t'}\right) - \xi_{t_\ast}\left(\cu^\ast\right)} \\
&\quad+ \abs[2]{\xi_{t_\ast}\left(\cu^\ast\right) - \frac{\ed_+^\ast+\ed_-^\ast}{2}} \\
& \lesssim \big[t_\ast^{7/4}+t_\ast (t_\ast^{7/4})^{1/3}\big] + t_\ast (t')^{1/2}+t_\ast^{11/6} \lesssim t_\ast^{3/2}(t_\ast^{1/12}+ (t'/t_\ast)^{1/2}).
\end{split}
\end{equation*}
In the first term we used \eqref{continuity xi} and the second estimate of \eqref{eq min case b}. In the second term we used \eqref{xi mtilde c} together with $\Im m_{t_\ast+t'}(\wt\mi_{t_\ast+t'})\sim(t')^{1/2}$ from \eqref{eq min case b} and $1/3$-H\"older continuity of $m_{t_\ast}^\fc$. Finally, the last term was already estimated in the exact cusp case, i.e.~in \eqref{xi cusp ineq}. 
\end{proof}

\subsection{Correlation kernel as contour integral}\label{sec: contour integral}
We denote the eigenvalues of $H_t$ by $\lambda_1,\dots,\lambda_N$. Following the work of Br\'ezin and Hikami (see e.g.~\cite[Eq.~(2.14)]{MR1618958} or \cite[Eq.~(3.13)]{MR2662426} for the precise version used in the present context) the correlation kernel of $\wt H_t=H_t+\sqrt{ct} U$ can be written as
\begin{equation*}
\wh K_N^t(u,v) \defeq \frac{N}{(2\pi \ii)^2 ct} \int_\Upsilon\diff z \int_\Gamma \diff w \frac{\exp\left(N\left[ w^2-2vw+v^2 -z^2 +2zu-u^2\right]/2ct\right)}{w-z} \prod_{i}\frac{w-\lambda_i}{z-\lambda_i}, 
\end{equation*}
where $\Upsilon$ is any contour around all $\lambda_i$, and $\Gamma$ is any vertical line not intersecting $\Upsilon$. With this notation, the $k$-point correlation function of the eigenvalues of $\wt H_t$ is given by 
\[ p_k^{(N)}(x_1,\dots,x_k)=\det\Big(\frac{1}{N}\wh K_N^t(x_i,x_j)\Big)_{i,j\in[k]}.\]
Due to the determinantal structure we can freely conjugate $K_N$ with $v\mapsto e^{N(\xi v-v^2/2)/ct}$ for $\xi\defeq\xi_{ct}(\bu)$ to redefine the correlation kernel as 
\begin{equation*}
K_N^t(u,v) \defeq \frac{N}{(2\pi \ii)^2 c t} \int_\Upsilon\diff z \int_\Gamma \diff w \frac{\exp\left(N\left[ w^2-2v(w-\xi) -z^2+2u(z-\xi)\right]/2ct\right)}{w-z} \prod_{i}\frac{w-\lambda_i}{z-\lambda_i}. \label{Brezin Hikami}
\end{equation*}
This redefinition $K_N^t$ does not agree point-wise with the previous definition $\wh K_N^t$, but gives rise to the same determinant, and in particular to the same $k$-point correlation function. Here $\bu$ is the base point chosen in Theorem \ref{thr:cusp universality}. The central result concerning the correlation kernel is the following proposition.
\begin{proposition}\label{prop pearcey kernel}
Under the assumptions of Theorem \ref{thr:cusp universality}, the rescaled correlation kernel 
\begin{equation}\label{wt K def}
\wt K_N^t(x,y)\defeq \frac{1}{N^{3/4}\gamma} K_N^t\left(\bu+\frac{x}{N^{3/4}\gamma},\bu+\frac{y}{N^{3/4}\gamma}\right)
\end{equation}
around the base point $\bu$ chosen in \eqref{eq pearcey param choice} converges uniformly to the Pearcey kernel from \eqref{Pearcey kernel} in the sense that
\[ \abs{\wt K_N^t(x,y)-K_\alpha(x,y)}\le C N^{-c}\]
for $x,y\in[-R,R]$. Here $R$ is an arbitrary large threshold, $c>0$ is some universal constant, $C>0$ is a constant depending only on the model parameters and $R$, and $\alpha$ is chosen according to \eqref{eq pearcey param choice}.
\end{proposition}
\begin{proof}
We now split the contour $\Upsilon$ into two parts, one encircling all eigenvalues $\lambda_i$ to the left of $\xi=\bu+ct\braket{M(\bu)}$, and the other one encircling all eigenvalues $\lambda_i$ to the right of $\xi$, which does not change the value of $K_N^t$. We then move the vertical $\Gamma$ contour so that it crosses the real axis in $\xi$. This does also not change the value $K_N^t$ as the only pole is the one in $z$ for which the residue reads 
\begin{equation*}
\frac{N}{(2\pi\ii)^2 ct}\int_\Upsilon\diff z \exp\left( \frac{N}{ct\gamma}(u-v)(z-\xi)\right)=0.
\end{equation*} 

We now perform a linear change of variables $z\mapsto \xi+\Delta_0 z$, $w\mapsto \xi+ \Delta_0 w$ in \eqref{wt K def} to transform the contours $\Upsilon,\Gamma$ into contours 
\begin{equation}\label{wh Gamma Upsilon} \wh\Gamma \defeq (\Gamma-\xi)/\Delta_0,\qquad \wh\Upsilon \defeq (\Upsilon-\xi)/\Delta_0\end{equation}
to obtain
\begin{equation}\label{tilde Kernel}
\wt K_N^t(x,y)= \frac{N^{1/4}\Delta_0}{(2\pi \ii)^2ct\gamma} \int_{\widehat \Upsilon}\diff z\int_{\widehat\Gamma} \diff w \frac{\exp\left(\Delta_0 N^{1/4} (xz-yw)/ct\gamma + N\Delta_0^2[\wt f(w)-\wt f(z)]/ct\right)}{w-z},
\end{equation}
where 
\begin{equation*}
\wt f(z) \defeq \frac{z^2}{2}-\frac{ct}{\Delta_0^2}\int_\xi^{\xi+\Delta_0 z} \braket{G_t(u)-M_t(\xi)}\diff u.
\end{equation*}
Here $\Delta_0\defeq \ed_0^+-\ed_0^-$ indicates the length of the gap $[\ed_0^-,\ed_0^+]$ in the support of $\rho_t$. From Lemma \ref{general lemma fc 1} with $\rho^\ast=\rho_t$ and $t_\ast=ct$ we infer $\Delta_0\sim t^{3/2}\sim N^{-3/4+3\epsilon/2}$. In order to obtain \eqref{tilde Kernel} we used the relation $\xi-\bu=ct m^\fc_{ct}(\bu)=ct \braket{M_t(\bu+ct m^\fc_{ct}(\bu))}=ct \braket{M_t(\xi)}$.  

We begin by analysing the deterministic variant of $\wt f(z)$, 
\begin{equation*}
f(z)\defeq \frac{z^2}{2}-\frac{ct}{\Delta_0^2}\int_\xi^{\xi+\Delta_0 z} \braket{M_t(u)-M_t(\xi)}\diff u.
\end{equation*}
We separately analyse the large- and small-scale behaviour of $f(z)$.  On the one hand, using the $1/3$-H\"older continuity of $u\mapsto\braket{M_t(u)}$, eq.~\eqref{Delta approx eq} and 
\begin{equation*}
\frac{ct}{\Delta_0^2}\int_\xi^{\xi+\Delta_0 z} \abs{\braket{M_t(u)-M_t(\xi)}}\diff u \lesssim \frac{t(\Delta_0\abs{z})^{4/3}}{\Delta_0^2} \lesssim \abs{z}^{4/3}.
\end{equation*}
we conclude the large-scale asymptotics
\begin{equation}\label{f asymp large}
f(z)=\frac{z^2}{2}+\landauO{\abs{z}^{4/3}}, \qquad \abs{z}\gg 1.
\end{equation}

We now turn to the small-scale $\abs{z}\ll 1$ asymptotics. We first specialize Lemma \ref{general lemma fc 1} and Lemma \ref{general lemma fc 2} to $\rho^\ast=\rho_t$ and collect the necessary conclusions in the following Lemma.
\begin{lemma}\label{lemma Delta xi}
Under the assumptions of Theorem \ref{thr:cusp universality} it follows that $\rho_t$ has a spectral gap $[\ed_0^-,\ed_0^+]$ of size 
\begin{subequations}
\begin{equation}\label{gap size general}
\Delta_0 = \ed_0^+-\ed_0^-= \Delta(ct \pm t^\rho) \left[1+\landauO{t^{1/3}}\right],\quad\text{where}\quad \pm t^\rho\defeq \begin{cases}
0 &\text{in case (i)}\\
3 (\Delta^\rho)^{2/3} /(2\gamma)^{4/3} &\text{in case (ii)}\\
-\pi^2 \rho(\mi^\rho)^2/\gamma^4 &\text{in case (iii)}.
\end{cases}
\end{equation}
Furthermore, in all three cases we have that $\xi$ is is very close to the centre of the gap in the support of $\rho_t$ in the sense that 
\begin{equation}\label{general xi ineq}
\abs{\xi-\frac{\ed_0^++\ed_0^-}{2}} = \landauO{t^{3/2}N^{-\epsilon/2}}.
\end{equation}
\end{subequations}
\end{lemma}
\begin{proof}
We prove \eqref{gap size general}--\eqref{general xi ineq} separately in cases (i), (ii) and (iii).
\begin{enumerate}[(i)]
\item Here \eqref{gap size general} follows directly from \eqref{Delta approx eq} with $\rho^\ast=\rho_t$, $t_\ast=ct$, $s=0$ and $\cu^\ast=\cu^\rho$. Furthermore \eqref{general xi ineq} follows from \eqref{xi cusp ineq} with $\rho^\ast=\rho_t$, $t_\ast=ct$ and $\cu^\ast=\cu^\rho$.
\item We apply \eqref{Delta approx eq} with $\rho^\ast=\rho=\rho_{ct}^\fc$, $t_\ast=t^\rho$, $s=0$ to conclude that $\Delta^\rho = (2\gamma)^2 (t^\rho/3)^{3/2}[1+\landauO[0]{(t^\rho)^{1/3}}]$, and that $\rho_{ct+t^\rho}^\fc$ has an exact cusp in some point $\cu$. Thus \eqref{gap size general} follows from another application of \eqref{Delta approx eq} with $\rho^\ast=\rho_t$, $t_\ast=ct+t^\rho$, $s=0$ and $\cu^\ast=\cu$. Furthermore, \eqref{general xi ineq} follows again from \eqref{xi almost cusp ineq} but this time with $\rho^\ast=\rho_t$, $t_\ast=ct+t^\rho$, $t'=t^\rho$ and $\ed^\pm_{t_\ast-t'}=\ed^\rho_\pm$, and using that $t_\ast^{1/9}\le N^{-\epsilon/2}$ for sufficiently small $\epsilon$.
\item From \eqref{eq min case} with $\rho^\ast=\rho_t$, $t_\ast=ct-t^\rho$, $s=ct$ to conclude $\rho(\mi^\rho)=[1+\landauO[0]{(t^\rho)^{1/2}}]\gamma^2 \sqrt{t^\rho}/\pi$, and that $\rho_{ct-t^\rho}$ has an exact cusp in some point $\cu$. Finally, \eqref{general xi ineq} follows again from \eqref{xi almost cusp ineq} but with $\rho^\ast=\rho_t$, $t_\ast=ct-t^\rho$, $t'=t^\rho$ and $\mi_{t_\ast+t'}=\mi^\rho$, and using $t'/t_\ast\lesssim t^\rho/ct\lesssim N^{-\epsilon}$ and $t_\ast^{1/12}\le N^{-\epsilon/2}$ for sufficiently small $\epsilon$. \qedhere
\end{enumerate}
\end{proof}

Equipped with Lemma \ref{lemma Delta xi} we can now turn to the small scale analysis of $f(z)$ and write out the Stieltjes transform to find 
\begin{equation*}
\begin{split}
  f(z)&=\frac{z^2}{2}-\frac{ct}{\Delta_0^2}\int_{\R} \int_{\xi}^{\xi+\Delta_0 z} \frac{u-\xi}{(x-u)(x-\xi)}\rho_t(x)\diff u\diff x \\
  &= \frac{z^2}{2}-\frac{ct}{\Delta_0}\int_\R \int_{0}^{z} \frac{u}{(x-u)x}\rho_t(\xi+\Delta_0 x)\diff u\diff x.
\end{split}
\end{equation*}
Note that these integrals are not singular since $\rho_{t}(\xi+\Delta_0 x)$ vanishes for $\abs{x}\le 1/2$. We now perform the $u$ integration to find
\begin{equation}\label{f int eq}
f(z)= \frac{z^2}{2}-\frac{ct}{\Delta_0}\int_\R \left[ \log x -\log(x-z)-\frac{z}{x} \right]\rho_t(\xi+\Delta_0 x)\diff x.
\end{equation}
By using the precise shape \eqref{edge shape formula} (with $s=0$) of $\rho_t$ close to the edges $\ed_0^\pm$, and recalling the gap size from \eqref{gap size general} and location of $\xi$ from \eqref{general xi ineq} we can then write
\begin{equation}\label{fg diff}
f(z) = (1+\landauO[0]{t^{1/3}})\wt g(z) + \landauO{\abs{z}^2 t^{1/3}}
\end{equation}
with 
\begin{equation*}
\wt g(z)\defeq \frac{z^2}{2} - \frac{3\sqrt 3}{2\pi (1\pm t^\rho/ct)} \int_\R \left[ \log x -\log(x-z)-\frac{z}{x} \right] \Psi_\textrm{edge}(\abs{x}-1/2) \1_{\abs{x}\ge 1/2}\diff x
\end{equation*}
being the leading order contribution. Here $\pm$ indicates that the formula holds for all three cases (i), (ii) and (iii) simultaneously, where $t^\rho=0$ in case (i). The contribution of the error term in \eqref{edge shape formula} to the integral in \eqref{f int eq} is of order $\landauO[0]{\abs{z}^2 t^{1/2}}$ using that $\log x-\log(x-z)-z/x=\landauO[0]{\abs{z/x}^2}$ and that $\abs{x}\ge 1/2$ on the support of $\rho_t(\xi+\Delta_0 x)$. By the explicit integrals
\begin{equation*}\label{g expl int}
\frac{3\sqrt 3}{2\pi}\int_0^\infty \frac{\Psi_\text{edge}(x)}{(x+1/2)^2}\diff x = \frac{1}{2}, \qquad \frac{3\sqrt 3}{2\pi}\int_0^\infty \frac{\Psi_\text{edge}(x)}{(x+1/2)^4}\diff x = \frac{8}{27}
\end{equation*}
and a Taylor expansion of the logarithm $\log(x-z)$ we find that the quadratic term $z^2/2$ almost cancels and we conclude the small-scale asymptotics 
\begin{equation}\label{f asymp small}
\wt g(z) = \left(\frac{\pm t^\rho}{ct}\frac{z^2}{2} - \frac{4z^4}{27}\right)\Big(1+\landauO{t^\rho/t}\Big)+\landauO{\abs{z}^5}, \qquad \abs{z}\ll 1.
\end{equation}

\subsection{Contour deformations}\label{sec: contour deformation}
We now argue that we can deform the contours $\Upsilon,\Gamma$ and thereby via \eqref{wh Gamma Upsilon} the derived contours $\wh\Upsilon,\wh\Gamma$, in a way which bounds the sign of $\Re g$ away from zero along the contours. Here $g(z)$ is the $N$-independent variant of $\wt g(z)$ given by
\begin{equation}\label{g wt g approx}
\begin{split}
  g(z) \defeq{}& \frac{z^2}{2} - \frac{3\sqrt 3}{2\pi} \int_\R \left[ \log x -\log(x-z)-\frac{z}{x} \right] \Psi_\textrm{edge}(\abs{x}-1/2) \1_{\abs{x}\ge 1/2}\diff x \\
  ={}& \wt g(z) + \landauO{ N^{-\epsilon} \abs{z}^2 }.
\end{split}\end{equation}
The topological aspect of our argument is inspired by the approach in \cite{MR3502605,MR3485343,MR3440796}.
\begin{lemma}\label{contour deform lemma}
For all sufficiently small $\delta>0$ there exists $K=K(\delta)$ such that the following holds true. The contours $\Upsilon,\Gamma$ then can be deformed, without touching $(\supp\rho_t+[-1,1])\setminus\{\xi\}$ or each other, in such a way that the rescaled contours $\wh\Upsilon,\wh\Gamma$ defined in \eqref{wh Gamma Upsilon} satisfy $\Re g\ge K$ on $\wh\Upsilon\cap\{\abs{z}>\delta\}$ and $\Re g\le -K$ on $\wh\Gamma\cap\{\abs{z}>\delta\}$. Furthermore, locally around $0$ the contours can be chosen in such a way that 
\begin{equation} \label{wh contour local}
\begin{split}
  \wh\Gamma\cap \Set{z\in \C|\abs{z}\le \delta} &= (-\ii\delta,\ii\delta), \\ \wh\Upsilon\cap \Set{z\in \C|\abs{z}\le \delta} &=  (- \delta e^{\ii \pi/4},\delta e^{\ii \pi/4})\cup (-\delta e^{-\ii \pi/4},\delta e^{-\ii \pi/4}).
\end{split} \end{equation}
\end{lemma}
\begin{proof}
Just as in \eqref{f asymp small} we have the expansion
\begin{equation}\label{g asymp}
g(z) = - \frac{4z^4}{27}+\landauO{\abs{z}^5}, \qquad \abs{z}\ll 1.
\end{equation}
It thus follows that for some small $\delta>0$, and
\begin{equation*}
\Omega_k^< \defeq \Set{z\in\C|\abs{z}<\delta,\abs[5]{\arg z-\frac{k\pi}{4}}<\delta}
\end{equation*}
we have $\Omega_{\pm 1}^<,\Omega_{\pm 3}^<\subset \Omega_+\defeq \Set{\Re g>0}$ and $\Omega_{0}^<,\Omega_{\pm 2}^<,\Omega_4^<\subset \Omega_-\defeq\Set{\Re g<0}$
in agreement with Figure \ref{small_scale_phase}. For large $z$, however, it also follows from \eqref{f asymp large} together with \eqref{g wt g approx} and \eqref{fg diff} that for some large $R$, and 
\begin{equation*}
\Omega_k^> \defeq \Set{z\in\C|\abs{z}>R, \frac{(k-1)\pi}{4}+\delta<\arg z<\frac{(k+1)\pi}{4}+\delta}
\end{equation*}
we have $\Omega_0^>,\Omega_4^>\subset \Omega_+$ and $\Omega_{\pm 2}^>\subset\Omega_-$, in agreement with Figure \ref{large_scale_phase}. We denote the connected component of $\Omega_\pm$ containing some set $A$ by $\cco(A)$.  

\begin{figure}
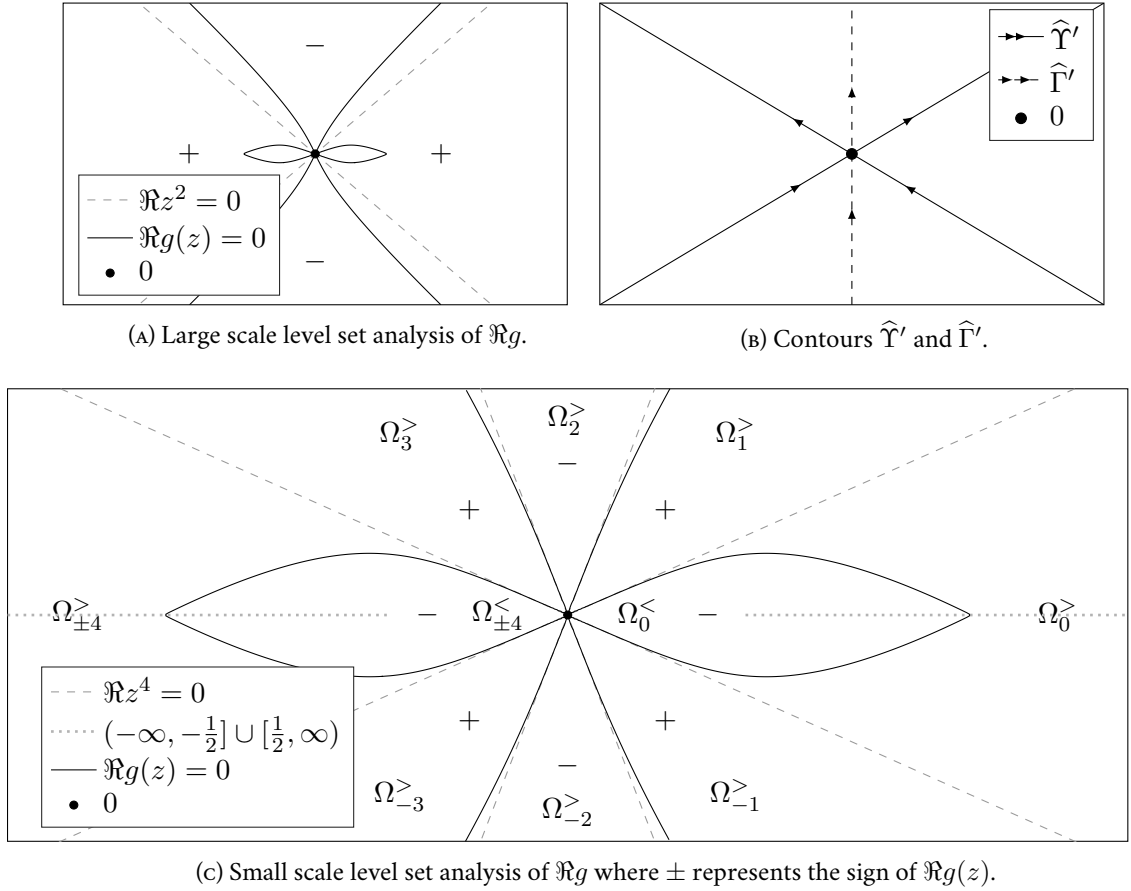

\centering
\begin{subfigure}[t]{.45\textwidth}
\setlength{\figurewidth}{\textwidth}
\setlength{\figureheight}{4cm}
\input{phase_plot_ls}
\caption{Large scale level set analysis of $\Re g$.}
\label{large_scale_phase}
\end{subfigure}
\begin{subfigure}[t]{.45\textwidth}
\setlength{\figurewidth}{\textwidth}
\setlength{\figureheight}{4cm}
\input{gammas}
\caption{Contours $\wh\Upsilon'$ and $\wh\Gamma'$.}
\label{gammas_plot}
\end{subfigure}\\[1em]
\noindent\begin{subfigure}[t]{\textwidth}
\setlength{\figurewidth}{\textwidth}
\setlength{\figureheight}{6cm}
\input{phase_plot}
\caption{Small scale level set analysis of $\Re g$ where $\pm$ represents the sign of $\Re g(z)$.}
\label{small_scale_phase}
\end{subfigure}
\caption{Representative cusp analysis. Figures \ref{small_scale_phase} and \ref{large_scale_phase} show the level set $\Re g(z)=0$. On a small scale $g(z)\sim z^4$, while on a large scale $g(z)\sim z^2$. Figure \ref{gammas_plot} shows the final deformed and rescaled contours $\wh\Upsilon'$ and $\wh\Gamma'$. Figure \ref{small_scale_phase} furthermore shows the cone sections $\Omega_k^>$ and $\Omega_k^{<}$, where we for clarity do not indicate the precise area thresholds given by $\delta$ and $R$. We also do not specifically indicate $\Omega_k^<$ for $k = \pm 1,\pm 2,\pm 3$ as then $\cco(\Omega_k^<)=\cco(\Omega_k^>)$, cf.~Claims 4--5 in the proof of Lemma \ref{contour deform lemma}.}
\label{cusp figure} 
\end{figure} 
\begin{description}
\item[Claim 1 -- $\cco(\Omega_{0}^>),\cco(\Omega_{4}^>)$ are the only two unbounded connected components of $\Omega_+$] Suppose there was another unbounded connected component $A$ of $\Omega_+$. Since $\Omega_{\pm_2}^>\subset\Omega_-$ we would be able to find some $z_0\in A$ with arbitrarily large $\abs{\Re z_0}$. If $\Re z_0>0$, then we note that the map $x\mapsto\Re g(z_0+x)$ is increasing, and otherwise we note that the map $x\mapsto \Re g(z_0-x)$ is increasing. Thus it follows in both cases that the connected component $A$ actually coincides with $\cco(\Omega_0^>)$ or with $\cco(\Omega_4^>)$, respectively.
\item[Claim 2 -- $\cco(\Omega_{\pm 2}^>)$ are the only two unbounded connected components of $\Omega_-$] This follows very similarly to Claim 1.
\item[Claim 3 -- $\cco(\Omega_{\pm 1}^<),\cco(\Omega_{\pm 2}^<),\cco(\Omega_{\pm 3}^<)$ are unbounded]
We note that the map $z\mapsto \Re g(z)$ is harmonic on $\C\setminus([1/2,\infty)\cup(-\infty,-1/2])$ and subharmonic on $\C$. Therefore it follows that $\cco(\Omega^<_{\pm1}),\cco(\Omega^<_{\pm3})\subset \Omega_+$ are unbounded. Since these sets are moreover symmetric with respect to the real axis it then also follows that $\cco(\Omega_{\pm 2})\cap((-\infty,-1/2]\cup [1/2,\infty))=\emptyset$. This implies that $\Re g(z)$ is harmonic on $\cco(\Omega^<_{\pm 2})$ and consequently also that $\cco(\Omega^<_{\pm 2})$ are unbounded.
\item[Claim 4 -- $\cco(\Omega_{1}^<)=\cco(\Omega_{-1}^<)=\cco(\Omega_0^>)$ and $\cco(\Omega_{3}^<)=\cco(\Omega_{-3}^<)=\cco(\Omega_4^>)$] This follows from Claims 1--3. 
\item[Claim 5 -- $\cco(\Omega_2^<)=\cco(\Omega_2^>)$ and $\cco(\Omega_{-2}^<)=\cco(\Omega_{-2}^>)$] This also follows from Claims 1--3.
\end{description}

The claimed bounds on $\Re g$ now follow from Claims 4--5 and compactness. The claimed small scale shape \eqref{wh contour local} follows by construction of the sets $\Omega^<_k$. 
\end{proof}

From Lemma \ref{contour deform lemma} and Lemma \ref{cor no eigenvalues outside} it follows that $K_N^t$ and thereby also $\wt K_N^t$ remain, with overwhelming probability, invariant under the chosen contour deformation. Indeed, $K_N^t$ only has poles where $z=w$ or $z=\lambda_i$ for some $i$. Due to self-adjointness and Lemma \ref{contour deform lemma}, $z=\lambda_i$ can only occur if $\lambda_i=\xi$ or $\dist(\lambda_i,\supp\rho_t)>1$. Both probabilities are exponentially small as a consequence of Lemma \ref{cor no eigenvalues outside}, since for the former we have $\eta_{\mathrm{f}}(\xi)\sim N^{-3/4+\epsilon/6}$ according to \eqref{eta* in gap}, while $\dist(\xi,\supp\rho_t)\sim N^{-3/4+3\epsilon/2}$.

For $z\in\wh\Gamma\cup\wh\Upsilon$ it follows from \eqref{wh contour local} that we can estimate
\begin{equation} \label{f ftilde}\abs[0]{f(z)-\wt f(z)} = \frac{ct}{\Delta_0^2}\abs{\int_\xi^{\xi+\Delta_0 z} \braket{\wt G_t(u)-M_t(u)}\diff u} \prec \frac{t\Delta_0\abs{z}}{Nt^{3/2}\Delta_0^2} \sim \frac{\abs{z}}{Nt^2}=\abs{z}N^{-2\epsilon}.\end{equation}
Indeed, for \eqref{f ftilde} we used \eqref{wh contour local} to obtain $\dist(\Re u,\supp\rho_t)\gtrsim t^{3/2}$, so that $\abs[0]{\braket{ \wt G_t(u)-M_t(u)}}\prec 1/Nt^{3/2}$ follows from the local law from \eqref{average local law inside spectrum}.

We now distinguish three regimes: $\abs{z}\lesssim N^{-\epsilon/2}$, $N^{-\epsilon/2}\lesssim \abs{z}\ll 1$ and finally $\abs{z}\gtrsim1$ which we call microscopic, mesoscopic and macroscopic. We first consider the latter two regimes as they only contribute small error terms. 

\subsubsection*{Macroscopic regime.} If either $\abs{z}\ge\delta$ or $\abs{w}\ge\delta$, it follows from Lemma \ref{contour deform lemma} that $\Re g(w)\le -K$ and/or $\Re g(z)\ge K$, and therefore together with \eqref{fg diff},\eqref{g wt g approx} and \eqref{f ftilde} that $\Re \wt f(w)\lesssim -K$ and/or $\Re \wt f(z)\gtrsim K$ with overwhelming probability. Using $\Delta_0\sim N^{-3/4+3\epsilon/2}$ from \eqref{gap size general}, we find that $N\Delta_0^2/ct\sim N^{2\epsilon}$ and $\Delta_0 N^{1/4}/ct\gamma\sim N^{\epsilon/2}$, so that the integrand in \eqref{tilde Kernel} in the considered regime is exponentially small.

\subsubsection*{Mesoscopic regime.} 
If either $\delta\ge\abs{z}\gg N^{-\epsilon/2}$ or $\delta\ge\abs{w}\gg N^{-\epsilon/2}$, then $\Re g(w)\sim -\abs{w}^4 \ll - N^{-2\epsilon}$ and/or $\Re g(z)\sim \abs{z}^4 \gg N^{-2\epsilon}$ from \eqref{g asymp}. Thus it follows from \eqref{fg diff} and \eqref{g wt g approx} that also $\Re f(w)\ll - N^{-2\epsilon}$ and/or $\Re f(z)\gg N^{-2\epsilon}$ and by \eqref{f ftilde} that with overwhelming probability $\Re \wt f(w)\ll - N^{-2\epsilon}$ and/or $\Re \wt f(z)\gg N^{-2\epsilon}$. Since $1/\abs{w-z}$ is integrable over the contours it thus follows that the contribution to $\wt K_N^t(x,y)$, as in \eqref{tilde Kernel}, from $z,w$ with either $\abs{z}\gg N^{-\epsilon/2}$ or $\abs{w}\gg N^{-\epsilon/2}$ is negligible. 

\subsubsection*{Microscopic regime.} We can now concentrate on the important regime where $\abs{z}+\abs{w}\lesssim N^{-\epsilon/2}$ and to do so perform another change of variables $z\mapsto ct \gamma z/\Delta_0 N^{1/4}\sim N^{-\epsilon/2} z$, $w\mapsto ct \gamma w/\Delta_0 N^{1/4}\sim N^{-\epsilon/2} w$ which gives rise to two new contours 
\[\wh\Gamma'\defeq \frac{\Delta_0 N^{1/4}}{ct\gamma}\wh\Gamma,\qquad \wh\Upsilon' \defeq \frac{\Delta_0 N^{1/4}}{ct\gamma}\wh\Upsilon,\]
as depicted in Figure \ref{gammas_plot}, and the kernel
\begin{equation}\label{Kernel final rescaling}
\wt K_N^t(x,y)= \frac{1}{(2\pi \ii)^2} \int_{\widehat \Upsilon'}\diff z\int_{\widehat\Gamma'} \diff w \frac{\exp\left( xz-yw + \frac{N\Delta_0^2}{ct}[\wt f(\frac{ct\gamma w}{\Delta_0 N^{1/4}})-\wt f(\frac{ct\gamma z}{\Delta_0 N^{1/4}})]\right)}{w-z}.\end{equation}
We only have to consider $w,z$ with $\abs{w}+\abs{z}\lesssim 1$ in \eqref{Kernel final rescaling} since $t/\Delta_0 N^{1/4}\sim N^{-\epsilon/2}$ and the other regime has already been covered in the previous paragraph before the change of variables. 

We now separately estimate the errors stemming from replacing $\wt f(z)$ first by $f(z)$, then by $\wt g(z)$ and finally by $\pm t^\rho z^2/2ct-4z^4/27$. We recall that $\Delta_0\sim t^{3/2}=N^{-3/4+3\epsilon/2}$ from \eqref{gap size general}, $t^\rho\lesssim N^{-1/2}$ from the definition of $t^\rho$ in \eqref{gap size general},  and that $t=N^{-1/2+\epsilon}$ which will be used repeatedly in the following estimates. According to \eqref{f ftilde}, we have 
\begin{subequations}\label{f g error bounds}
\begin{equation} \frac{N\Delta_0^2}{ct}\abs{\wt f\Big(\frac{ct\gamma z}{\Delta_0 N^{1/4}}\Big)-f\Big(\frac{ct\gamma z}{\Delta_0 N^{1/4}}\Big)} \prec \frac{N\Delta_0^2}{t}\frac{t}{\Delta_0 N^{1/4}} N^{-2\epsilon} \abs{z} \lesssim N^{-\epsilon/2}.\end{equation}
Next, from \eqref{fg diff} we have
\begin{equation} \frac{N\Delta_0^2}{ct}\abs{f\Big(\frac{ct\gamma z}{\Delta_0 N^{1/4}}\Big)-\wt g\Big(\frac{ct\gamma z}{\Delta_0 N^{1/4}}\Big)} \lesssim  t^{1/3} \abs{\frac{ct\gamma z}{\Delta_0 N^{1/4}}}^2 \frac{N\Delta_0^2}{ct} + t^{1/3} \frac{N\Delta_0^2}{ct} \lesssim  N^{-1/6+7\epsilon/3}.\end{equation}
Finally, we have to estimate the error from replacing $\wt g(z)$ by its Taylor expansion with \eqref{f asymp small} and find
\begin{equation} \frac{N\Delta_0^2}{ct} \abs{\wt g\Big(\frac{ct\gamma z}{\Delta_0 N^{1/4}}\Big)-\frac{\pm t^\rho}{2ct}\Big(\frac{ct\gamma z}{\Delta_0 N^{1/4}}\Big)^2+\frac{4}{27} \Big(\frac{ct\gamma z}{\Delta_0 N^{1/4}}\Big)^4} \lesssim  N^{-\epsilon/2}. \end{equation}
Finally, from \eqref{gap size general} and the definition of $\alpha$ from \eqref{eq pearcey param choice} we obtain that
\begin{equation} \frac{N\Delta_0^2}{ct}\left[ \frac{\pm t^\rho}{2ct}\left(\frac{ct\gamma z}{\Delta_0 N^{1/4}}\right)^2-\frac{4}{27}\left(\frac{ct\gamma z}{\Delta_0 N^{1/4}}\right)^4\right] =\left(\alpha \frac{z^2}{2}-\frac{z^4}{4}\right)[1+\landauO[0]{t^{1/3}}].\end{equation}
\end{subequations}
From \eqref{f g error bounds} and the integrability of $1/\abs{z-w}$ for small $z,w$ along the contours we can thus conclude
\begin{equation}\label{eq general pearcey}
\wt K_N^t(x,y)= (1+\landauO{N^{-c}})\frac{1}{(2\pi \ii)^2} \int_{\wh\Upsilon'}\diff z\int_{\wt \Gamma'} \diff w \frac{e^{xz-yw + z^4/4 - \alpha z^2/2-w^4/4+\alpha w^2/2}}{w-z}.
\end{equation}
Furthermore, it follows from \eqref{wh contour local} that, as $N\to\infty$, the contours $\wh\Upsilon',\wh\Gamma'$ are those depicted in Figure \ref{gammas_plot}, i.e.
\[\wh\Upsilon'= (- e^{\ii\pi/4}\infty,e^{\ii\pi/4}\infty) \cup (- e^{-\ii\pi/4}\infty,e^{-\ii\pi/4}\infty),\qquad \wh\Gamma'\defeq (-\ii\infty,\ii\infty).\]
We recognize \eqref{eq general pearcey} as the extended Pearcey kernel from \eqref{Pearcey kernel}.

It is easy to see that all error terms along the contour integration are uniform in $x,y$ running over any fixed compact set. This proves that $\wt K_N^t(x,y)$ converges to $K_\alpha(x,y)$ uniformly in $x,y$ in a compact set. This completes the proof of Proposition \ref{prop pearcey kernel}.
\end{proof}

\subsection{Green function comparison}
We will now complete the proof of Theorem \ref{thr:cusp universality} by demonstrating that the local $k$-point correlation function at the common physical cusp
 location $\tau_0$ of the matrices $\wt{H}_t$ does not change along the flow \eqref{OU flow}. Together with Proposition \ref{prop pearcey kernel} this completes the proof of Theorem \ref{thr:cusp universality}. A version of this \emph{continuity of the matrix Ornstein-Uhlenbeck process} with respect to the local correlation functions that is valid in the bulk or at regular edges is the third step in the well known three step approach to universality \cite{MR3699468}. We will present this argument in the more general setup of 
correlated random matrices, i.e.~in the setting of \cite{MR3941370}. In particular, we assume
that the cumulants of the matrix elements $w_{ab}$ satisfy the decay conditions \cite[Assumptions (C,D)]{MR3941370}, an assumption that is obviously fulfilled for deformed Wigner-type matrices. 

We claim that the $k$-point correlation function \smash{$p_k^{(N)}$} of $H=\wt{H}_0$ and the corresponding $k$-point correlation function \smash{$\wt{p}_{k,t}^{(N)}$} of $\wt{H}_t$ stay close along the OU-flow in the sense that 
\begin{equation}\label{eq OU cont} \abs{\int_{\R^k} F(\vx)\left[ N^{k/4} p_k^{(N)}\left( \bu + \frac{\vx}{\gamma N^{3/4}}\right)- \wt{p}_{k,t}^{(N)}\left( \bu + \frac{\vx}{\gamma N^{3/4}}\right) \right] \diff x_1\dots \diff x_k} = \landauO{N^{-c}},\end{equation}
for $\epsilon>0$, $t \le N^{-1/4-\epsilon}$, smooth functions $F$ and some constant $c=c(k,\epsilon)$, where $\bu$ is the physical cusp point. The proof of \eqref{eq OU cont} follows the standard arguments of computing $t$-derivatives of products of traces of resolvents \smash{$\wt{G}^{(t)}=(\wt{H}_t-z)$} at spectral parameters $z$ just below the fluctuation scale of eigenvalues, i.e.~for $\Im z\ge N^{-{\zeta}}\eta_f(\Re z)$. Since the procedure detailed e.g.~in \cite[Chapter 15]{MR3699468} is well established and not specific to the cusp scaling, we keep our explanations brief. 

The only cusp-specific part of the argument is estimating products of random variables
\[ X_t=X_t(x)\defeq N^{1/4} \braket{\Im \wt G^{(t)}(\bu+\gamma^{-1}N^{-3/4}x + \ii N^{-3/4-\zeta})}\] and we claim that 
\begin{equation} \label{eq prod X}\E \biggl[\prod_{j=1}^k X_t(x_j)-\prod_{j=1}^k X_0(x_j)\biggr] \lesssim N^{-c}\end{equation}
as long as $t \le N^{-1/4-\epsilon}$ for some $c=c(k,\epsilon,\zeta)$. For simplicity we first consider $k=1$ and find from It\^o's Lemma that 
\begin{equation} \label{eq ito}\E \frac{\diff X_t}{\diff t} = \E\biggl[ -\frac{1}{2}\sum_\alpha w_\alpha \partial_\alpha X_t + \frac{1}{2}\sum_{\alpha,\beta}\kappa(\alpha,\beta)\partial_\alpha\partial_\beta X_t \biggr],\end{equation}
which we further compute using a standard cumulant expansion, as already done in the bulk regime in \cite[Proof of Corollary 2.6]{MR3941370} and in the edge regime in \cite[Section 4.2]{1804.07744}. We recall that $\kappa(\alpha,\beta)$, and more generally $\kappa(\alpha,\beta_1,\dots,\beta_k)$ denote the joint cumulants of the random variables $w_\alpha,w_\beta$ and $w_\alpha,w_{\beta_1},\dots,w_{\beta_k}$, respectively, which accordingly scale like $N^{-1}$ and $N^{-(k+1)/2}$. Here greek letters $\alpha,\beta \in [N]^2$ are double indices. After cumulant expansion, the leading term in \eqref{eq ito} cancels, and the next order contribution is 
\[ \sum_{\alpha,\beta_1,\beta_2} \kappa(\alpha,\beta_1,\beta_2) \E\bigl[ \partial_\alpha\partial_{\beta_1}\partial_{\beta_2} X_t \bigr],\]
with $N^{-3/2}$ being the size of the cumulant $\kappa(\alpha,\beta_1,\beta_2)$. With $\alpha=(a,b)$ and $\beta_i=(a_i,b_i)$ we then estimate
\[ 
\begin{split}
  &N^{-3/4}\sum_{a,b,c}\sum_{a_1,b_1,a_2,b_2} \abs{\kappa(ab,a_1b_1,a_2b_2)}\E\abs{\wt G_{ca}^{(t)}\wt G^{(t)}_{ba_1}\wt G^{(t)}_{b_1a_2}\wt G^{(t)}_{b_2c}}\\
  &\quad\le N^{-3/4-3/2+2+3/4+{\zeta}} \norm[0]{\Im \wt G^{(t)}}_3 \norm[0]{\wt G^{(t)}}_3^2,
\end{split}\]
where we used the Ward-identity and that $\max_\alpha\sum_{\beta_1,\beta_2}\kappa(\alpha,\beta_1,\beta_2)\lesssim N^{-3/2}$. We now use that according to \cite[Proof of Prop.~5.5]{MR3941370}, $\eta\mapsto\eta\norm[0]{\wt G^{(t)}}_p$ and similarly $\eta\mapsto\eta\norm[0]{\Im \wt G^{(t)}}_p$ are monotonically increasing with $\eta'=N^{-3/4+{\zeta}}$ to find $\norm[0]{\Im\wt G^{(t)}}_p\le_p N^{3{\zeta}-1/4}$ and $\norm[0]{\wt G^{(t)}}_p\le_p N^{3{\zeta}}$ from the local law from Theorem \ref{thr:Local law} and the scaling of $\rho$ at $\eta'$. Since all other error terms can be handled similarly and give an even smaller contribution it follows that 
\begin{equation}\label{dX eq}
\abs{\E\frac{\diff X_t}{\diff t}}\lesssim N^{1/4+{C\zeta}}\quad \text{and similarly, but more generally,}\quad \abs{\E\frac{\diff}{\diff t} \prod_{j=1}^k X_t(x_j)}\lesssim N^{1/4+{Ck \zeta}},
\end{equation}
for some constant $C>0$.
Now \eqref{eq prod X} and therefore \eqref{eq OU cont} follow from \eqref{dX eq} as in \cite[Theorem 15.3]{MR3699468} using the choice $t=N^{-1/2+\epsilon} \le N^{-1/4-\epsilon}$ and choosing $\zeta$ sufficiently small.

\appendix

\section{Technical lemmata}
\begin{lemma} 
\label{lmm:cubic equation}
Let $\C^{N \times N}$ be equipped with a norm $\norm{\cdot}$. Let $\mathcal{A}\colon \C^{N \times N} \times \C^{N \times N}\to \C^{N \times N}$ be a bilinear form and let $\mathcal{B}\colon \C^{N \times N} \to \C^{N \times N}$ a linear operator with a non-degenerate isolated eigenvalue $\beta$. Denote the spectral projection corresponding to $\beta$ by $\mathcal{P}$  and by $\mathcal{Q}$ the one corresponding to the spectral complement of $\beta$, i.e.
\[
\mathcal{P}\defeq  -\lim_{\epsilon \searrow 0}\frac{1}{2\pi \ii}\oint_{\partial B_\epsilon(\beta)}\frac{\dd \omega}{\mathcal{B}-\omega}= \braket{V_\mathrm{l}, \cdot} V_\mathrm{r},\qquad \mathcal{Q}\defeq  1-\mathcal{P},
\]
where $V_\mathrm{r}$ is the eigenmatrix corresponding to $\beta$ and $\braket{V_\mathrm{l}, \cdot}$ a linear functional. Assume that for some positive constant $\lambda>1$ the bounds 
\begin{equation} \label{upper bound in cubic lemma} \norm{\mathcal{A}}+ \norm{\mathcal{B}^{-1}\mathcal{Q}} +  \norm{\braket{V_\mathrm{l}, \cdot}} + \norm{V_\mathrm{r}} \le \lambda, \end{equation}
are satisfied, where we denote the induced norms on linear operators, linear functionals and bilinear forms on $\C^{N \times N}$ by the same symbol $\norm{\cdot}$. Then there exists a universal constant $c>0$ such that for any $\delta \in (0,1)$ and any ${Y}, {X} \in \C^{N \times N}$ with $\norm{Y}+ \norm{{X}}\le c\lambda^{-4}$ that satisfies  the quadratic equation
\begin{equation} \label{quadratic eq} \mathcal{B}[{Y}] - \mathcal{A}[{Y},{Y}]+ {X}= 0, \end{equation}
the following holds: The scalar quantity 
\[
\Theta \defeq  \braket{V_\mathrm{l},Y},
\]
fulfils the cubic equation
\begin{equation} \label{cubic with mu-coefficients} \mu_3 \Theta^3+\mu_2 \Theta^2 + \mu_1 \Theta + \mu_0= \lambda^{12}\landauO{\delta \abs{\Theta}^3+\abs{\Theta}^4 + \delta^{-2}\norm{{X}}^3}, \end{equation}
with coefficients
\begin{equation}\begin{split} \label{coefficients cubic lemma} \mu_3&=\braket{V_\mathrm{l},\mathcal{A}[V_\mathrm{r},\mathcal{B}^{-1}\mathcal{Q}\mathcal{A}[V_\mathrm{r},V_\mathrm{r}]]+\mathcal{A}[\mathcal{B}^{-1}\mathcal{Q}\mathcal{A}[V_\mathrm{r},V_\mathrm{r}],V_\mathrm{r}]}
\\
\mu_2&=\braket{V_\mathrm{l},\mathcal{A}[V_\mathrm{r},V_\mathrm{r}]}
\\
\mu_1&=-\braket{V_\mathrm{l},\mathcal{A}[\mathcal{B}^{-1}\mathcal{Q}[{X}],V_\mathrm{r}]+\mathcal{A}[V_\mathrm{r},\mathcal{B}^{-1}\mathcal{Q}[{X}]]}-\beta
\\
\mu_0&=\braket{V_\mathrm{l},\mathcal{A}[\mathcal{B}^{-1}\mathcal{Q}[{X}],\mathcal{B}^{-1}\mathcal{Q}[{X}]]-{X}}. \end{split}\end{equation}
Furthermore, 
\begin{equation} \label{R leading order} {Y}= \Theta  V_\mathrm{r}- \mathcal{B}^{-1}\mathcal{Q}[{X}] +\Theta^2\mathcal{B}^{-1}\mathcal{Q}\mathcal{A}[V_\mathrm{r},V_\mathrm{r}] +\lambda^7\landauO{\abs{\Theta}^3+\abs{\Theta}\norm{X}+\norm{{X}}^2}. \end{equation}
Here, the constants implicit in the $\ord$-notation depend on $c$ only.
\end{lemma}
\begin{proof}
We decompose ${Y}$ as 
\[
{Y}= {Y}_1+{Y}_2,\qquad {Y}_1 = \Theta  V_\mathrm{r}- \mathcal{B}^{-1}\mathcal{Q}[{X}],\qquad {Y}_2= \mathcal{Q}[{Y}]+ \mathcal{B}^{-1}\mathcal{Q}[{X}].
\]
Then \eqref{quadratic eq} takes the form 
\begin{equation} \label{quadratic with Omega} \Theta \beta  V_\mathrm{r} + \mathcal{P}[{X}]+\mathcal{B}\mathcal{Q}[{Y}_2]=  \mathcal{A}[{Y},{Y}]. \end{equation}
We project both sides with $\mathcal{Q}$, invert $\mathcal{B}$ and take the norm to conclude 
\[
\norm{{Y}_2}= \lambda^2\ord(\norm{{Y}_1}^2 + \norm{{Y}_2}^2),
\]
Then we use the smallness of ${Y}_2$ by properly choosing $\delta$ and the definition of ${Y}_1$ to infer 
${Y}_2= \lambda^4 \ord_2$, where we introduced the notation
\[
\ord_k = \ord(\abs{\Theta}^k + \norm{{X}}^k).
\]
Inserting this information back into \eqref{quadratic with Omega} and using $\abs{\Theta} +\norm{X}=\ord(\lambda^{-3})$ reveals 
\begin{equation} \label{R2 expansion} {Y}_2= \mathcal{B}^{-1}\mathcal{Q}\mathcal{A}[{Y}_1,{Y}_1] + \lambda^7\ord_3. \end{equation} 
In particular, \eqref{R leading order} follows.
Plugging \eqref{R2 expansion} into \eqref{quadratic with Omega} and applying the projection $\mathcal{P}$ yields
\begin{equation*}
\begin{split}
\Theta \beta  V_\mathrm{r} + \mathcal{P}[{X}]&= \mathcal{P}\Big[ \mathcal{A}[{Y}_1,{Y}_1]+ \mathcal{A}[{Y}_1,{Y}_2]+ \mathcal{A}[{Y}_2,{Y}_1] \Big] + \lambda^{11}\ord_4
\\
&= \mathcal{P}\Big[ \mathcal{A}[{Y}_1,{Y}_1]+ \mathcal{A}[{Y}_1,\mathcal{B}^{-1}\mathcal{Q}\mathcal{A}[{Y}_1,{Y}_1]]+ \mathcal{A}[\mathcal{B}^{-1}\mathcal{Q}\mathcal{A}[{Y}_1,{Y}_1],{Y}_1] \Big] + \lambda^{11}\ord_4.
\end{split}
\end{equation*}
For a linear operator $\mathcal{K}_1$ and a bilinear form $\mathcal{K}_2$ with $\norm{\mathcal{K}_1}+\norm{\mathcal{K}_2}\le 1$ we use the general bounds
\[
\Theta \mathcal{K}_2[{R},{R}]\le \delta\Theta^3 +\delta^{-1/2} \norm{{R}}^{3},\qquad \Theta^2\mathcal{K}_1[{R}] \le \delta\Theta^3 + \delta^{-2}\norm{{R}}^3,
\]
for any $R \in \C^{N \times N}$ and $\delta >0$ to find
\begin{equation*}\begin{split} 
\Theta \beta  V_\mathrm{r} + \mathcal{P}[{X}]
&=  \mathcal{P}\Big[ \mathcal{A}[\Theta  V_\mathrm{r}- \mathcal{B}^{-1}\mathcal{Q}[{X}],\Theta  V_\mathrm{r}- \mathcal{B}^{-1}\mathcal{Q}[{X}]]+ \Theta^3\mathcal{A}[V_\mathrm{r},\mathcal{B}^{-1}\mathcal{Q}\mathcal{A}[V_\mathrm{r},V_\mathrm{r}]]\\
&\qquad\qquad+ \Theta^3\mathcal{A}[\mathcal{B}^{-1}\mathcal{Q}\mathcal{A}[V_\mathrm{r},V_\mathrm{r}],V_\mathrm{r}] \Big] 
\\
& \qquad +\lambda^8  \ord\big(\delta \abs{\Theta}^3+\lambda^3 \abs{\Theta}^4+ \delta^{-2}\norm{{X}}^3\big), \end{split}\end{equation*}
which proves \eqref{cubic with mu-coefficients}. 
\end{proof}
\begin{proof}[Proof of Lemma \ref{lemma tilde xi}]
Due to the asymptotics $\Psi_{\mathrm{edge}}\sim \min\{\lambda^{1/2},\lambda^{1/3}\}$ and $\Psi_{\mathrm{min}}\sim\min\{\lambda^2,\abs{\lambda}^{1/3}\}$ and the classification of singularities in \eqref{gamma def eqs}, we can infer the following behaviour of the self-consistent fluctuation scale from Definition~\ref{def:sc fluctuation scale}. There exists a constant $c>0$ only depending on the model parameters such that we have the following asymptotics. First of all, in the spectral bulk we trivially have that $\eta_\mathrm{f}(\tau)\sim N^{-1}$ as long as $\tau$ is at least a distance of $c>0$ away from local minima of $\rho$. In the remaining cases we use the explicit shape formulae from \eqref{gamma def eqs} to compute $\eta_\mathrm{f}$ directly from Definition \ref{def:sc fluctuation scale}. 
\begin{subequations}
\label{eta* behavior}
\begin{enumerate}[(a)]
\item \emph{Non-zero local minimum or cusp.} Let $\tau$ be the location of a non-zero local minimum $\rho(\tau)=\rho_0>0$ or a cusp $\rho(\tau)=\rho_0=0$. Then 
\begin{equation}\label{eta* at cusp or minimum}  \eta_{\mathrm{f}}(\tau +\omega)\sim 
\begin{cases}1/(N\max\{\rho_0, \abs{\omega}^{1/3}\}) , & \max\{\rho_0, \abs{\omega}^{1/3}\} > N^{-1/4},\\
N^{-3/4} , &   \max\{\rho_0, \abs{\omega}^{1/3}\} \le N^{-1/4},
\end{cases}\end{equation}
for $\omega \in (-c,c)$.
\item \emph{Edge.} Let $\tau=\ed_\pm$ be the position of a left/right edge at a gap in $\supp \rho \cap(\ed_\pm-\kappa,\ed_\pm + \kappa)$ of size $\Delta \in (0, \kappa]$ (cf.~\eqref{gamma def edge}). Then 
\begin{equation} \label{eta* at edge} \eta_{\mathrm{f}}(\ed_\pm \pm \omega)\sim 
\begin{cases}
N^{-3/4}, & \omega \le \Delta \le N^{-3/4},
\\
\Delta^{1/6}/\omega^{1/2}N, &\Delta^{1/9}/N^{2/3} < \omega \le \Delta,
\\
\Delta^{1/9}/N^{2/3}, & \omega \le \Delta^{1/9}/N^{2/3},\; \Delta > N^{-3/4},
\\
N^{-3/4}, &  \Delta < \omega \le N^{-3/4},
\\
1/\omega^{1/3}N, & \omega \ge N^{-3/4},\; \omega > \Delta,
\end{cases} \end{equation}
for $\omega \in [0,c)$.
\end{enumerate}
\end{subequations}
The claimed bounds in Lemma \ref{lemma tilde xi} now follow directly from \eqref{xi tilde def} and \eqref{eta* behavior} by distinguishing the respective regimes.
\end{proof}
\begin{proof}[Proof of Lemma \ref{lemma ward moment}]
We start from \eqref{value def} and estimate all vertex weights $\bm w^{(v)}$, interaction matrices $R^{(e)}$ and weight matrices $K^{(e)}$ trivially by 
\[ 
\abs[0]{w^{(v)}_a}\le C, \qquad \abs[0]{r^{(e)}_{ab}}\le C N^{-\deg(e)/2},\qquad \abs[0]{k^{(e)}_{ab}}\le C N^{-l(e)}, \qquad \forall a,b
  \]
 to obtain 
\[ 
\abs{\Val(\Gamma)} \le C^{\abs{V}+\abs{\IE}+\abs{\WeE}} N^{n(\Gamma)-\abs{V}} \norm[3]{\bigg(\prod_{v\in V} \sum_{a_v\in J}\bigg) \prod_{e\in\GE} G_e }_1.
 \]
We now choose the vertex ordering $V=\{v_1,\dots,v_m\}$ as in Lemma \ref{lemma equiv coloring deg}. In the first step we partition the set of $G$-edges into three parts $\GE=E_1\cup E_2\cup E_3$: the edges not adjacent to $v_m$, $E_1=\GE\setminus N(v_m)$, the non-Wardable edges adjacent to $v_m$, $E_2=\GE\cap N(v_m)\setminus\WE$ and the Wardable edges adjacent to $v_m$, $E_3=\WE\cap N(v_m)$. By the choice of ordering it holds that $\abs{E_3}\le 2$. We introduce the shorthand notation $G_{E_i}=\prod_{e\in E_i} G_e$ and use the general H\"older inequality for any collection of random variables $\{ X_A\}$ and $\{Y_A\}$ indexed by some arbitrary index set $\mathcal{A}$ 
\[
\norm[3]{\sum_{A\in\mathcal A} \abs{X_A Y_A}}_q \le \norm[3]{\sum_{A\in\mathcal A}\abs{X_A}}_{q_1} \abs{\mathcal A}^{1/q_2} \max_{A\in\mathcal A} \norm{Y_A}_{q_2}, \qquad \frac{1}{q}=\frac{1}{q_1}+\frac{1}{q_2}
\]
to compute
\begin{equation*}
\begin{split}
&\norm[3]{\sum_{a_{v_1},\dots,a_{v_{m-1}}} \abs{G_{E_1}}\sum_{a_{v_m}} \abs{G_{E_2} G_{E_3}} }_q \\
&\qquad\le N^{(m-1)/q_2} \norm[3]{\sum_{a_{v_1},\dots,a_{v_{m-1}}} \abs{G_{E_1}}}_{q_1}\max_{a_1,\dots,a_{v_{m-1}}} \Bigg( \norm[3]{\sum_{a_{v_m}}\abs{G_{E_3}}}_{2q_2} N^{1/2q_2} \max_{a_{v_m}} \norm{G_{E_2}}_{2q_2} \Bigg),
\end{split}
 \end{equation*}
where we choose $1/q=1/q_1 + 1/q_2$ in such a way that $q_2\ge p/c\epsilon$. Since $\abs{E_3}\le 2$ we can use \eqref{ward identity} to estimate 
\[\norm[3]{\sum_{a_{v_m}} \abs{G_{E_3}} }_{2q_2}\le N (\psi_{2q_2}')^{\abs{E_3}}\le N (\psi+\psi_{2q_2}')^{\abs{E_3}}  \]
and it thus follows from 
\[ 
\norm{G_{E_2}}_{2q_2} \le \prod_{e\in E_2} \norm{G_e}_{2\abs{E_2}q_2} = \norm{G-M}_{2\abs{E_2}q_2}^{\abs{E_2\cap\GE_{g-m}}} \norm{G}_{2\abs{E_2}q_2}^{\abs{E_2\setminus\GE_{g-m}}}
 \]
that 
\begin{equation}\begin{split}\label{Holder induction}
&\norm[3]{\sum_{a_{v_1},\dots,a_{v_{m-1}}} \abs{G_{E_1}}\sum_{a_{v_m}} \abs{G_{E_2} G_{E_3}} }_q \\
&\quad\le N^{\epsilon/c} \norm[3]{\sum_{a_{v_1},\dots,a_{v_{m-1}}} \abs{G_{E_1}}}_{q_1} N(\psi+\psi_{q'}')^{\abs{E_3}} (\psi+\psi_{q'}'+\psi_{q'}'')^{\abs{E_2\cap \GE_{g-m}}}(1+\norm{G}_{q'})^{\abs{E_2}} \end{split}
\end{equation}
for $q'\ge 2q_2\abs{\GE}$. By using \eqref{Holder induction} inductively  $m=\abs{V}\le cp$ times it thus follows that 
\[ 
\norm[3]{\bigg(\prod_{v\in V} \sum_{a_v\in J}\bigg) \prod_{e\in\GE} G_e }_1 \le N^{p\epsilon} N^{\abs{V}} (\psi+\psi_{q'}')^{\abs{\WE}} (\psi+\psi_{q'}'+\psi_{q'}'')^{\abs{\GE_{g-m}}} \big(1+\norm{G}_{q'}\big)^{\abs{\GE}},
 \]
proving the lemma. 
\end{proof}

\begin{lemma} 
\label{lmm:expansion of coefficient}
For the coefficient in \eqref{PP resolution} we have the expansion
\begin{equation} \label{coef exp eq}
\frac{\braket{\vb^{(B)}\vp\vf(R\vb^{(B')})}\braket{\vl^{(B')},\overline{\vl^{(B)}}}}{\braket{\overline{\vb^{(B)}},\overline{\vl^{(B)}}}\braket{\vl^{(B')},\vb^{(B')}}}= c \sigma\norm{F}
\braket{\abs{\vm}^{-2}\vf^2}+ \ord(\rho+{\eta/\rho}), \end{equation}
for some $\abs{c}\sim 1$, provided $\lVert B^{-1}\rVert_{\infty \to\infty} \ge C$ for some  large enough constant $C>0$.
\end{lemma}
\begin{proof}Recall from the explanation after \eqref{PP resolution} that $R'=S,T,T^t$ if $R=S,T^t,T$, respectively. As we saw in the proof of Lemma~\ref{lemma B inverse}, in the case $R=T,T^t$  in the complex Hermitian symmetry class, the operator $B$ as well as $B'$ has a bounded inverse. Since we assume that $\norm[0]{B^{-1}}_{\infty\to\infty}$ is large, we have $R=R'=S$, which also includes the real symmetric symmetry class. In particular, we also have $\norm[0]{(B')^{-1}}_{\infty\to\infty}\ge C$ and all subsequent statements hold simultaneously for $B$ and $B'$. We call $\vf^{(S)}$ the normalised eigenvector corresponding to the eigenvalue with largest modulus of $F^{(S)}\defeq \abs{M}S\abs{M}$, recalling $M=\diag(\vm)$.  Since $B=\abs{M}(1-F^{(S)} + \ord(\rho))\abs{M}^{-1}$ we can use perturbation theory of $F^{(S)}$ to analyse spectral properties of $B$. In particular, we find
\begin{equation}\begin{split} \label{perturbation of Bs} \vb^{(B)}&= \abs{M}\vf^{(S)} + \ord(\rho),\qquad \vl^{(B)}=  \abs{M}^{-1}\vf^{(S)}+ \ord(\rho),
\\
B^{-1}Q_B&=\abs{M}\big(1-F^{(S)} \big)^{-1}(1-P_{\vf^{(S)}})\abs{M}^{-1}+ \ord(\rho), \end{split}\end{equation}
where $P_{\vf^{(S)}}$ is the orthogonal projection onto the $\vf^{(S)}$ direction. The error terms are measured in $\norm{\cdot}_\infty$-norm. For the expansions \eqref{perturbation of Bs} we used that $F$ has a spectral gap in the sense that 
\[
\Spec(F^{(S)} / \norm[0]{F^{(S)}}) \subseteq [-1+c,1-c] \cup\{1\},
\]
for some  constant $c>0$, depending only on model parameters. By using \eqref{perturbation of Bs} we see that the lhs.~of \eqref{coef exp eq} becomes $\pm\braket{(\vf^{(S)})^2 \vp \vf}\norm[0]{F^{(S)}}\braket{\abs{\vm}^{-2}(\vf^{(S)})^2}+\landauO{\rho}$. To complete the proof of the Lemma we note that $\vf^{(S)}=\vf/\norm{\vf}+\landauO{\eta/\rho}$ according to \cite[Eq.~(5.10)]{1804.07752}.
\end{proof}

\printbibliography

\end{document}

%% file: drawGraphs.tex
\usepackage{ifdraft}
\usepackage{tikz,pgfplots}
\usepackage{etoolbox,listings,xstring,wasysym}

\makeatletter

\tikzset{
  fl/.style = {path fading=fade l},
  fr/.style = {path fading=fade r},
  wh/.style = {draw=none,fill=none},
  Gc/.style = {draw=none,circle split, inner sep=0pt,minimum size=8pt,rotate=90,path picture={\draw[pattern=#1] (0,0.07) circle (1.5pt); }},
  Gd/.style = {draw=none,circle split, inner sep=0pt,minimum size=8pt,rotate=270,path picture={\draw[pattern=#1] (0,0.07) circle (1.5pt); }},
  Gf/.style={draw=none,circle split, inner sep=1pt,minimum size=8pt,rotate=90},
  G/.style={circle,draw,minimum size=8pt,inner sep=1pt,font=\tiny},
  xx/.style={circle,fill,draw,inner sep=0pt,minimum size=3pt},
  ab/.style={circle,fill,draw=none,inner sep=0pt,minimum size=0pt,as=},
  Gend/.style={inner sep=0pt,minimum size=3pt},
  r/.style={draw=red},
  o/.style={draw=black,fill=white},
  lb/.style args={#1}{label=below:#1},
  lt/.style args={#1}{label=above:#1},
  lr/.style args={#1}{label=right:#1},
  ll/.style args={#1}{label=left:#1},
  g/.style={postaction={decorate,decoration={
        markings,
        mark=at position .7 with {\arrow[#1]{Stealth[sep=-3pt]}}
      }}},
  s/.style={densely dashed,postaction={decorate,decoration={
        markings,
        mark=at position .7 with {\arrow[#1]{Stealth[sep=-3pt]}}
      }}},
  R/.style ={draw=lightgray, thick,densely dotted,edge node={node[above=-7pt] {\color{gray} $R$ }},anchor=south,pos=0.5,postaction={decoration={
        markings,
        mark=at position .7 with {\arrow[#1]{Stealth[sep=-3pt]}}
      },decorate}},
we/.style args ={#1}{draw=lightgray, thick,densely dotted,edge node={node[above=-7pt] {\color{gray} #1 }},anchor=south,pos=0.5},
IEg/.style ={draw=lightgray, thick,densely dotted},
  S/.style ={draw=lightgray, thick,densely dotted,edge node={node[above=-7pt] {\color{gray}$S$}},anchor=south,pos=0.5},  
  T/.style ={draw=lightgray, thick,densely dotted,edge node={node[above=-7pt] {\color{gray}$T$}},anchor=south,pos=0.5,postaction={decoration={
        markings,
        mark=at position .7 with {\arrow[#1]{Stealth[sep=-3pt]}}
      },decorate}},
  rpd/.style ={draw=lightgray,thick,densely dotted,edge node={node[above=-7pt] {\color{red}$\partial$}},anchor=south,pos=0.5,postaction={decoration={
        markings,
        mark=at position .7 with {\arrow[#1]{Stealth[sep=-3pt]}}
      },decorate}},
  pd/.style ={draw=lightgray,thick,densely dotted,edge node={node[above=-7pt] {\color{gray}$\partial$}},anchor=south,pos=0.5,postaction={decoration={
        markings,
        mark=at position .7 with {\arrow[#1]{Stealth[sep=-3pt]}}
      },decorate}},
  pdr/.style ={g,draw=lightgray,thick,densely dotted,edge node={node[above=-7pt] {\color{red}$\partial$}},anchor=south,pos=0.5},
  pdkr/.style args={#1}{g,draw=lightgray,thick,densely dotted,edge node={node[above=-7pt] { \color{gray}$\partial^{#1}$\color{red}$\partial$}},anchor=south,pos=0.5},
  pdh/.style ={g,draw=lightgray,thick,densely dotted,edge node={node[above=-7pt] {\color{gray}$\widehat\partial$}},anchor=south,pos=0.5},
  pdk/.style args={#1}{g,draw=lightgray,thick,densely dotted,edge node={node[above=-7pt] {\color{gray}$\partial^{#1}$}},anchor=south,pos=0.5},
  eq/.style = {double,postaction={decoration={name=none}}},
  inl/.style args={#1}{initial,initial where=left, initial text=#1,initial distance=10pt},
  pdn/.style args={#1}{g,draw=lightgray,thick,densely dotted,edge node={node[above=-7pt] {\color{gray}$\partial^{#1}$}},anchor=south,pos=0.5},
  inr/.style args={#1}{initial,initial where=right, initial text=#1,initial distance=10pt},
  int/.style args={#1}{initial,initial where=above, initial text=#1,initial distance=10pt},
  inb/.style args={#1}{initial,initial where=below, initial text=#1,initial distance=10pt},
  lpf/.style args={#1}{initial,initial where=left, initial text=$\vp\vf$,initial distance=10pt},
  tpf/.style args={#1}{initial,initial where=above, initial text=$\vp\vf$,initial distance=10pt},
  bpf/.style args={#1}{initial,initial where=below, initial text=$\vp\vf$,initial distance=10pt},
  rpf/.style args={#1}{initial,initial where=right, initial text=$\vp\vf$,initial distance=10pt},
  l1/.style args={#1}{initial,initial where=left, initial text=$\bm 1$,initial distance=10pt},
  t1/.style args={#1}{initial,initial where=above, initial text=$\bm 1$,initial distance=10pt},
  b1/.style args={#1}{initial,initial where=below, initial text=$\bm 1$,initial distance=10pt},
  r1/.style args={#1}{initial,initial where=right, initial text=$\bm 1$,initial distance=10pt},
  lm/.style args={#1}{initial,initial where=left, initial text=$\vm$,initial distance=10pt},
  tm/.style args={#1}{initial,initial where=above, initial text=$\vm$,initial distance=10pt},
  bm/.style args={#1}{initial,initial where=below, initial text=$\vm$,initial distance=10pt},
  rm/.style args={#1}{initial,initial where=right, initial text=$\vm$,initial distance=10pt},
  inr/.style args={#1}{initial,initial where=right, initial text=#1,initial distance=10pt},
  B/.style args={#1}{thick,draw=lightgray,decorate,decoration={snake,amplitude=.4mm,segment length=.8mm,post length=1.4mm},edge node={node[above=-7pt] {\color{gray} #1 }},anchor=south,pos=0.5},
  br/.style = {bend right},
  b0/.style = {bend left=0},
  bl/.style = {bend left},
  glb/.style = {looseness=20,in =220, out=320},
  gll/.style = {looseness=20,in =130, out=230},
  glr/.style = {looseness=20,in =310, out=50},
  glt/.style = {looseness=20,in =40, out=140},
  gm/.style={postaction={decorate,decoration={
        markings,
        mark=at position .3 with {\arrow[#1]{Diamond[open,sep=-3pt,width=5pt]}}
      }}},
}

\newcommand\sGraph[1]{
\begin{tikzpicture}[grow=right,baseline={([yshift=-2pt]current bounding box.center)},font=\footnotesize,>=Stealth]
  \graph[simple necklace layout,components go down left aligned,nodes={draw,circle,as=,minimum size=3pt,inner sep=0pt,fill},node distance = 40pt,component sep=25pt,edges=g]{ #1  };
  \end{tikzpicture}
}

\newcommand\ssGraph[1]{
\begin{tikzpicture}[subgraph text none,grow=right,baseline={([yshift=-2pt]current bounding box.center)},font=\footnotesize,>=Stealth]
  \graph[spring electrical layout,components go down left aligned,nodes={draw,circle,as=,minimum size=3pt,inner sep=0pt,fill},node distance = 15pt,component sep=15pt]{ #1  };
  \end{tikzpicture}
}

\tikzset{circle split part fill/.style  args={#1}{%
 alias=tmp@name, 
  postaction={%
    insert path={
     \pgfextra{%
     \pgfpointdiff{\pgfpointanchor{\pgf@node@name}{center}}%
                  {\pgfpointanchor{\pgf@node@name}{east}}%
     \pgfmathsetmacro\insiderad{\pgf@x}
      \fill[white,fill opacity=0] (\pgf@node@name.base) ([xshift=-\pgflinewidth]\pgf@node@name.east) arc
                          (0:180:\insiderad-\pgflinewidth)--cycle;
      \fill[fill=white,preaction={fill, white},pattern=#1] (\pgf@node@name.base) ([xshift=\pgflinewidth]\pgf@node@name.west)  arc
                           (180:360:\insiderad-\pgflinewidth)--cycle;   
      \draw[line width=0.4pt] (\pgf@node@name.base) ([xshift=\pgflinewidth]\pgf@node@name.west)  arc
                           (180:360:\insiderad-\pgflinewidth)--cycle;                                
         }}}}}  
\makeatother  
\makeatletter
\tikzset{my loop/.style =  {to path={
  \pgfextra{}
  [looseness=6,min distance=4mm]
  \tikz@to@curve@path},font=\sffamily\small
  }}  
\makeatletter

\csedef{pat0}{}
\csedef{pat1}{north east lines}
\csedef{pat2}{crosshatch dots}
\csedef{pat3}{crosshatch}
\csedef{pat4}{grid}

\definecolor{col0}{HTML}{FFFFFF}
\definecolor{col1}{HTML}{A2B969}
\definecolor{col2}{HTML}{EBCB38}
\definecolor{col3}{HTML}{0D95BC}
\definecolor{col4}{HTML}{063951}
\definecolor{col5}{HTML}{F36F13}
\definecolor{col6}{HTML}{C13018}
\definecolor{lightgray}{HTML}{CCCCCC}

\newcommand\csum[1]{%
\sum_{\forcsvlist{\createColorCircle@item}{#1}}
}

\newcommand\createColorCircle@item[1]{
\StrDel{#1}{h}[\colNum]
\newif\ifhalf
\IfSubStr{#1}{h}{\halftrue}{\halffalse}
{\color{col\colNum}\ifhalf\circ\else\bullet\fi}
}

\newcommand\plotLambda[1]{
  \def\mgraphspecs{}
  \foreach \ll [count = \countl] in {#1} {
    \csedef{firstCol}{col0}
    \def\graphspecs{}
    \StrCount{\ll}{,}[\graphlen]
    \ifnum\graphlen>1 
    \foreach \node [count = \g] in \ll {
      \StrDel{\node}{m}[\nodenumber]
      \StrDel{\nodenumber}{.}[\nodenumber]
      \global\csedef{tempCol}{col\nodenumber}
      \global\csdef{tempPat}{\col{1}}
      \ifnum\g=1
        \global\csedef{firstCol}{\tempCol}
        \IfSubStr{\node}{.}{\global\csedef{open}{y}}{\global\csedef{open}{n}}
        \IfStrEqCase{\open}{
          {y}{\xappto\graphspecs{\countl\g[as=,draw=none] }}
          {n}{\xappto\graphspecs{\countl\g[as=,preaction={fill, white},fill=\tempCol,pattern=\csname pat\nodenumber\endcsname] }}
        }
      \else
        \IfBeginWith{\node}{m}{\global\csedef{secondArrow}{{<[sep=-3pt,length=8pt]}}}{\global\csedef{secondArrow}{}}%
        \ifnum\g=2
          \IfStrEqCase{\open}{
            {y}{\xappto\graphspecs{ --[dotted,thick] \countl\g[as=,preaction={fill, white},fill=\tempCol,pattern=\csname pat\nodenumber\endcsname] }}
            {n}{\xappto\graphspecs{ --[\firstArrow-\secondArrow] \countl\g[as=,preaction={fill, white},fill=\tempCol,pattern=\csname pat\nodenumber\endcsname] }}
          }
        \else
          \xappto\graphspecs{ --[\firstArrow-\secondArrow] \countl\g[as=,preaction={fill, white},fill=\tempCol,pattern=\csname pat\nodenumber\endcsname] }%
        \fi
      \fi
      \IfEndWith{\node}{m}{\global\csedef{firstArrow}{{>[sep=-3pt,length=8pt]}}}{\global\csedef{firstArrow}{}}
      \global\csedef{prevTempCol}{\tempCol} 
    }
    \IfStrEqCase{\open}{
      {n}{\IfBeginWith{\ll}{m}{\global\csedef{secondArrow}{{<[sep=-3pt,length=8pt]}}}{\global\csedef{secondArrow}{}}
          \xappto\graphspecs{ --[\firstArrow-\secondArrow,decorate,decoration={snake,amplitude=.3mm,segment length=.6mm}] \countl1; } }
      {y}{\xappto\graphspecs{ --[dotted,thick] \countl0[draw=none,as=] ; } }
    }
    \else
      \ifnum\graphlen=0
        \StrDel{\ll}{m}[\nodenumber]
        \csedef{tempCol}{col\nodenumber}
        \IfSubStr{\ll}{m}{\csedef{firstArrow}{{>[sep=-3pt,length=8pt]}}}{\csedef{firstArrow}{}}
        \def\graphspecs{ \countl1[as=,preaction={fill, white},fill = \tempCol,pattern=\csname pat\nodenumber\endcsname] --[my loop, decorate,decoration={snake,amplitude=.3mm,segment length=.6mm},\firstArrow-] \countl1; }
      \else
        \IfSubStr{\ll}{.}{
          \StrDel{\ll}{.}[\nodenumber]
          \StrDel{\nodenumber}{,}[\nodenumber]
          \csedef{tempCol}{col\nodenumber}
          \def\graphspecs{ \countl0[draw=none,as=,orient = left] --[dotted,thick] \countl1[as=,fill = \tempCol,pattern=\csname pat\nodenumber\endcsname] --[dotted,thick] \countl2[draw=none,as=,nudge down=10pt]; }
        }{
          \StrBefore{\ll}{,}[\nodeOne]
          \StrBehind{\ll}{,}[\nodeTwo]
          \StrDel{\nodeOne}{m}[\nodeOneNumber]
          \StrDel{\nodeTwo}{m}[\nodeTwoNumber]
          \def\tempColOne{col\nodeOneNumber}
          \def\tempColTwo{col\nodeTwoNumber}
          \IfEndWith{\nodeOne}{m}{\global\csedef{firstArrowOne}{{>[sep=-3pt,length=8pt]}}}{\global\csedef{firstArrowOne}{}}
          \IfEndWith{\nodeTwo}{m}{\global\csedef{firstArrowTwo}{{<[sep=-3pt,length=8pt]}}}{\global\csedef{firstArrowTwo}{}}
          \IfBeginWith{\nodeOne}{m}{\global\csedef{secondArrowOne}{{>[sep=-3pt,length=8pt]}}}{\global\csedef{secondArrowOne}{}}
          \IfBeginWith{\nodeTwo}{m}{\global\csedef{secondArrowTwo}{{<[sep=-3pt,length=8pt]}}}{\global\csedef{secondArrowTwo}{}}
          \def\graphspecs{ \countl1[as=,preaction={fill, white},fill = \tempColOne,pattern=\csname pat\nodeOneNumber\endcsname] --[bend right,decorate,decoration={snake,amplitude=.3mm,segment length=.6mm}, \firstArrowOne-\secondArrowTwo ] \countl2[as=,preaction={fill, white},fill = \tempColTwo,pattern=\csname pat\nodeTwoNumber\endcsname]; \countl1 --[bend left,\secondArrowOne-\firstArrowTwo] \countl2;}
        }
      \fi
    \fi
    \xappto\mgraphspecs{ \graphspecs }
  }
  \xdef\mgraphspecs{\noexpand\graph[simple necklace layout,componentwise,component packing=skyline,components go right center aligned,orient=0,nodes=G]{ \mgraphspecs }}
  \begin{tikzpicture}[baseline={([yshift=-2pt]current bounding box.center)},font=\tiny,>=Stealth, node distance = 15pt,node sep=12pt,component sep=5pt]
    \mgraphspecs;
  \end{tikzpicture}%
}
\newcommand\plotLambdas[1]{
  \begin{tikzpicture}[baseline={([yshift=-2pt]current bounding box.center)},font=\tiny,>=Stealth, node distance = 5pt,node sep=12pt,component sep=5pt]
   \graph[simple necklace layout,componentwise,component packing=skyline,components go right center aligned,orient=0,nodes=G] { #1; };
  \end{tikzpicture}
}
\newcommand\pB[1]{
  \StrDel{#1}{h}[\patNum]
  \IfSubStr{#1}{h}{
  \begin{tikzpicture}[baseline={([yshift=-2pt]current bounding box.center)},font=\tiny,>=Stealth, node distance = 0pt,node sep=1pt,component sep=1pt]
   \graph[simple necklace layout,componentwise,component packing=skyline,components go right center aligned,orient=0,nodes=G] { 1[Gf,rotate=270,circle split part fill={\csname pat\patNum\endcsname},as=,minimum size=6pt]; };
  \end{tikzpicture}
  }{
  \begin{tikzpicture}[baseline={([yshift=-2pt]current bounding box.center)},font=\tiny,>=Stealth, node distance = 0pt,node sep=1pt,component sep=1pt]
   \graph[simple necklace layout,componentwise,component packing=skyline,components go right center aligned,orient=0,nodes=G] { 1[pattern=\csname pat\patNum\endcsname,as=,minimum size=6pt]; };
  \end{tikzpicture}
  }
}
\newcommand\plotlLambda[1]{
  \def\mgraphspecs{}
  \foreach \ll [count = \countl] in {#1} {
    \csedef{firstCol}{col0}
    \csedef{openEnd}{no}
    \def\graphspecs{}
    \StrCount{\ll}{,}[\graphlen]
    \foreach \node [count = \g] in \ll {
      \StrDel{\node}{m}[\nodenumber]
      \csedef{type}{n}
      \csedef{marking}{no}
      \IfSubStr{\node}{x}{\csedef{type}{x}}{}
      \IfSubStr{\node}{c}{\csedef{type}{c}}{}
      \IfSubStr{\node}{d}{\csedef{type}{d}}{}
      \IfSubStr{\node}{.}{\csedef{type}{open}}{}
      \IfBeginWith{\node}{m}{\csedef{firstMarking}{yes}\csedef{marking}{yes}}{\csedef{firstMarking}{no}}
      \IfEndWith{\node}{m}{\csedef{secondMarking}{yes}\csedef{marking}{yes}}{\csedef{secondMarking}{no}}
      \StrDel{\nodenumber}{.}[\nodenumber]
      \StrDel{\nodenumber}{c}[\nodenumber]
      \StrDel{\nodenumber}{d}[\nodenumber]
      \StrDel{\nodenumber}{x}[\nodenumber]
      \global\csedef{tempCol}{col\nodenumber}
      \ifnum\g=1
        \global\csedef{firstCol}{\tempCol}
        \IfStrEq{\type}{open}{
          \xappto\graphspecs{ \countl0[xx,as=,grow right,draw=none,fill=white] --[dotted,thick] \countl1[fill=white,preaction={fill, white},pattern=\csname pat\nodenumber\endcsname,as=]}
        }{
        \IfSubStr{\node}{x}{
          \xappto\graphspecs{ \countl0[xx,as=,grow right,fill=lightgray,draw=none] --[color=lightgray]\countl1[Gf,circle split part fill={\csname pat\nodenumber\endcsname},as=]}}
        {
          \IfStrEq{\marking}{yes}
          {
            \xappto\graphspecs{ \countl0[xx,as=,grow right] --[-{<[fill=lightgray,color=black,sep=-3pt,length=8pt]}] \countl1[fill=white,preaction={fill, white},pattern=\csname pat\nodenumber\endcsname,as=]}
          }{
            \xappto\graphspecs{ \countl0[xx,as=,grow right] -- \countl1[fill=white,preaction={fill, white},pattern=\csname pat\nodenumber\endcsname,as=]}
          }
        }}
      \else
        \IfStrEq{\firstMarking}{yes}{\csedef{secondArrow}{{<[sep=-3pt,length=8pt]}}}{\csedef{secondArrow}{}}
        \IfStrEqCase{\type}{
          {n}{\xappto\graphspecs{ --[\firstArrow-\secondArrow] \countl\g[fill=white,preaction={fill, white},pattern=\csname pat\nodenumber\endcsname,as=] }}%
          {c}{\xappto\graphspecs{ --[\firstArrow-\secondArrow] \countl\g[Gc=\csname pat\nodenumber\endcsname,rotate=180,circle split part fill={\csname pat\nodenumber\endcsname},as=] }}%
          {d}{\xappto\graphspecs{ --[\firstArrow-\secondArrow] \countl\g[Gd=\csname pat\nodenumber\endcsname,rotate=180,circle split part fill={\csname pat\nodenumber\endcsname},as=] }}%
          {open}{\global\csedef{openEnd}{yes}}%
        }
      \fi
      \IfStrEq{\secondMarking}{yes}{
        \global\csedef{firstArrow}{{>[sep=-3pt,length=8pt]}}%
      }{
        \global\csedef{firstArrow}{}%
      }
      \global\csedef{prevTempCol}{\tempCol} 
    }
    \IfStrEqCase{\openEnd}{
      {yes}{\xappto\mgraphspecs{ \graphspecs --[dotted,thick] \countl17[Gend,as=,draw=none]; }}%
      {no}{\xappto\mgraphspecs{ \graphspecs -- \countl17[Gend,as=]; }}%
    }%
  }
  \xdef\mgraphspecs{\noexpand\graph[tree layout,componentwise,component packing=skyline,components go down left aligned,nodes=G]{ \mgraphspecs }}
  \begin{tikzpicture}[grow=right,baseline={([yshift=-2pt]current bounding box.center)},font=\tiny,>=Stealth, node distance = 5pt,node sep=12pt,component sep=5pt]
  \mgraphspecs;
  \end{tikzpicture}%
}

%% file: xis.tex
\begin{tikzpicture}

\begin{axis}[%
width=0.951\figurewidth,
height=\figureheight,
at={(0\figurewidth,0\figureheight)},
scale only axis,
xmin=1.663,
xmax=1.675,
xtick={\empty},
ymin=0,
ymax=0.05,
ytick={\empty},
axis background/.style={fill=white},
legend style={legend cell align=left, align=left, draw=white!15!black}
]
\addplot [color=black, dashed]
  table[row sep=crcr]{%
1.663	0.0257911102590932\\
1.66302404809619	0.0256990109654146\\
1.66304809619238	0.0256065138390255\\
1.66307214428858	0.0255136142931417\\
1.66309619238477	0.0254203076530901\\
1.66312024048096	0.0253265891539877\\
1.66314428857715	0.0252324539385217\\
1.66316833667335	0.0251378970541979\\
1.66319238476954	0.0250429134505194\\
1.66321643286573	0.0249474979768735\\
1.66324048096192	0.0248516453791297\\
1.66326452905812	0.0247553502966495\\
1.66328857715431	0.0246586072599295\\
1.6633126252505	0.0245614106865828\\
1.66333667334669	0.0244637548788486\\
1.66336072144289	0.0243656340195713\\
1.66338476953908	0.0242670421688814\\
1.66340881763527	0.0241679732606786\\
1.66343286573146	0.0240684210986728\\
1.66345691382766	0.0239683793523582\\
1.66348096192385	0.0238678415531753\\
1.66350501002004	0.0237668010897605\\
1.66352905811623	0.0236652512041468\\
1.66355310621242	0.023563184986595\\
1.66357715430862	0.02346059537052\\
1.66360120240481	0.0233574751290274\\
1.663625250501	0.023253816867568\\
1.66364929859719	0.0231496130200647\\
1.66367334669339	0.0230448558427872\\
1.66369739478958	0.0229395374082515\\
1.66372144288577	0.0228336495992236\\
1.66374549098196	0.0227271841025205\\
1.66376953907816	0.0226201324018253\\
1.66379358717435	0.0225124857712416\\
1.66381763527054	0.0224042352671828\\
1.66384168336673	0.0222953717213166\\
1.66386573146293	0.0221858857323259\\
1.66388977955912	0.0220757676574559\\
1.66391382765531	0.0219650076036496\\
1.6639378757515	0.021853595418119\\
1.6639619238477	0.0217415206789451\\
1.66398597194389	0.0216287726847639\\
1.66401002004008	0.0215153404441155\\
1.66403406813627	0.0214012126642917\\
1.66405811623247	0.0212863777394663\\
1.66408216432866	0.0211708237392773\\
1.66410621242485	0.0210545383946677\\
1.66413026052104	0.0209375090852649\\
1.66415430861723	0.0208197228249201\\
1.66417835671343	0.0207011662466831\\
1.66420240480962	0.0205818255867021\\
1.66422645290581	0.0204616866680964\\
1.664250501002	0.0203407348826335\\
1.6642745490982	0.0202189551720172\\
1.66429859719439	0.0200963320079755\\
1.66432264529058	0.0199728493714715\\
1.66434669338677	0.0198484907302269\\
1.66437074148297	0.0197232390150755\\
1.66439478957916	0.0195970765956916\\
1.66441883767535	0.0194699852537729\\
1.66444288577154	0.0193419461563142\\
1.66446693386774	0.0192129398252389\\
1.66449098196393	0.0190829461073775\\
1.66451503006012	0.0189519441407411\\
1.66453907815631	0.0188199123191681\\
1.66456312625251	0.0186868282539873\\
1.6645871743487	0.0185526687337585\\
1.66461122244489	0.0184174096791644\\
1.66463527054108	0.0182810260972364\\
1.66465931863727	0.018143492030568\\
1.66468336673347	0.0180047805051457\\
1.66470741482966	0.0178648634737664\\
1.66473146292585	0.0177237117570352\\
1.66475551102204	0.0175812949786301\\
1.66477955911824	0.0174375814964734\\
1.66480360721443	0.0172925383263371\\
1.66482765531062	0.0171461310584065\\
1.66485170340681	0.0169983237671983\\
1.66487575150301	0.0168490789136143\\
1.6648997995992	0.0166983572414715\\
1.66492384769539	0.0165461176672359\\
1.66494789579158	0.0163923171638279\\
1.66497194388778	0.0162369106316417\\
1.66499599198397	0.0160798507546314\\
1.66502004008016	0.0159210878395167\\
1.66504408817635	0.0157605696356474\\
1.66506813627255	0.0155982411432621\\
1.66509218436874	0.0154340444108056\\
1.66511623246493	0.0152679183207685\\
1.66514028056112	0.0150997983502442\\
1.66516432865731	0.0149296162902645\\
1.66518837675351	0.0147572999234166\\
1.6652124248497	0.014582772670565\\
1.66523647294589	0.0144059532279315\\
1.66526052104208	0.0142267551822282\\
1.66528456913828	0.014045086553349\\
1.66530861723447	0.0138608492310204\\
1.66533266533066	0.0136739383489824\\
1.66535671342685	0.0134842416588262\\
1.66538076152305	0.013291638852613\\
1.66540480961924	0.0130960006738718\\
1.66542885771543	0.0128971878444262\\
1.66545290581162	0.0126950500204601\\
1.66547695390782	0.0124894246564051\\
1.66550100200401	0.0122801353896386\\
1.6655250501002	0.0120669902466973\\
1.66554909819639	0.0118497799521719\\
1.66557314629259	0.0116282754480492\\
1.66559719438878	0.0114022249953956\\
1.66562124248497	0.0111713514002274\\
1.66564529058116	0.0109353476476197\\
1.66566933867735	0.010693872661095\\
1.66569338677355	0.0104465455081255\\
1.66571743486974	0.0101929386230658\\
1.66574148296593	0.00993256931835811\\
1.66576553106212	0.00966488949524457\\
1.66578957915832	0.00938927200294131\\
1.66581362725451	0.00910499426909048\\
1.6658376753507	0.00881121623114321\\
1.66586172344689	0.00850695164562793\\
1.66588577154309	0.00819102986478854\\
1.66590981963928	0.00786204352840562\\
1.66593386773547	0.00751827551327092\\
1.66595791583166	0.00715759371154841\\
1.66598196392786	0.00677729567579974\\
1.66600601202405	0.00637386818056258\\
1.66603006012024	0.00594259934717353\\
1.66605410821643	0.00547691292633579\\
1.66607815631263	0.00496713401019835\\
1.66610220440882	0.00439795221834069\\
1.66612625250501	0.00374235523543849\\
1.6661503006012	0.00294314334638833\\
1.66617434869739	0.0018200406068567\\
1.66619839679359	7.86113163878356e-06\\
1.66622244488978	4.16644322250337e-06\\
1.66624649298597	3.19238267463158e-06\\
1.66627054108216	2.69254526412672e-06\\
1.66629458917836	2.37598848145366e-06\\
1.66631863727455	2.15270035338697e-06\\
1.66634268537074	1.98444534928447e-06\\
1.66636673346693	1.85185330975127e-06\\
1.66639078156313	1.74392489426676e-06\\
1.66641482965932	1.65388994646755e-06\\
1.66643887775551	1.57732298800646e-06\\
1.6664629258517	1.51119339720052e-06\\
1.6664869739479	1.45334628659205e-06\\
1.66651102204409	1.40220253226374e-06\\
1.66653507014028	1.35657432149863e-06\\
1.66655911823647	1.31554936464068e-06\\
1.66658316633267	1.27841344207756e-06\\
1.66660721442886	1.24459848502188e-06\\
1.66663126252505	1.21364612950961e-06\\
1.66665531062124	1.18518224997084e-06\\
1.66667935871743	1.15889764170213e-06\\
1.66670340681363	1.13453438386282e-06\\
1.66672745490982	1.11187568749451e-06\\
1.66675150300601	1.09073709367432e-06\\
1.6667755511022	1.07096132142799e-06\\
1.6667995991984	1.05241320469214e-06\\
1.66682364729459	1.0349751643242e-06\\
1.66684769539078	1.01854522044436e-06\\
1.66687174348697	1.00303413012861e-06\\
1.66689579158317	9.8836310877504e-07\\
1.66691983967936	9.74462832750851e-07\\
1.66694388777555	9.61271496584705e-07\\
1.66696793587174	9.4873434146708e-07\\
1.66699198396794	9.36802108484541e-07\\
1.66701603206413	9.25431070008814e-07\\
1.66704008016032	9.14581420508266e-07\\
1.66706412825651	9.04217367436843e-07\\
1.66708817635271	8.94306504452512e-07\\
1.6671122244489	8.84819411673076e-07\\
1.66713627254509	8.75729312224614e-07\\
1.66716032064128	8.67011775955209e-07\\
1.66718436873747	8.58644631419252e-07\\
1.66720841683367	8.50607407137285e-07\\
1.66723246492986	8.42880958806277e-07\\
1.66725651302605	8.35448327436077e-07\\
1.66728056112224	8.28293564823989e-07\\
1.66730460921844	8.2140176839332e-07\\
1.66732865731463	8.14759310583562e-07\\
1.66735270541082	8.08353739565969e-07\\
1.66737675350701	8.02173338570534e-07\\
1.66740080160321	7.96207041366091e-07\\
1.6674248496994	7.90444888232425e-07\\
1.66744889779559	7.84877255367129e-07\\
1.66747294589178	7.79495504860624e-07\\
1.66749699398798	7.74291400742766e-07\\
1.66752104208417	7.69257059338009e-07\\
1.66754509018036	7.64385624483896e-07\\
1.66756913827655	7.59669790257894e-07\\
1.66759318637275	7.55103750582853e-07\\
1.66761723446894	7.50681178767954e-07\\
1.66764128256513	7.46396648507001e-07\\
1.66766533066132	7.42244698738794e-07\\
1.66768937875752	7.38220540481144e-07\\
1.66771342685371	7.34319486211289e-07\\
1.6677374749499	7.30537113576152e-07\\
1.66776152304609	7.2686924726523e-07\\
1.66778557114228	7.23311757524494e-07\\
1.66780961923848	7.19861097467223e-07\\
1.66783366733467	7.16513737266397e-07\\
1.66785771543086	7.13266335439531e-07\\
1.66788176352705	7.1011572699128e-07\\
1.66790581162325	7.07058725675219e-07\\
1.66792985971944	7.04092672387843e-07\\
1.66795390781563	7.01214868118043e-07\\
1.66797795591182	6.9842275140582e-07\\
1.66800200400802	6.95713704036508e-07\\
1.66802605210421	6.93085787596615e-07\\
1.6680501002004	6.90536435207923e-07\\
1.66807414829659	6.88063746738552e-07\\
1.66809819639279	6.85665738810504e-07\\
1.66812224448898	6.83340525636188e-07\\
1.66814629258517	6.81086313912356e-07\\
1.66817034068136	6.789012107108e-07\\
1.66819438877756	6.76783780477182e-07\\
1.66821843687375	6.74732665231775e-07\\
1.66824248496994	6.72746024048698e-07\\
1.66826653306613	6.70822646581179e-07\\
1.66829058116232	6.68961204367611e-07\\
1.66831462925852	6.67160247296491e-07\\
1.66833867735471	6.65418947555579e-07\\
1.6683627254509	6.63735790924939e-07\\
1.66838677354709	6.6210987956462e-07\\
1.66841082164329	6.60540368980206e-07\\
1.66843486973948	6.59025910394071e-07\\
1.66845891783567	6.57565762872479e-07\\
1.66848296593186	6.56158862440674e-07\\
1.66850701402806	6.54804747996963e-07\\
1.66853106212425	6.53502447669611e-07\\
1.66855511022044	6.52251032800016e-07\\
1.66857915831663	6.51050172240689e-07\\
1.66860320641283	6.49898835042695e-07\\
1.66862725450902	6.48796585173897e-07\\
1.66865130260521	6.4774283930927e-07\\
1.6686753507014	6.46737051331884e-07\\
1.6686993987976	6.45778711336404e-07\\
1.66872344689379	6.44867161204962e-07\\
1.66874749498998	6.4400232713651e-07\\
1.66877154308617	6.43183619573554e-07\\
1.66879559118236	6.42410481151133e-07\\
1.66881963927856	6.41682754203056e-07\\
1.66884368737475	6.41000128690958e-07\\
1.66886773547094	6.40362325864552e-07\\
1.66889178356713	6.39769097817488e-07\\
1.66891583166333	6.39220227103363e-07\\
1.66893987975952	6.38715526410575e-07\\
1.66896392785571	6.38254838294971e-07\\
1.6689879759519	6.37837854477021e-07\\
1.6690120240481	6.37464836774541e-07\\
1.66903607214429	6.37135356248708e-07\\
1.66906012024048	6.36849553183614e-07\\
1.66908416833667	6.36607416741733e-07\\
1.66910821643287	6.36408965520592e-07\\
1.66913226452906	6.36254247673254e-07\\
1.66915631262525	6.36143341085588e-07\\
1.66918036072144	6.36076353611157e-07\\
1.66920440881764	6.36053247032172e-07\\
1.66922845691383	6.36074542071811e-07\\
1.66925250501002	6.3614008688848e-07\\
1.66927655310621	6.36250288377493e-07\\
1.6693006012024	6.36405409146873e-07\\
1.6693246492986	6.36605744455605e-07\\
1.66934869739479	6.36851622889253e-07\\
1.66937274549098	6.37143407108352e-07\\
1.66939679358717	6.37481494672746e-07\\
1.66942084168337	6.37866147782305e-07\\
1.66944488977956	6.38298178311627e-07\\
1.66946893787575	6.38777752969655e-07\\
1.66949298597194	6.39305589224476e-07\\
1.66951703406814	6.39882273668872e-07\\
1.66954108216433	6.40508434545769e-07\\
1.66956513026052	6.41184743219916e-07\\
1.66958917835671	6.41911915765726e-07\\
1.66961322645291	6.42690548422985e-07\\
1.6696372745491	6.43521783934282e-07\\
1.66966132264529	6.44406151755347e-07\\
1.66968537074148	6.45344730719827e-07\\
1.66970941883768	6.46338488070038e-07\\
1.66973346693387	6.47388448174143e-07\\
1.66975751503006	6.48495695173086e-07\\
1.66978156312625	6.496612141426e-07\\
1.66980561122244	6.50886540430469e-07\\
1.66982965931864	6.52172632963695e-07\\
1.66985370741483	6.53521006460493e-07\\
1.66987775551102	6.54933090268282e-07\\
1.66990180360721	6.56410236702444e-07\\
1.66992585170341	6.57954355171604e-07\\
1.6699498997996	6.59566816103893e-07\\
1.66997394789579	6.61249709834252e-07\\
1.66999799599198	6.63004603867737e-07\\
1.67002204408818	6.64833484953771e-07\\
1.67004609218437	6.66738757836114e-07\\
1.67007014028056	6.6872233301688e-07\\
1.67009418837675	6.70786707307435e-07\\
1.67011823647295	6.72934359253161e-07\\
1.67014228456914	6.75167911790445e-07\\
1.67016633266533	6.77489993689021e-07\\
1.67019038076152	6.79903842595077e-07\\
1.67021442885772	6.82412280144171e-07\\
1.67023847695391	6.85018613809952e-07\\
1.6702625250501	6.87726641950734e-07\\
1.67028657314629	6.90539797991073e-07\\
1.67031062124248	6.93462183476115e-07\\
1.67033466933868	6.96497863489592e-07\\
1.67035871743487	6.99651458527936e-07\\
1.67038276553106	7.02927736698671e-07\\
1.67040681362725	7.06331781772774e-07\\
1.67043086172345	7.09868881352867e-07\\
1.67045490981964	7.13544969447294e-07\\
1.67047895791583	7.17366107379915e-07\\
1.67050300601202	7.21339065544856e-07\\
1.67052705410822	7.25470953093013e-07\\
1.67055110220441	7.29769267041044e-07\\
1.6705751503006	7.34242337802497e-07\\
1.67059919839679	7.38898854344441e-07\\
1.67062324649299	7.43748321661835e-07\\
1.67064729458918	7.48801248401622e-07\\
1.67067134268537	7.54068333235315e-07\\
1.67069539078156	7.59561824877396e-07\\
1.67071943887776	7.65294603976899e-07\\
1.67074348697395	7.7128092147171e-07\\
1.67076753507014	7.77536030668436e-07\\
1.67079158316633	7.84076700525819e-07\\
1.67081563126253	7.90921256410716e-07\\
1.67083967935872	7.98089884863095e-07\\
1.67086372745491	8.05604385547693e-07\\
1.6708877755511	8.13489237756027e-07\\
1.67091182364729	8.21770969607438e-07\\
1.67093587174349	8.30479411393297e-07\\
1.67095991983968	8.39647195563935e-07\\
1.67098396793587	8.49311008893546e-07\\
1.67100801603206	8.59511597969764e-07\\
1.67103206412826	8.70294864252751e-07\\
1.67105611222445	8.81712230918802e-07\\
1.67108016032064	8.93822015788442e-07\\
1.67110420841683	9.06690075412962e-07\\
1.67112825651303	9.2039182278051e-07\\
1.67115230460922	9.35013550417575e-07\\
1.67117635270541	9.50654583996948e-07\\
1.6712004008016	9.6743024415241e-07\\
1.6712244488978	9.85474759097799e-07\\
1.67124849699399	1.00494603571098e-06\\
1.67127254509018	1.02603101652026e-06\\
1.67129659318637	1.04895291588988e-06\\
1.67132064128257	1.07398062767486e-06\\
1.67134468937876	1.10144183083995e-06\\
1.67136873747495	1.13174059290311e-06\\
1.67139278557114	1.16538130304842e-06\\
1.67141683366733	1.20300382199045e-06\\
1.67144088176353	1.24543373916362e-06\\
1.67146492985972	1.29375804609711e-06\\
1.67148897795591	1.34944181502258e-06\\
1.6715130260521	1.41451666239326e-06\\
1.6715370741483	1.49189635614266e-06\\
1.67156112224449	1.58593986098365e-06\\
1.67158517034068	1.70352561110178e-06\\
1.67160921843687	1.85629628609676e-06\\
1.67163326653307	2.06595868881853e-06\\
1.67165731462926	2.3791916696607e-06\\
1.67168136272545	2.92314161439832e-06\\
1.67170541082164	4.26824694106278e-06\\
1.67172945891784	0.000739381086462169\\
1.67175350701403	0.00196440068950283\\
1.67177755511022	0.00267656299756582\\
1.67180160320641	0.00323451799647463\\
1.67182565130261	0.00370854125858777\\
1.6718496993988	0.00412765717564253\\
1.67187374749499	0.00450721403728159\\
1.67189779559118	0.00485650044852805\\
1.67192184368737	0.00518164545204045\\
1.67194589178357	0.00548694637359346\\
1.67196993987976	0.00577555447881121\\
1.67199398797595	0.00604986257613509\\
1.67201803607214	0.00631173858229931\\
1.67204208416834	0.00656267390765469\\
1.67206613226453	0.00680388209293281\\
1.67209018036072	0.00703636624676205\\
1.67211422845691	0.00726096703533264\\
1.67213827655311	0.00747839730592889\\
1.6721623246493	0.00768926779833324\\
1.67218637274549	0.00789410658040051\\
1.67221042084168	0.00809337390178009\\
1.67223446893788	0.00828747377919973\\
1.67225851703407	0.00847676336840778\\
1.67228256513026	0.00866156009444453\\
1.67230661322645	0.00884214758698717\\
1.67233066132265	0.00901878071793141\\
1.67235470941884	0.0091916892353643\\
1.67237875751503	0.00936108144591508\\
1.67240280561122	0.00952714661609882\\
1.67242685370741	0.00969005755622677\\
1.67245090180361	0.00984997247901976\\
1.6724749498998	0.010007036588448\\
1.67249899799599	0.0101613837581914\\
1.67252304609218	0.010313137500957\\
1.67254709418838	0.0104624121624249\\
1.67257114228457	0.0106093139348999\\
1.67259519038076	0.0107539414997899\\
1.67261923847695	0.0108963867998478\\
1.67264328657315	0.0110367357572034\\
1.67266733466934	0.0111750687404704\\
1.67269138276553	0.0113114610158946\\
1.67271543086172	0.0114459832538753\\
1.67273947895792	0.011578701938003\\
1.67276352705411	0.0117096796624321\\
1.6727875751503	0.0118389754221839\\
1.67281162324649	0.0119666449247114\\
1.67283567134269	0.0120927408670926\\
1.67285971943888	0.0122173131486322\\
1.67288376753507	0.0123404090544093\\
1.67290781563126	0.0124620734393429\\
1.67293186372746	0.0125823489151612\\
1.67295591182365	0.012701276016652\\
1.67297995991984	0.0128188933397968\\
1.67300400801603	0.012935237660311\\
1.67302805611222	0.0130503440447564\\
1.67305210420842	0.0131642459626211\\
1.67307615230461	0.0132769753926966\\
1.6731002004008	0.013388562920275\\
1.67312424849699	0.0134990378215553\\
1.67314829659319	0.0136084281374704\\
1.67317234468938	0.0137167607423632\\
1.67319639278557	0.0138240614090001\\
1.67322044088176	0.0139303548729275\\
1.67324448897796	0.0140356648921295\\
1.67326853707415	0.0141400143032969\\
1.67329258517034	0.0142434250721453\\
1.67331663326653	0.0143459183385074\\
1.67334068136273	0.0144475144582163\\
1.67336472945892	0.0145482330420908\\
1.67338877755511	0.0146480929930261\\
1.6734128256513	0.0147471125416258\\
1.67343687374749	0.0148453092802365\\
1.67346092184369	0.0149427001950528\\
1.67348496993988	0.015039301695495\\
1.67350901803607	0.0151351296420232\\
1.67353306613226	0.0152301993714474\\
1.67355711422846	0.0153245257201621\\
1.67358116232465	0.0154181230471383\\
1.67360521042084	0.0155110052540019\\
1.67362925851703	0.0156031858056093\\
1.67365330661323	0.0156946777490371\\
1.67367735470942	0.0157854937319713\\
1.67370140280561	0.0158756460194835\\
1.6737254509018	0.015965146511316\\
1.673749498998	0.0160540067563844\\
1.67377354709419	0.0161422379677579\\
1.67379759519038	0.0162298510361005\\
1.67382164328657	0.0163168565424459\\
1.67384569138277	0.0164032647705574\\
1.67386973947896	0.0164890857179948\\
1.67389378757515	0.0165743291075321\\
1.67391783567134	0.0166590043971585\\
1.67394188376754	0.0167431207905657\\
1.67396593186373	0.0168266872462082\\
1.67398997995992	0.0169097124870043\\
1.67401402805611	0.0169922050083998\\
1.6740380761523	0.0170741730875054\\
1.6740621242485	0.0171556247900598\\
1.67408617234469	0.0172365679788999\\
1.67411022044088	0.0173170103205816\\
1.67413426853707	0.0173969592922982\\
1.67415831663327	0.0174764221885368\\
1.67418236472946	0.0175554061271213\\
1.67420641282565	0.0176339180550483\\
1.67423046092184	0.0177119647545628\\
1.67425450901804	0.0177895528483301\\
1.67427855711423	0.0178666888042448\\
1.67430260521042	0.017943378941198\\
1.67432665330661	0.0180196294333612\\
1.67435070140281	0.0180954463149762\\
1.674374749499	0.0181708354847655\\
1.67439879759519	0.0182458027101331\\
1.67442284569138	0.018320353631457\\
1.67444689378758	0.0183944937656435\\
1.67447094188377	0.0184682285102382\\
1.67449498997996	0.0185415631469703\\
1.67451903807615	0.0186145028453109\\
1.67454308617234	0.0186870526654941\\
1.67456713426854	0.018759217562098\\
1.67459118236473	0.0188310023868855\\
1.67461523046092	0.018902411891883\\
1.67463927855711	0.0189734507322755\\
1.67466332665331	0.01904412346909\\
1.6746873747495	0.0191144345716936\\
1.67471142284569	0.0191843884207235\\
1.67473547094188	0.01925398931007\\
1.67475951903808	0.0193232414497639\\
1.67478356713427	0.0193921489677541\\
1.67480761523046	0.0194607159125332\\
1.67483166332665	0.0195289462551696\\
1.67485571142285	0.0195968438911912\\
1.67487975951904	0.0196644126427373\\
1.67490380761523	0.0197316562605389\\
1.67492785571142	0.0197985784256245\\
1.67495190380762	0.0198651827513282\\
1.67497595190381	0.0199314727847615\\
1.675	0.0199974520088734\\
};
\addlegendentry{$\rho^\ast$}

\addplot[only marks, mark=*, mark options={}, mark size=1.5000pt, color=black, fill=black] table[row sep=crcr]{%
x	y\\
1.66895	0\\
};
\addlegendentry{$\xi_{t_\ast}(\cu^\ast)$}

\addplot[only marks, mark=square, mark options={}, mark size=1.5000pt, draw=black] table[row sep=crcr]{%
x	y\\
1.6662	0\\
};
\addlegendentry{$\ed^\ast_\pm=\ed^\pm_0$}

\addplot[only marks, mark=o, mark options={}, mark size=1.5000pt, draw=black] table[row sep=crcr]{%
x	y\\
1.66785	0\\
};
\addlegendentry{$\xi_s(\ed_s^\pm)$}

\addplot[only marks, mark=square, mark options={}, mark size=1.5000pt, draw=black] table[row sep=crcr]{%
x	y\\
1.6717	0\\
};
\addplot[only marks, mark=o, mark options={}, mark size=1.5000pt, draw=black] table[row sep=crcr]{%
x	y\\
1.67005	0\\
};
\addplot [color=black, mark=|, mark options={solid, rotate=180, black}, forget plot]
  table[row sep=crcr]{%
1.6662	0.0015\\
1.66785	0.0015\\
};
\addplot [color=black, mark=|, mark options={solid, rotate=180, black}, forget plot]
  table[row sep=crcr]{%
1.67005	0.0015\\
1.6717	0.0015\\
};
\addplot [color=black, mark=|, mark options={solid, rotate=180, black}, forget plot]
  table[row sep=crcr]{%
1.6662	0.008\\
1.6717	0.008\\
};
\node[align=center]
at (axis cs:1.667,0.005) {$D_s^-$};
\node[align=center]
at (axis cs:1.671,0.005) {$D_s^+$};
\node[align=center]
at (axis cs:1.669,0.012) {$\Delta_0=\Delta^\ast$};
\end{axis}
\end{tikzpicture}%

%% file: gammas.tex
\usetikzlibrary{decorations,decorations.markings}
\tikzset{forward arrow/.style={
 postaction={
  decorate,
  decoration={
   markings,
   mark=at position .4 with {\arrow{latex}},%
   mark=at position .623 with {\arrow{latex}}
  }
 }
}}
\tikzset{forward arrow2/.style={
 postaction={
  decorate,
  decoration={
   markings,
   mark=at position .32 with {\arrow{latex}},%
   mark=at position .723 with {\arrow{latex}}
  }
 }
}}
\tikzset{backward arrow/.style={
 postaction={
  decorate,
  decoration={
   markings,
   mark=at position .4 with {\arrow{latex reversed}},%
   mark=at position .623 with {\arrow{latex reversed}}
  }
 }
}}
\begin{tikzpicture}

\begin{axis}[%
width=0.951\figurewidth,
height=\figureheight,
at={(0\figurewidth,0\figureheight)},
scale only axis,
xmin=-1,
xmax=1,
xtick={\empty},
ymin=-1,
ymax=1,
ytick={\empty},
axis background/.style={fill=white},
legend style={legend cell align=left, align=left, draw=white!15!black}
]
\addplot [color=black,forward arrow]
  table[row sep=crcr]{%
-1	-1\\
-0.9	-0.9\\
-0.8	-0.8\\
-0.7	-0.7\\
-0.6	-0.6\\
-0.5	-0.5\\
-0.4	-0.4\\
-0.3	-0.3\\
-0.2	-0.2\\
-0.1	-0.1\\
0	0\\
0.1	0.1\\
0.2	0.2\\
0.3	0.3\\
0.4	0.4\\
0.5	0.5\\
0.6	0.6\\
0.7	0.7\\
0.8	0.8\\
0.9	0.9\\
1	1\\
};
\addlegendentry{$\widehat\Upsilon'$}

\addplot [color=black,backward arrow,forget plot]
  table[row sep=crcr]{%
-1	1\\
-0.9	0.9\\
-0.8	0.8\\
-0.7	0.7\\
-0.6	0.6\\
-0.5	0.5\\
-0.4	0.4\\
-0.3	0.3\\
-0.2	0.2\\
-0.1	0.1\\
0	-0\\
0.1	-0.1\\
0.2	-0.2\\
0.3	-0.3\\
0.4	-0.4\\
0.5	-0.5\\
0.6	-0.6\\
0.7	-0.7\\
0.8	-0.8\\
0.9	-0.9\\
1	-1\\
};
\addplot [color=black, dashed,forward arrow2]
  table[row sep=crcr]{%
0	-1\\
0	-0.9\\
0	-0.8\\
0	-0.7\\
0	-0.6\\
0	-0.5\\
0	-0.4\\
0	-0.3\\
0	-0.2\\
0	-0.1\\
0	0\\
0	0.1\\
0	0.2\\
0	0.3\\
0	0.4\\
0	0.5\\
0	0.6\\
0	0.7\\
0	0.8\\
0	0.9\\
0	1\\
};
\addlegendentry{$\widehat\Gamma'$}

\addplot [color=black, draw=none, only marks, mark options={solid, black}]
  table[row sep=crcr]{%
0	0\\
};
\addlegendentry{$0$}

\end{axis}
\end{tikzpicture}%